%% file: Thesis.tex
\documentclass[12pt]{book}

\usepackage[T1]{fontenc}
\usepackage[utf8]{inputenc}
\usepackage{lmodern}

\usepackage[english]{babel}
\usepackage[autostyle, english=american]{csquotes}
\MakeOuterQuote{"}

\usepackage[numbers,square]{natbib}

\usepackage[margin=1in,includehead,headheight=15pt]{geometry}
\usepackage{setspace}
\usepackage{graphicx}
\usepackage{adjustbox}
\usepackage{wrapfig}
\usepackage{rotating}
\usepackage{pdflscape} 
\usepackage{float}
\usepackage{subcaption} 

\usepackage{array}
\usepackage{booktabs}
\usepackage{ragged2e}
\usepackage{tabularx}
\usepackage{hhline}
\usepackage{multirow}

\usepackage[dvipsnames]{xcolor}
\usepackage{enumitem}

\usepackage{amsmath,amsthm,amssymb,amsfonts,mathtools,amscd}
\usepackage{bbm}

\usepackage{tikz-cd}
\usetikzlibrary{arrows.meta,calc,positioning}
\tikzset{
  symbol/.style={
    draw=none,
    every to/.append style={
      edge node={node [sloped, allow upside down, auto=false]{$#1$}}
    }
  }
}
\usepackage[all,cmtip]{xy}

\usepackage{emptypage}

\usepackage{acronym}
\usepackage{glossaries}

\usepackage{titlesec}
\titleformat{\chapter}{\centering\Large\bfseries}{\thechapter}{1em}{}
\titleformat{\section}{\normalsize\bfseries}{\thesection}{0.8em}{}
\titleformat{\subsection}{\normalsize\bfseries}{\thesubsection}{0.6em}{}

\usepackage{fancyhdr}
\fancypagestyle{custompagestyle}{%
  \fancyhf{}
  \fancyhead[R]{\thepage}
}
\fancypagestyle{noheaderfooter}{%
  \fancyhf{}
  
}

\usepackage[colorlinks=true,linkcolor=blue,urlcolor=blue,citecolor=blue,anchorcolor=blue,unicode]{hyperref}

\DeclareRobustCommand{\Linfty}{\texorpdfstring{$L_\infty$}{L-infinity}}
\DeclareRobustCommand{\Twoalgebras}{\texorpdfstring{$2$-Algebras }{2-algebras}}
\DeclareRobustCommand{\Twoalgebra}{\texorpdfstring{$2$-Algebra }{2-algebra}}
\DeclareRobustCommand{\Gspace}{\texorpdfstring{$G$-space }{G-space}}
\DeclareRobustCommand{\nplectic}{\texorpdfstring{$n$-plectic }{n-plectic}}

\newcommand{\bra}{[\hspace{-0.17em}[}
\newcommand{\ket}{]\hspace{-0.17em}]}

\theoremstyle{definition}
\newtheorem{lem}[subsection]{Lemma}
\newtheorem{cor}[subsection]{Corollary}

\newtheorem{pro}[subsection]{Proposition}
\newtheorem{exa}[subsection]{Example}
\newtheorem{defn}[subsection]{Definition}
\newtheorem{rem}[subsection]{Remark}

\newtheorem{thm}[subsection]{Theorem}

{}
{}
{}
{}

\newcommand{\Ad}{\mathrm{Ad}}

\makeatletter
\newtheorem*{rep@theorem}{\rep@title}
\newcommand{\newreptheorem}[2]{%
  \newenvironment{rep#1}[1]{%
    \def\rep@title{#2 ##1}%
    \begin{rep@theorem}}{\end{rep@theorem}}}
\makeatother
\newreptheorem{thm}{Theorem}
\newreptheorem{pro}{Proposition}
\newreptheorem{lem}{Lemma}
\newreptheorem{defn}{Definition}

\let\origdoublepage\cleardoublepage
\newcommand{\clearemptydoublepage}{%
  \clearpage{\pagestyle{empty}\origdoublepage}}
\let\cleardoublepage\clearemptydoublepage

\usepackage{bbm}
\usepackage{enumitem}

\usepackage{amsmath}

\usepackage{layout}
\usepackage{fancyhdr}
\fancypagestyle{custompagestyle}{%
    \fancyhf{}
    \fancyhead[R]{\thepage}
}

\fancypagestyle{noheaderfooter}{%
    \fancyhf{} 
}

\titleformat{\chapter}[display]
  {\centering\Large\bfseries}
  {\chaptername~\thechapter}
  {1.4ex}         
  {\Large}
  [\vspace{2ex}]  

\titlespacing*{\chapter}{0pt}{3ex plus 1ex minus .2ex}{4ex plus .2ex}

\usepackage{epstopdf}

\begin{document}
\pagenumbering{roman}

\input{titlepage}


   \chapter*{Abstract}

A manifold is said to be \( n \)-plectic if it is equipped with a closed, nondegenerate \( (n+1) \)-form. This thesis develops the theory of \emph{relative \( n \)-plectic structures}, where the classical condition is replaced by a closed, nondegenerate \emph{relative} \( (n+1) \)-form defined with respect to a smooth map. Analogous to how \( n \)-plectic manifolds give rise to \( L_\infty \)-algebras of observables, we show that relative \( n \)-plectic structures naturally induce corresponding \( L_\infty \)-algebras. These structures provide a conceptual bridge between the frameworks of quasi-Hamiltonian \( G \)-spaces and \( 2 \)-plectic geometry.

As an application, we examine the relative \( 2 \)-plectic structure canonically associated to quasi-Hamiltonian \( G \)-spaces. We show that every quasi-Hamiltonian \( G \)-space defines a closed, nondegenerate relative \( 3 \)-form, and that the group action induces a Hamiltonian infinitesimal action compatible with this structure. We then construct explicit homotopy moment maps as \( L_\infty \)-morphisms from the Lie algebra \( \mathfrak{g} \) into the Lie \( 2 \)-algebra of relative observables, extending the moment map formalism to the higher and relative geometric setting.

 \addcontentsline{toc}{chapter}{Abstract}

\chapter*{Acknowledgements}

I would like to begin by expressing my deepest gratitude to my supervisor, Dr. Derek Krepski. His unwavering support, guidance, and insightful feedback have been instrumental throughout this Ph.D. journey. Coming from a background in algebraic geometry, I am especially thankful to him for introducing me to the area of multisymplectic geometry and for his constant encouragement and mentorship.

I am also thankful to the members of my committee, Dr. Adam Clay and Dr. Eric Schippers, for their valuable suggestions and feedback, which greatly improved the quality of this thesis. I sincerely thank Dr. Lisa Jeffrey, my external examiner, for taking the time to review this work and provide thoughtful comments.

With heartfelt remembrance, I acknowledge my late Master's supervisor and mentor, Prof. Marco Garuti. His support during my studies at the Università di Padova and AIMS-Cameroon left a lasting impact on my academic path.

I am deeply grateful to my wife, Delphine Poloksi, whose love, patience, and encouragement have been my greatest source of strength throughout this journey.

I would also like to thank Dr. Shaun Lui and Dr. Nicholas Harland for giving me the opportunity to teach as a sessional instructor—an experience that has shaped my aspirations for a teaching career. I extend my thanks to the entire staff and support team in the Department of Mathematics, as well as to my fellow graduate students for their friendship and support.

Finally, I gratefully acknowledge the financial support I received, including the Faculty of Graduate Studies Scholarships, Jiri Sichler Memorial Scholarship in Algebra, and funding through GETS and SEGS. This support allowed me to focus fully on my research and teaching responsibilities.


\tableofcontents{}



\thispagestyle{noheaderfooter}

\pagenumbering{arabic}

\include{Chapter1}

\include{Chapter2}
\include{Chapter3}

\include{Chapter4}

\include{Chapter5}

 \bibliographystyle{plainnat}
 \bibliography{References}

\appendix

\end{document}

%% file: titlepage.tex
\begin{titlepage}
    \begin{center}
        \vspace*{1cm}

        \large\uppercase{
        {Observables of Relative Structures and Lie 2-algebras associated with Quasi-Hamiltonian $G$-spaces}}

        \vspace{1.5cm}

        by

        \vspace{1.5cm}

        \Large{\uppercase{Dinamo Djounvouna}}

        \vspace{1.5cm}

        \normalsize A thesis submitted to the \\
        Department of Mathematics \\
        in conformity with the requirements for \\
        the degree of Doctor of Philosophy in Mathematics

        \vspace{1.5cm}




        \small Copyright \textcopyright\ Dinamo Djounvouna, 2025
    \end{center}
\end{titlepage}

%% file: Chapter1.tex

 \chapter{Introduction and summary} \label{ch1}

\section{Introduction}

A multisymplectic manifold \cite{Cantrijn1998, Ryvkin_2019}, also known as an \(n\)-plectic manifold \cite{rogers2011higher}, is a smooth manifold $M$ equipped with closed and nondegenerate \((n+1)\)-form $\omega$. The form \(\omega\) is called the \(n\)-plectic structure. When \(\omega\) is only closed, then \((M, \omega)\) is called pre-\(n\)-plectic manifold, and the form \(\omega\) is the pre-\(n\)-plectic structure. A symplectic manifold is a special case of \(n\)-plectic manifold, being a smooth manifold endowed with a closed and nondegenerate $2$-form. So, multisymplectic manifolds extend symplectic manifolds by generalizing closed, nondegenerate $2$-forms with closed, nondegenerate \((n+1)\)-forms for \(n > 1\). However, this extension should not be interpreted to mean that every property found in 
  symplectic geometry is preserved in multisymplectic geometry. Instead, for \(n=1\), we recover the symplectic case.  
   The primary differences between symplectic and multisymplectic manifolds are as follows. First, symplectic manifolds are necessarily even-dimensional, as the nondegeneracy of a symplectic 2-form \(\omega\) requires an even number of linearly independent tangent vectors at each point. In contrast, multisymplectic manifolds, which are equipped with a closed nondegenerate differential form of degree \(k > 2\), can exist on manifolds of arbitrary dimension. 

Second, the well-known Darboux Theorem in symplectic geometry, which guarantees a local canonical form for the symplectic structure, does not hold in the multisymplectic case. Darboux's Theorem is stated as follows:

\begin{repthm}{\ref{thm:darboux}'}[Darboux's Theorem]\label{thm:darboux1}
Let \((M, \omega)\) be a symplectic manifold, and let \(p\) be any point in \(M\). Then there exists a local chart \((\mathcal{U}, x_1, \ldots, x_n, y_1, \ldots, y_n)\) centered at \(p\) such that:
\[
\omega = \sum_{i=1}^n dx_i \wedge dy_i.
\]
\end{repthm}

This local representation, often referred to as \emph{Darboux coordinates}, implies that an \(n\)-dimensional symplectic manifold is locally diffeomorphic to \((\mathbb{R}^{2n}, \omega_0)\) equipped with its standard symplectic form. Consequently, \((\mathbb{R}^{2n}, \omega_0)\) serves as the canonical model for any \(2n\)-dimensional symplectic manifold.

This result extends to \emph{volume forms}---that is, closed nondegenerate top-degree forms on a manifold---where a similar local trivialization exists (see \cite[Theorem 3.6]{ryvkin2025darbouxtypetheoremsmultisymplectic} for instance). However, unlike the symplectic case, \emph{Darboux's theorem fails in general for multisymplectic geometry}. For instance, consider \((\mathbb{R}^6, \omega)\), where the 3-form 
\[
\omega = dx^1 \wedge dx^3 \wedge dx^5 - dx^1 \wedge dx^4 \wedge dx^6 - dx^2 \wedge dx^3 \wedge dx^6 + x_2 \cdot dx^2 \wedge dx^4 \wedge dx^5
\]
can be shown to be closed and nondegenerate. This structure defines a 2-plectic manifold \((\mathbb{R}^6, \omega)\). However, as shown in \cite[Proposition 2.3, Corollary 2.4]{ryvkin2016multisymplectic}, the 2-plectic manifold \((\mathbb{R}^6, \omega)\) does not admit Darboux coordinates near the origin \(0\).

Symplectic geometry and multisymplectic geometry still share several fundamental features.
While the theory of quantization in \emph{multisymplectic geometry} is still in its formative stages and lacks the depth and maturity found in symplectic geometry \cite{wernli2023lecturesgeometricquantization, batesweinstein_quantization, IOOS2021107840}, there remain important conceptual bridges between the two frameworks. Key ideas such as \emph{reduction}, \emph{prequantization}, and the use of \emph{(homotopy) moment maps} appear in both settings, though often in more higher-categorical forms in the multisymplectic case. Notably, reduction in the multisymplectic context has been explored in \cite{blacker2023reduction, Blacker2018PolysymplecticRA, de_Le_n_2024, ECHEVERRIAENRIQUEZ2018415}, while early steps toward a prequantization framework can be found in \cite{Sevestre2020OnTP, Derek2008}, to mention a few of the references.

Just as a symplectic manifold gives rise to a Poisson algebra of observables, a multisymplectic manifold naturally gives rise to higher structures generalizing Poisson brackets—namely, \( L_\infty \)-algebras. In the symplectic case, the Lie algebra of observables is given by the space \( C^\infty(M) \) of smooth functions on a symplectic manifold \( (M, \omega) \), equipped with the Poisson bracket. The symplectic form \( \omega \in \Omega^2(M) \) is a closed, non-degenerate 2-form, which allows one to associate to each smooth function \( f \in C^\infty(M) \) a unique Hamiltonian vector field \( X_f \) defined by the equation \( \iota_{X_f} \omega = df \). The Poisson bracket of two functions \( f, g \in C^\infty(M) \) is then defined as \( \{f, g\} := \omega(X_f, X_g) \). This bracket is bilinear, antisymmetric, satisfies the Jacobi identity, and obeys the Leibniz rule, thereby endowing \( C^\infty(M) \) with the structure of a Lie algebra. Thus, the Lie algebra of observables in symplectic geometry is \( (C^\infty(M), \{\cdot, \cdot\}) \), where the bracket encodes the infinitesimal dynamics of the system.

Similarly, multisymplectic geometry features analogous structures known as \(L_\infty\)-algebras, as shown in \cite[Theorem 3.14]{rogers2011higher} and \cite[Theorem 4.6]{callies2016homotopy}. See section \ref{section: L_infinity algebra} for a brief background on \(L_\infty\)-algebras.

\begin{repthm}{\ref{def:RogersAlgebra1}'}
	Given an $n$-plectic manifold $(M,\omega)$, there exists an $L_\infty$-algebra $L_{\infty}(M,\omega)=(L,\{l_k\}_{k\geq 1})$ with 
	\begin{itemize}
		\item 
		the underlying graded vector space $L$, where
	\begin{equation*}
		L^i=\begin{cases}
			\Omega_{\mathrm{Ham}}^{n-1}(M,\omega) 
			& \quad~\text{if } i=0
			\\
			\Omega^{n-1+i}(M) 
		 	&  \quad~\text{if } 1-n \leq i\leq -1
		 	\\			
			0 & \quad ~\text{otherwise,}
		\end{cases}
	\end{equation*}

	\item 
		$n+1$ nontrivial multibrackets $\lbrace l_k : L^{\wedge k} \to L\rbrace_{1\leq k\leq n+1}$, given by
	\begin{displaymath}
		l_1(\alpha) = 
		\begin{cases}
			0 & \quad\text{if~} \deg(\alpha) = 0	
			\\
			\mathrm{d} \alpha & \quad \text{if~} \deg(\alpha) \leq -1,
		\end{cases}
	\end{displaymath}
	and, for $ 2 \leq k \leq n+1$, as
	\begin{displaymath}
		l_k(\alpha_1,\dots,\alpha_k) = 
		\begin{cases}
			\epsilon(k) \iota_{ v_{\alpha_k}}\ldots\iota_{ v_{\alpha_1}}\omega
			& \quad\text{if~} \deg(\alpha_i)=0 \text{ for } 1\leq i \leq k
			\\
			0 & \quad\text{otherwise,}		
		\end{cases}
	\end{displaymath}
	\end{itemize}

\noindent		where $ v_{\alpha_k}$ denotes any Hamiltonian vector field corresponding to $\alpha_k\in \Omega^{n-1}_{\mathrm{Ham}}(M,\omega)$ and $\epsilon(k) = - (-1)^{\frac{k(k+1)}{2}}$ is the total Koszul sign.         
  
\end{repthm}

In this thesis, we will be interested in the construction of an \(L_\infty\)-algebra arising from structures defined \emph{relative} to a smooth map. Let \( F: M \to N \) be a smooth map between manifolds. The \emph{algebraic mapping cone} of the pullback map \( F^* : \Omega^\bullet(N) \to \Omega^\bullet(M) \) is defined by
\[
\Omega^\bullet(F) := \mathrm{Cone}^\bullet(F^*) = \Omega^{\bullet-1}(M) \oplus \Omega^\bullet(N),
\]
with differential
\[
d(\alpha, \beta) := \left(F^* \beta + d\alpha,\ -d\beta\right).
\]
This differential is called the \emph{relative de Rham differential}, and elements of \( \Omega^n(F) \) are called \emph{relative \(n\)-forms}.

A \emph{relative \((n+1)\)-form} \( (\omega_M, \omega_N) \in \Omega^{n+1}(F) \) is said to be \emph{closed} if \( d(\omega_M, \omega_N) = 0 \), and \emph{nondegenerate} if, for every point \( x \in M \) and every pair \( (u, v) \in T_x M \times T_{F(x)} N \) satisfying \( v = dF_x(u) \), the condition
\[
\iota_{(u, v)}(\omega_M, \omega_N) = \left( \iota_u \omega_M(x),\, \iota_v \omega_N(F(x)) \right) = (0, 0)
\]
implies \( (u, v) = (0, 0) \). When both conditions are satisfied, we refer to \( (F, \omega_M, \omega_N) \) as a \emph{relative \(n\)-plectic structure}.

\medskip

A motivational example of a relative 3-form arises in the theory of \emph{quasi-Hamiltonian \(G\)-spaces}, introduced by Alekseev, Malkin, and Meinrenken in \cite{10.4310/jdg/1214460860}. Unlike classical Hamiltonian spaces, the moment map in this theory takes values in the Lie group \(G\), rather than in the dual of its Lie algebra.

A quasi-Hamiltonian \(G\)-space is a triple \( (M, \omega, \Phi) \), where:
\begin{itemize}
    \item \( M \) is a smooth \(G\)-manifold;
    \item \( \omega \in \Omega^2(M)^G \) is a \(G\)-invariant 2-form;
    \item \( \Phi: M \to G \) is a \(G\)-equivariant map (the group-valued moment map),
\end{itemize}
satisfying the conditions:
\begin{enumerate}
    \item \( d\omega = -\Phi^* \eta \), where \( \eta \) is the Cartan 3-form on \(G\), bi-invariant under both left and right translations;
    \item \( \iota_{v_x} \omega = \frac{1}{2} \Phi^* \left( (\theta^L + \theta^R) \cdot x \right) \) for all \( x \in \mathfrak{g} \);
    \item \( \ker \omega_m \cap \ker d\Phi_m = \{0\} \) for all \( m \in M \).
\end{enumerate}

Here, \( \theta^L \) and \( \theta^R \) are the left and right Maurer–Cartan forms on \(G\).  The relevant definitions of \( \theta^L \), \( \theta^R \), and \( \eta \) are recalled in Section~\ref{sec:cartan_forms}.

While the 2-form \( \omega \) in the definition of a quasi-Hamiltonian is generally neither closed nor nondegenerate, the relative 3-form \( (\omega, \eta) \) is both closed (see Proposition~\ref{closed}) and nondegenerate (see Theorem~\ref{Non-deg}). Therefore, the data \( (M, \Phi, \omega) \) defines a relative 2-plectic structure.  According to Rogers’ work, a 2-plectic manifold gives rise to a Lie 2-algebra of observables (see \cite{rogers2011higher} and \cite{rogers20132plectic}). Just as a 2-plectic manifold gives rise to a Lie 2-algebra, the relative 2-plectic structure arising from a quasi-Hamiltonian \(G\)-space leads to a Lie 2-algebra (see Theorem \ref{theorem:semi strict Lie2}). This analogy motivates our interest in extending the notion of observables and moment maps to the relative setting.

In multisymplectic geometry, it is well established that an \(n\)-plectic manifold gives rise to an \(L_\infty\)-algebra of observables~\cite{rogers2011higher}. Moreover, when a Lie group acts on such a manifold in a Hamiltonian fashion, one can define a \emph{homotopy moment map}~\cite[Theorem 9.6]{callies2016homotopy}.

We aim to extend these constructions to the \emph{relative setting} for the Lie \(2\)-lgebra associated with a quasi-Hamiltonian \(G\)-spce.
The construction of an \(L_\infty\)-algebra in multisymplectic geometry relies on two key ingredients: Hamiltonian \((n-1)\)-forms and their associated Hamiltonian vector fields. To carry out a similar construction in the relative setting, one must first identify appropriate analogues of these geometric objects. This, in turn, requires a careful formulation of standard operations—such as the Lie derivative and the interior product—in the context of a smooth map between manifolds.

Although the relative differential is well-established in the literature—most notably through the construction of the algebraic mapping cone—explicit definitions of the \emph{relative Lie derivative} and \emph{relative interior product} are not commonly formulated.  In this thesis, we include relative analogues of these classical Cartan operations both for completeness and to establish the sign conventions adopted throughout. These operations are defined as \emph{pairs} acting on relative differential forms and are constructed to be fully compatible with the relative de Rham differential, thereby extending the standard Cartan calculus to the relative setting in a coherent way.

\begin{repdefn}{\ref{def:relative interior and Lie}'}
Let \( F: M \to N \) be a smooth map, and let \( (\alpha, \beta) \in \Omega^n(F) := \Omega^{n-1}(M) \oplus \Omega^n(N) \) be a relative \(n\)-form. Let \( (u, v) \in \mathfrak{X}(F) \) be a pair of vector fields such that \( v = dF(u) \). We define:

\begin{enumerate}
    \item The \emph{relative interior product} (or contraction) of \( (\alpha, \beta) \) with \( (u, v) \) is the relative \((n-1)\)-form:
    \[
    \iota_{(u,v)}(\alpha,\beta) := \left(\iota_u \alpha,\; -\iota_v \beta\right).
    \]

    \item The \emph{relative Lie derivative} of \( (\alpha, \beta) \) along \( (u, v) \) is the relative \(n\)-form:
    \[
    \mathcal{L}_{(u,v)}(\alpha,\beta) := \left(\mathcal{L}_u \alpha,\; \mathcal{L}_v \beta\right).
    \]
\end{enumerate}
\end{repdefn}

These adapted operations form the basis for defining \emph{relative Hamiltonian forms} and \emph{relative Hamiltonian vector fields}. To validate these constructions, we establish that the classical relations of Cartan calculus remain valid under our definitions. In particular, we prove that Cartan’s magic formula and its associated identities hold in the relative setting, as shown in the following proposition.

\begin{reppro}{\ref{relative cartan magic formula}'}
Let \( F: M \to N \) be a smooth map between manifolds, and let \( (u, v), (a, b) \in \mathfrak{X}(F) \) be pairs of \(F\)-related vector fields. Then the following identities hold for relative differential forms:
\begin{enumerate}
    \item \textbf{Commutator of Lie derivatives:}
    \[
    \mathcal{L}_{(u,v)} \circ \mathcal{L}_{(a,b)} - \mathcal{L}_{(a,b)} \circ \mathcal{L}_{(u,v)} = \mathcal{L}_{[(u,v), (a,b)]}.
    \]

    \item \textbf{Anti-symmetry of contractions:}
    \[
    \iota_{(u,v)} \circ \iota_{(a,b)} + \iota_{(a,b)} \circ \iota_{(u,v)} = 0.
    \]

    \item \textbf{Lie derivative commutes with the differential:} If \( u \) is \(F\)-related to \( v \), then
    \[
    d \circ \mathcal{L}_{(u,v)} - \mathcal{L}_{(u,v)} \circ d = 0.
    \]

    \item \textbf{Lie derivative and contraction:}
    \[
    \mathcal{L}_{(u,v)} \circ \iota_{(a,b)} - \iota_{(a,b)} \circ \mathcal{L}_{(u,v)} = \iota_{[(u,v), (a,b)]}.
    \]

    \item \textbf{Cartan's magic formula (relative version):} If \( u \) is \(F\)-related to \( v \), then
    \[
    \iota_{(u,v)} \circ d + d \circ \iota_{(u,v)} = \mathcal{L}_{(u,v)}.
    \]
\end{enumerate}
\end{reppro}

These formulas, proven in Section~\ref{section:relative cartan magic formulas}, confirm that the classical Cartan identities hold in the relative framework, thereby supporting the use of these operations in defining relative observables.

With this structure in place, we extend the construction of Rogers~\cite[Theorem 3.14]{rogers2011higher} to the relative setting, demonstrating that closed and nondegenerate relative \((n+1)\)-forms naturally give rise to \(L_\infty\)-algebra structures.

\section{Historical Origins and Motivation for  \Linfty-Algebras}

The concept of \( L_\infty \)-algebras, also known as strongly homotopy Lie algebras \cite{Lada1993}, emerged in the mid-1980s from several independent developments in deformation theory, rational homotopy theory, and mathematical physics. 
A major motivation came from the realization that many deformation problems are governed by differential graded Lie algebras (DGLAs). However, these DGLAs often fail to be uniquely associated to a deformation problem, even up to isomorphism. This led to the introduction of homotopy Lie algebras that generalize Lie algebras by incorporating higher homotopies to satisfy weakened Jacobi identities \cite{Stasheff2016}. 

\begin{quote}
    \emph{"Every deformation problem is controlled by a differential graded Lie algebra—but not uniquely, even up to isomorphism."} — Stasheff
\end{quote}

This guiding principle continues to shape modern developments in mathematical physics and higher geometry \cite{rogers2011higher, Forger2005, FORGER_2013}.

\subsection{Stasheff’s Retrospective on the Origins of \Linfty-Algebras}

As observed by Stasheff \cite{Stasheff2016}, \( L_\infty \)-algebras often appeared in disguise long before their axiomatic definition. One early source was Dennis Sullivan, a leading figure in rational homotopy theory, who was motivated by questions arising from the discretization of partial differential equations governing fluid dynamics. In particular, he sought a discrete version of volume-preserving vector fields that preserved key physical invariants such as energy, helicity, and the ``frozen-in'' property of vorticity. In doing so, he encountered a Lie bracket that satisfied the Jacobi identity only \emph{up to homotopy}—a phenomenon that could not be accounted for by standard Lie algebra theory. 

A turning point occurred when Sullivan spoke with Schechtman, who introduced him to the idea that such ``Lie algebras up to homotopy'' could be encoded as derivations on free graded commutative algebras endowed with differential structures. This insight provided a unifying algebraic language that clarified the previously mysterious behavior observed in Sullivan’s models. What had appeared to be anomalous or pathological was, in fact, an instance of a richer structure: an \( L_\infty \)-algebra.

Simultaneously, in the Soviet school, Vladimir Drinfel'd had introduced similar concepts in private correspondence with Schechtman. In letters dating from 1983 and 1988 \cite{Drinfeld2014Letter, Stasheff2016}, Drinfel'd described what he called \emph{Lie-Sugawara} algebras: differentials of square zero on the cofree cocommutative coalgebra generated by a graded vector space—precisely the modern coalgebraic formulation of an \( L_\infty \)-algebra. Despite the clarity of his constructions, Drinfel'd saw himself more as elucidating known ideas than inventing new ones, even speculating that his work might already be implicit in Quillen's foundational paper on rational homotopy theory.

Concurrently, in the work of Mike Schlessinger and Jim Stasheff on the Lie algebraic structure of tangent cohomology and its role in deformation theory \cite{Schlessinger1985TheLA}, the need to generalize beyond strict differential graded Lie algebras (DGLAs) was becoming increasingly apparent. Their investigations naturally led to the consideration of morphisms defined only up to coherent homotopies—ideas that pointed directly toward the formalism of \( L_\infty \)-algebras \cite{Stasheff1988, Stasheff1992, Stasheff1997, Stasheff1998}.

The convergence of these threads became especially clear during discussions at the Institut des Hautes \'{E}tudes Scientifiques (IHES) in the early 1990s, where Sullivan, Schechtman, Ginzburg, Kapranov, and others recognized that familiar objects from rational homotopy theory—such as Sullivan’s models—were governed by \( L_\infty \)-structures in disguise. What had previously been viewed as ad hoc constructions or technical artifacts were reinterpreted as special cases of a general and flexible homotopical algebra.

In this way, Dennis Sullivan came to the retrospective realization that he had, in effect, \emph{invented} \( L_\infty \)-algebras. The invention lay not in an act of axiomatic definition, but in uncovering the deep algebraic coherence hidden within models he had already constructed—models that, unbeknownst to him at the time, belonged to a new and powerful class of homotopy-theoretic structures \cite{Stasheff2016}.

\vspace{0.5em}
For more detailed accounts, see Stasheff’s retrospective \cite{stasheff2018linftyainftystructures, Stasheff2016}, as well as comprehensive surveys on the use of \( L_\infty \)-structures in physics and geometry.

\subsection{Applications in Physics}

In theoretical physics, \( L_\infty \)-algebras have emerged as indispensable tools across a wide range of contexts. As observed already in the early 1980s, such structures often appeared in disguise. In particular, the work of D’Auria and Fré on supergravity (1982) introduced what they termed \emph{free differential algebras}, which in hindsight can be recognized as precursors of \( L_\infty \)-algebras \cite{Stasheff2016}. 

A few years later, deeper algebraic structures surfaced in the BRST formalism within Batalin–Fradkin–Vilkovisky theory \cite{BatalinFradkin1983, FradkinFradkina1978, FradkinVilkovisky1975}. There, gauge symmetries and their higher-order relations naturally pointed to \( L_\infty \)-type structures, as was later made precise in the mathematical work of Stasheff and collaborators \cite{Stasheff1988, Stasheff1992, Stasheff1997, Stasheff1998}. 

A decisive breakthrough came in string theory: in 1989 Zwiebach discovered that the closed string field theory is governed by an \( L_\infty \)-algebra structure \cite{Zwiebach1993}, which encodes both the interactions and the gauge symmetries of the theory. This result motivated Lada and Stasheff to provide the first systematic algebraic formulation of \( L_\infty \)-algebras \cite{Lada1993}, thereby establishing a rigorous framework for the structures physicists had encountered. 

By the late 1990s, applications had spread further. In particular, Roytenberg and Weinstein showed that Courant algebroids admit a natural \( L_\infty \)-algebra description \cite{roytenberg1998courant}, thus linking the theory of higher homotopy algebras with generalized complex geometry and Dirac structures. These developments opened the door to a broad range of interactions between geometry, mathematical physics, and higher algebra.

\subsection{Modern Relevance}

Today, \( L_\infty \)-algebras serve as a foundational language in derived geometry, deformation quantization, and multisymplectic geometry. In particular, they provide the natural algebraic framework for describing observables and symmetries in higher geometric contexts such as \( n \)-plectic \cite{rogers2011higher, ryvkin2016multisymplectic, callies2016homotopy, Fiorenza2014, CohomologicalFramework2015}. The foundational work by Rogers \cite{rogers2011higher} and Zambon et al. \cite{ callies2016homotopy} provided explicit constructions of these algebras, showing that closed and nondegenerate \((n+1)\)-forms on a manifold naturally induce \(L_\infty\)-algebras. This correspondence has since become a cornerstone of higher symplectic geometry, allowing the extension of classical moment map theory and Poisson brackets to field-theoretic and higher categorical contexts.

Building on these developments, the present thesis proposes a generalization to the \emph{relative} setting. That is, we investigate how multisymplectic structures, observables, and symmetry actions behave when defined not on a single manifold, but relative to a smooth map \( F: M \to N \). This shift is motivated by examples such as quasi-Hamiltonian \( G \)-spaces, where the interaction between a manifold and a group target exhibits striking analogies with 2-plectic structures. In the next section, we motivate this perspective in detail and explain why relative geometry provides a natural framework for organizing such examples.

\section{Motivation for Relative Multisymplectic Geometry}
\label{sec:motivation_relative}

This work began with a concrete observation: the defining equations of quasi-Hamiltonian \( G \)-spaces bear a close resemblance to those of 2-plectic geometry—except that they hold \emph{relative to a map}. In the quasi-Hamiltonian framework introduced by Alekseev, Malkin, and Meinrenken~\cite{10.4310/jdg/1214460860}, one considers a \( G \)-manifold \( M \) equipped with a 2-form \( \omega \) and a group-valued moment map \( \Phi: M \to G \), satisfying
\[
d\omega + \Phi^* \eta=0,
\]
where \( \eta \) is a canonical 3-form on the Lie group \( G \). Here, \( \omega \) is neither closed nor nondegenerate in the usual sense, yet the pair \( (\omega, \eta) \) satisfies a kind of geometric consistency that mirrors the multisymplectic condition, albeit in a relative form. The nondegeneracy of the relative 3-form is controlled by the condition \[\ker \omega_m \cap \ker (d\Phi_m )= \{ 0 \}, \quad \forall m\in M,\]
a condition that comes from the definition of a quasi-Hamiltonian \(G\)-space itself.

This analogy led to a natural question: can we interpret quasi-Hamiltonian \( G \)-spaces as examples of 2-plectic geometry formulated \emph{relative} to a map? More broadly, can we extend multisymplectic geometry to accommodate such relative structures, where closedness and nondegeneracy conditions are expressed not absolutely, but in relation to a smooth map \( F: M \to N \)? This was the original motivation for the theory developed in this thesis.

Rather than treating the failure of closedness as a defect, we shift perspective and consider a pair of forms \( (\omega_M, \omega_N) \), one on \( M \) and one on \( N \), satisfying the relation
\[d\omega_M + F^* \omega_N=0
\]
and 
\[
\ker(\omega_M)_m \cap \ker(dF_m) = \{0\} \quad \text{for all } m \in M,
\]
with \( \omega_N \in \Omega^{n+1}(N) \)  nondegenerate in the usual \( n \)-plectic sense.

Then the relative form \( (\omega_M, \omega_N) \) is nondegenerate in the sense of Definition~\ref{relative nondegenerate}.

This defines what we call a \emph{relative \( n \)-plectic structure}. From this viewpoint, quasi-Hamiltonian \( G \)-spaces become prototypical examples of relative 2-plectic manifolds, with the moment map playing the role of the smooth map. This reinterpretation provides a unifying framework that clarifies the geometric content of quasi-Hamiltonian \(G\)-spaces and their relation to higher geometry.

Once this perspective is adopted, it becomes natural to develop the entire formalism of multisymplectic geometry—observables, Hamiltonian vector fields, and moment maps—in the relative setting. We show that relative \( n \)-plectic structures give rise to \( L_\infty \)-algebras of observables, extending foundational work by Rogers~\cite{rogers2011higher} and Zambon et al. \cite{ callies2016homotopy}. We also construct homotopy moment maps in the relative setting, and demonstrate that the quasi-Hamiltonian case arises as a special instance of this general theory.

Although this thesis is primarily motivated by the analogy with quasi-Hamiltonian spaces, we believe that the framework of relative multisymplectic geometry has broader potential. In particular, it opens new avenues for studying geometric structures governed by maps, such as those appearing in gauge theories with boundaries and group-valued moment maps.

The present work focuses on laying the foundations of this relative theory and clarifying its relationship with existing structures in multisymplectic and quasi-Hamiltonian geometry. Having introduced the main conceptual motivation, we now turn to a summary of the principal results established in this thesis.

\section{Main Results}

This section summarizes the main contributions of this thesis, which center around extending multisymplectic geometry into the relative setting. Our goal is to demonstrate how geometric and algebraic structures—such as \(L_\infty\)-algebras of observables and homotopy moment maps—can be formulated and studied when differential forms and vector fields are defined relative to a smooth map \(F: M \to N\).

\subsection{\Linfty-algebras of Relative Observables}

The first result is the construction of an \(L_\infty\)-algebra of observables associated to any relative \(n\)-plectic structure. This provides a direct analogue of the well-known construction for \(n\)-plectic manifolds due to Rogers~\cite[Theorem 3.14]{rogers2011higher}.

\begin{repthm} {\ref{Relative L-infinity algebra main thm}'}\label{them4.2.1'}
Given a relative $n$-plectic structure  $(F,\omega_M, \omega_N)$, there is a Lie $n$-algebra
\[
\resizebox{\textwidth}{!}{$
\begin{array}{ccccccccccccccc}
    0 & \longrightarrow & L_{n-1} & \longrightarrow & L_{n-2}  & \longrightarrow & \cdots & \longrightarrow & L_{k-2} & \longrightarrow & \cdots & \longrightarrow & L_{1} & \longrightarrow & L_0\\
     &  & \parallel &  & \parallel &  &  &  & \parallel &  &  &  & \parallel &  & \parallel \\
     &  & \Omega^{0}(F)&  & \Omega^{1}(F)  &  &  &  & \Omega^{k-2}(F) &  &  &  & \Omega^{n-2}(F) &  & \Omega^{n-1}_{\textrm{Ham}}(F)\\ 
\end{array}
$}\]
denoted $\mathrm{L}_\infty(F,\omega_M, \omega_N)=(\mathrm{L},\{l_{k} \})$ with 
\begin{itemize}
    \item 
underlying graded vector space \(L_0=\Omega_{\mathrm{Ham}}^{n-1}(F)\) and \(L_i=\Omega^{n-1-i}(F) \) for  \(0 \leq i < n-1\)

and 
\item maps  $\left \{l_{k} : \mathrm{L}^{\otimes k} \to \mathrm{L}\, \, |\, \,  1
  \leq k < \infty \right\}$ defined as
\[ 
l_{1}(f, \alpha)=\mathrm{d}(f, \alpha)=(F^*\alpha+\mathrm{d}f, -\mathrm{d}\alpha),
\]
if $\deg(f, \alpha)>0$ and

\begin{align*}
&l_{k}\left((f_1,\alpha_1), \ldots, (f_k,\alpha_k) \right)\\
&\qquad\qquad =
\begin{cases}
(0, 0) & \text{if $\deg{(f_1,\alpha_1)\otimes \cdots\otimes (f_k,\alpha_k)} > 0$}, \\
(-1)^{\frac{k}{2}+1} \iota_{(u_k, v_k)}\cdots \iota_{(u_1, v_1)}  (\omega_M, \omega_N)  & \text{if
  $\deg{(f_1,\alpha_1)\otimes \cdots\otimes (f_k,\alpha_k)}=0$ }\\
  & \text{and $k$ even},\\
(-1)^{\frac{k-1}{2}}\iota_{(u_k, v_k)}\cdots \iota_{(u_1, v_1)}  (\omega_M, \omega_N)  & \text{if
  $\deg{(f_1,\alpha_1)\otimes \cdots\otimes (f_k,\alpha_k)}=0$}\\
  & \text{and $k$ is odd},
\end{cases}
\end{align*}
for $k>1$, where $(u_i, v_i)$ is the unique Hamiltonian vector field
associated to $(f_i, \alpha_{i}) \in \Omega_{\mathrm{Ham}}^{n-1}(F)$.
\end{itemize}
\end{repthm}
Next, we generalize this construction to the case of relative pre-\(n\)-plectic structures, where the nondegeneracy condition is not required. This yields an extended class of \(L_\infty\)-algebras based on compatible pairs of vector fields and differential forms, extending the result of Callies et al.~\cite[Theorem 4.7]{callies2016homotopy} to the relative setting.

\begin{repthm}{\ref{theorem:Lie algebra of obs for pre-n-plectic}'}
Given a relative  pre-$n$-plectic structure \((F,\omega_M, \omega_N)\), there is a Lie $n$-algebra
\[
\resizebox{\textwidth}{!}{$
\begin{array}{ccccccccccccccc}
    0 & \longrightarrow & L_{n-1} & \longrightarrow & L_{n-2}  & \longrightarrow & \cdots & \longrightarrow & L_{k-2} & \longrightarrow & \cdots & \longrightarrow & L_{1} & \longrightarrow & L_0\\
     &  & \parallel &  & \parallel &  &  &  & \parallel &  &  &  & \parallel &  & \parallel \\
     &  & \Omega^{0}(F)&  & \Omega^{1}(F)  &  &  &  & \Omega^{k-2}(F) &  &  &  & \Omega^{n-2}(F) &  & \widetilde{\Omega^{n-1}_{\textrm{Ham}}}(F)\\ 
\end{array}
$}
\]
denoted $\mathrm{\textbf{Ham}}_{\infty}(F,\omega_M, \omega_N)$, with 
\begin{itemize}
    \item underlying graded vector space \(L\):
\begin{equation*}
\begin{split}
L_{0} & = \widetilde{\Omega^{n-1}_{\mathrm{Ham}}}(F)=\{ (u,  v)\oplus (f,\alpha) \in \mathfrak{X}(F)\oplus \Omega_{\mathrm{Ham}}^{n-1}(F) \, \mid \, \mathrm{d}(f, \alpha)= -\iota_{(u,v)}(\omega_M, \omega_N) \} \\
L_{i} & = \Omega^{n-1-i}(F) ,\quad 0 \leq i < n-1,
\end{split}
\end{equation*}

and 
\item structure maps:
\begin{equation*}
\begin{split}
\tilde{l}_{1}(f, \alpha)&=
\begin{cases}
0\oplus \mathrm{d}(f, \alpha) & \text{if $\deg{(f, \alpha)}=1$},\\
\mathrm{d} (f, \alpha) & \text{if $\deg{(f, \alpha)} >1$,}
\end{cases}\\
\tilde{l}_{2}(x_{1},x_{2}) &= 
\begin{cases}
\left( [(u_1, v_{1}), (u_2, v_{2})], \{(f_1, \alpha_1), (f_2, \alpha_2)\} \right) &
\text{if $\deg{x_{1} \otimes x_{2}}  = 0$,}\\
 0  & \text{otherwise},
\end{cases}
\end{split}
\end{equation*}
and, for $k > 2$: 
\[
l_k(x_1, \ldots, x_k) =
\begin{cases}
 0 & \text{if } \deg{x_1\otimes \cdots\otimes x_k}> 0\\
 \epsilon(k) \iota_{(u_k, v_k)}\cdots \iota_{(u_1, v_1)}  (\omega_M, \omega_N)  & \text{if } \deg{x_1\otimes \cdots\otimes x_k} = 0
\end{cases}
\]
where \(x_i=(u_i,  v_i)\oplus (f_i,\alpha_i)\) in \(L_0\) for \(i=1,2, \ldots, k\),  and \(x_i=(f_i,\alpha_i)\) in \(L_j\), for \(j>0\), and for \(i=1,2, \ldots, k\).

\end{itemize}
\end{repthm}

Moreover, by modifying the above construction to avoid reference to Hamiltonian vector fields, we obtain a purely form-based version of the \( L_\infty \)-algebra, once again extending the classical pre-\(n\)-plectic framework, as given in \cite[Theorem 4.6]{callies2016homotopy} , to the relative context.

\begin{repthm}{\ref{n-algebra from relative structure}'}
Given a relative pre-\(n\)-plectic structure  \((F,\omega_M, \omega_N)\), there exists a Lie \(n\)-algebra 
\[
\resizebox{\textwidth}{!}{$
\begin{array}{ccccccccccccccc}
    0 & \longrightarrow & L_{n-1} & \longrightarrow & L_{n-2}  & \longrightarrow & \cdots & \longrightarrow & L_{k-2} & \longrightarrow & \cdots & \longrightarrow & L_{1} & \longrightarrow & L_0\\
     &  & \parallel &  & \parallel &  &  &  & \parallel &  &  &  & \parallel &  & \parallel \\
     &  & \Omega^{0}(F) &  & \Omega^{1}(F) &  &  &  & \Omega^{k-2}(F) &  &  &  & \Omega^{n-2}(F) &  & \Omega^{n-1}_{\textrm{Ham}}(F)\\ 
\end{array}
$}
\]
denoted again \(L_\infty(M, \omega) = (L, \{l_k\})\), with the underlying graded vector space:
\[
L_i =
\begin{cases} 
   \Omega^{n-1}_{\mathrm{Ham}}(F) & \text{if } i = 0, \\
    \Omega^{n-1-i}(F) & \text{if } 0 \leq i < n-1,
\end{cases}
\]
and structure maps
\[
l_k : L^{\otimes k} \to L \quad \text{for } 1 \leq k < \infty,
\]
defined as:
\[
l_1(\alpha) = \mathrm{d}(f,\alpha), \quad \text{if } \deg(f, \alpha)> 0,
\]
and
\[
\resizebox{\textwidth}{!}{$
l_k((f_1, \alpha_1), \ldots, (f_k, \alpha_k)) =
\begin{cases}
 0, & \text{if } \deg\big((f_1,\alpha_1)\otimes \cdots\otimes (f_k,\alpha_k)\big) > 0, \\[1.5ex]
 \epsilon(k)\, \iota_{(u_k, v_k)} \cdots \iota_{(u_1, v_1)} (\omega_M, \omega_N), & \text{if } \deg\big((f_1,\alpha_1)\otimes \cdots\otimes (f_k,\alpha_k)\big) = 0.
\end{cases}
$}
\]
for \(k > 1\), where \((u_i, v_{i})\) is any Hamiltonian vector field associated to \((f_i, \alpha_i) \in\Omega^{n-1}_{\mathrm{Ham}}(F)\).
\end{repthm}

These constructions are not isolated: we further show that in the case where the relative form is nondegenerate—i.e., when we have a relative \(n\)-plectic structure—the complex of relative forms gives rise to a strict \(L_\infty\)-quasi-isomorphism between the algebra of Hamiltonian forms and the more general pre-\(n\)-plectic model. This result establishes an equivalence between the two descriptions at the homotopical level.

\[
\xymatrix{
\Omega^0(F) \ar[r]^-{\mathrm{d}} \ar[d]^{\mathrm{id}} & \Omega^{1}(F) \ar[d]^{\mathrm{id}} \ar[r]^-{\mathrm{d}} & \cdots \ar[r]^-{\mathrm{d}}
& \Omega^{n-2}(F) \ar[d]^{\mathrm{id}} \ar[r]^-{\mathrm{d}} & \Omega_{\mathrm{Ham}}^{n-1}(F) \ar[d]^{\phi} \\
\Omega^0(F) \ar[r]^-{\mathrm{d}} & \Omega^{1}(F) \ar[r]^-{\mathrm{d}} & \cdots \ar[r]^-{\mathrm{d}} &
\Omega^{n-2}(F) \ar[r]^-{0 \oplus \mathrm{d}} & \widetilde{\Omega_{\mathrm{Ham}}^{n-1}(F)} 
}
\]
\vspace{-2mm}
\[
L_{\infty}(F,\omega_M, \omega_N) \xrightarrow{\sim} \mathrm{\textbf{Ham}}_{\infty}(F,\omega_M, \omega_N),
\]
where the map \( \phi(f, \alpha) = (u, v) \oplus (f, \alpha) \) uses the unique Hamiltonian vector field associated to \((f, \alpha)\).

\subsection{The Differential Graded Leibniz Algebra Associated to a Relative \nplectic Structure}

In 2-plectic geometry, observables are modeled by Hamiltonian 1-forms, and the resulting algebraic structure is not a Lie algebra but a Lie 2-algebra. Two well-known brackets used in this context are the \emph{hemi-bracket} and the \emph{semi-bracket}.

\paragraph{Hemi-bracket.} Given Hamiltonian 1-forms \( \alpha, \beta \), the hemi-bracket is defined by
\[
\{\alpha, \beta\}_{\text{h}} := \mathcal{L}_{v_\alpha} \beta,
\]
where \( v_\alpha \) is the Hamiltonian vector field associated with \( \alpha \). 

\begin{itemize}
    \item The hemi-bracket is \emph{not} skew-symmetric.
    \item It satisfies the Jacobi identity. 
    \item It gives rise to a \emph{hemistrict Lie 2-algebra}.
\end{itemize}

\paragraph{Semi-bracket.} The semi-bracket is given by
\[
\{\alpha, \beta\}_{\text{s}} := \iota_{v_\beta} \iota_{v_\alpha} \omega,
\]
where \( \omega \) is the 2-plectic form.

\begin{itemize}
    \item The semi-bracket is skew-symmetric.
    \item The Jacobi identity holds only up to an exact form.
    \item It defines a \emph{semistrict Lie 2-algebra}.
\end{itemize}

\medskip

Replacing the semi-bracket with the hemi-bracket leads naturally to the structure of a \emph{differential graded Leibniz algebra}, as introduced in~\cite[Proposition 6.3]{Rogers_2011}. We recall the definition for context:

\begin{defn}
A \emph{differential graded Leibniz algebra} is a triple \( (L, \delta, \langle \cdot, \cdot \rangle) \) consisting of:
\begin{itemize}
    \item a graded vector space \(L = \displaystyle \bigoplus_{i \in \mathbb{Z}} L_i\),
    \item a differential \( \delta: L \to L \) of degree \( -1 \), and
    \item a bilinear bracket \( \langle \cdot, \cdot \rangle : L \otimes L \to L \) of degree 0,
\end{itemize}
satisfying the following identities for all homogeneous \( x, y, z \in L \):
\begin{align*}
\delta \circ \delta &= 0, \\
\delta \langle x, y \rangle &= \langle \delta x, y \rangle + (-1)^{\deg(x)} \langle x, \delta y \rangle, \\
\langle x, \langle y, z \rangle \rangle &= \langle \langle x, y \rangle, z \rangle + (-1)^{\deg(x)\deg(y)} \langle y, \langle x, z \rangle \rangle.
\end{align*}
\end{defn}

\medskip

We now extend this structure to the relative setting. Building on the above and following the construction of~\cite[Proposition 6.3]{Rogers_2011}, we prove that any relative \(n\)-plectic structure induces a differential graded Leibniz algebra.

\begin{reppro}{\ref{proposition: Leibniz dg algebra}'}
Given a relative \(n\)-plectic structure \( (F, \omega_M, \omega_N) \), there exists a differential graded Leibniz algebra
\[
\mathrm{Leib}(F,\omega_M, \omega_N) = (L, \delta, \langle \cdot, \cdot \rangle),
\]
where:
\begin{itemize}
    \item The underlying graded vector space is given by
    \[
    L_0 = \Omega_{\mathrm{Ham}}^{n-1}(F), \qquad L_i = \Omega^{n-1-i}(F) \quad \text{for } 0 < i \leq n-1.
    \]

    \item The differential \( \delta: L \to L \) is defined by
    \[
    \delta(f, \alpha) = d(f, \alpha) = \left( F^*\alpha + d f,\; -d\alpha \right),
    \]
    whenever \( \deg(f, \alpha) > 0 \).

    \item The bracket \( \langle \cdot, \cdot \rangle: L \otimes L \to L \) is given by
    \[
    \langle (f, \alpha), (g, \beta) \rangle =
    \begin{cases}
    \mathcal{L}_{(u_1, v_1)} (g, \beta) & \text{if } \deg(f, \alpha) = 0, \\
    (0, 0) & \text{if } \deg(f, \alpha) > 0,
    \end{cases}
    \]
    where \( (u_1, v_1) \) is the Hamiltonian vector field associated to \( (f, \alpha) \in \Omega_{\mathrm{Ham}}^{n-1}(F) \).
\end{itemize}
\end{reppro}

Now we discuss the Lie 2-algebra arising from a quasi-Hamiltonian \(G\)-space.

\subsection{Lie \Twoalgebras  Associated with a Quasi-Hamiltonian \Gspace}

As an application of Theorem~\ref{Relative L-infinity algebra main thm}, we now restrict to the case \(n = 2\), corresponding to relative 2-plectic structures. In this setting, we prove that a quasi-Hamiltonian \(G\)-space naturally gives rise to a Lie 2-algebra. In fact, two Lie 2-algebras arise depending on whether we use the semi-bracket or the hemi-bracket, resulting respectively in a semi-strict and a hemi-strict Lie 2-algebra structure.

We first describe the semi-strict version.

\begin{repthm}{\ref{theorem:semi strict Lie2}'}
Let \( (M, \Phi, \omega) \) be a quasi-Hamiltonian \(G\)-space. Then there exists a semistrict Lie 2-algebra
\[
\cdots \longrightarrow 0 \longrightarrow L_1 \longrightarrow L_0
\quad \text{with} \quad
L_1 = \Omega^0(\Phi), \quad L_0 = \Omega^1_{\mathrm{Ham}}(\Phi),
\]
denoted \( L_\infty(M, \Phi, \omega) = (L, [\cdot, \cdot], J) \), with the following structure:
\begin{itemize}
    \item The differential \( d: L_1 \to L_0 \) is the relative differential:
    \[
    f \mapsto (\Phi^* f, -df).
    \]
    \item The alternator \(S: L_0 \otimes L_0 \to L_1\) is trivial:
    \[
    S\left((f_1, \beta_1), (f_2, \beta_2)\right) = 0.
    \]
    \item The bracket is given by the \emph{semi-bracket}.
    \item The Jacobiator is the trilinear map \(J: L_0^{\otimes 3} \to L_1\), defined by:
    \[
    J\left((f_1, \beta_1), (f_2, \beta_2), (f_3, \beta_3)\right)
    = -\iota_{(u_1, v_1)} \iota_{(u_2, v_2)} \iota_{(u_3, v_3)} (\omega, \eta),
    \]
    where each \((u_i, v_i)\) is the Hamiltonian vector field corresponding to \((f_i, \beta_i)\).
\end{itemize}
\end{repthm}

We now describe the hemi-strict variant obtained by modifying the alternator and using the hemi-bracket:

\begin{repthm}{\ref{theorem:hemi strict Lie2}'}
Let \( (M, \Phi, \omega) \) be a quasi-Hamiltonian \(G\)-space. Then there exists a hemistrict Lie 2-algebra
\[
\cdots \longrightarrow 0 \longrightarrow L_1 \longrightarrow L_0
\quad \text{with} \quad
L_1 = \Omega^0(\Phi), \quad L_0 = \Omega^1_{\mathrm{Ham}}(\Phi),
\]
denoted \( L_\infty(M, \Phi, \omega)_h = (L, [\cdot, \cdot]_h, J) \), with the following structure:
\begin{itemize}
    \item The differential \( d: L_1 \to L_0 \) is again given by
    \[
    f \mapsto (\Phi^* f, -df).
    \]
    \item The alternator \(S: L_0 \otimes L_0 \to L_1\) is given by:
    \[
    S\left((f_1, \beta_1), (f_2, \beta_2)\right)
    = -\left(\iota_{(u_1, v_1)}(f_2, \beta_2) + \iota_{(u_2, v_2)}(f_1, \beta_1)\right).
    \]
    \item The bracket is given by the \emph{hemi-bracket}.
    \item The Jacobiator is trivial:
    \[
    J\left((f_1, \beta_1), (f_2, \beta_2), (f_3, \beta_3)\right) = 0.
    \]
\end{itemize}
\end{repthm}

These two Lie 2-algebras are shown to be quasi-isomorphic in Theorem~\ref{isomorphism}.

We conclude this sequence of results by constructing a \emph{homotopy moment map} for the infinitesimal action induced by the group action on a quasi-Hamiltonian \(G\)-space.

\subsection{Homotopy Moment Map}

In Section \ref{sec:homotopy_moment_maps}, we construct a homotopy moment map for Lie 2-algebras associated with quasi-Hamiltonian \( G \)-spaces.  Given a quasi-Hamiltonian \(G\)-space \((M, \Phi, \omega)\), we get a closed and non-degenerate \(3\)-form \((\omega, \eta)\) in the relative setting, where \(\eta\) is the Cartan bi-invariant 3-form. From Theorem \ref{theorem:semi strict Lie2}, we obtain a Lie 2-algebra structure \( L_\infty(M, \Phi, \omega) \). The Lie group \(G\) acts on the manifold \(M\) since \(M\) is a \(G\)-manifold by definition of a quasi-Hamiltonian \(G\)-space. The induced infinitesimal 
\[
(u, v): \mathfrak{g} \to \mathfrak{X}(\Phi)
\]
is Hamiltonian. Here, \(\mathfrak{X}(\Phi)\) denotes the pairs of vector fields \((u,v)\), where \(u\) is a vector field on \(M\) and \(v\) is a vector field on \(G\) and they are \(\Phi\)-related.
\begin{replem}{\ref{lemma homotopy relative}'}
   Let \( (M, \Phi, \omega) \) be a quasi-Hamiltonan \(G\)-space. The induced action  
     \begin{align*}
  (u, v):\mathfrak{g} &\to  \mathfrak{X}(\Phi)\\ 
   x & \mapsto  \left(u_x, v_{x} \right)
  \end{align*}
 is Hamiltonian (that is, the infinitesimal generators are Hamiltonian vector fields).
\end{replem}

 Following a similar construction in \cite{callies2016homotopy}, we prove in Theorem \ref{proposition: relative moment map} that the natural infinitesimal action $(u, v): \mathfrak{g} \to \mathfrak{X}(\Phi)$ lifts to a morphism of \(L_\infty\)-algebras $(f): \mathfrak{g} \to L_\infty(M, \Phi, \omega)$. That is, the data of a homotopy moment map $(f): \mathfrak{g} \to L_\infty(M, \Phi, \omega)$.

\begin{repthm}{\ref{proposition: relative moment map}'}
Let \(L_\infty(M, \Phi, \omega)\) be the Lie 2-algebra associated with the quasi-Hamiltonian \(G\)-space \( (M, \Phi, \omega) \), as defined in Theorem \ref{theorem:semi strict Lie2}. 

Let 
  \begin{align*}
  (u, v):\mathfrak{g} &\to  \mathfrak{X}(\Phi)\\ 
   x & \mapsto  (u_x, v_{x})
  \end{align*}

be the infinitesimal action of \(\mathfrak{g} \) on the space of relative vector fields \( \mathfrak{X}(\Phi) \). 

Then this action lifts to an \(L_\infty\)-morphism 
\[
(f) = (f_1, f_2): \mathfrak{g} \to L_\infty(M, \Phi, \omega)
\]
such that the diagram
\begin{equation*}
\label{the_lift}
\xymatrix{
    &&  L_\infty (M, \Phi, \omega) \ar[d]^{\pi} \\
     \mathfrak{g} \ar @{-->}[urr]^{(f)} \ar[rr]^{(u_{-}, v_{-})} && \mathfrak{X}_{\mathrm{Ham}}(\Phi)
}
\end{equation*}
commutes.

The components of the \(L_\infty\)-morphism \( (f_1, f_2) \) are given by:
\begin{align*}
    f_1: \mathfrak{g} &\to L_\infty(M, \Phi, \omega), \\
    x &\mapsto \left( 0, -\left(\frac{\theta^{L}+\theta^{R}}{2}\right)\cdot x \right),
\end{align*}
and
\begin{align*}
    f_2: \mathfrak{g} \otimes \mathfrak{g} &\to L_\infty(M, \Phi, \omega), \\
    x \otimes y &\mapsto \left( 0, \frac{1}{2} \iota_{v_x} \left( \left( \theta^L + \theta^R \right)\cdot y \right) \right).
\end{align*}
\end{repthm}

This completes the sequence of main results developed in the thesis. We now outline the structure of the document and the logical progression of ideas across the chapters in the following section.

\section{Organization of the Thesis}

This thesis is organized into five chapters, each focusing on distinct but interconnected aspects of relative \(n\)-plectic geometry and its applications to Lie \(n\)-algebras and quasi-Hamiltonian \(G\)-spaces. Below, we provide an overview of the structure and contents of each chapter:

\subsection*{Chapter 2: Preliminaries on Complexes and Cartan Forms}
This chapter introduces essential background material needed for the thesis. We begin with a review of homological algebra, focusing on complexes and the theory of relative (co)homology. The foundational concepts are presented to establish the necessary framework for later chapters. Primary references include \cite{Weibel1994}, \cite{BottTu1982}, and \cite{Loring2020}, which provide further depth on the topics discussed.

\subsection*{Chapter 3: Basics on Multisymplectic Geometry}
 This chapter focuses on:
\begin{itemize}
    \item Definitions and examples of multisymplectic manifolds.
    \item The role of multisymplectic geometry in the formulation of Hamiltonian classical field theory, as outlined in \cite{Marsden1998} and \cite{Forger2005}.
\end{itemize}
This chapter sets the stage for understanding \(n\)-plectic structures and their generalizations in the relative setting.

\subsection*{Chapter 4: Observables of Relative \nplectic Structures}
Here, we establish the central result that relative \(n\)-plectic structures give rise to both \(L_\infty\)-algebras and differential graded (dg) Lie algebras:
\begin{itemize}
    \item In Section \ref{section:relative cartan magic formulas}, we define relative Lie derivatives, contractions, and brackets, and verify that they satisfy Cartan’s calculus identities.
    \item We construct \(L_\infty\)-algebras for relative \(n\)-plectic structures, generalizing the results of \cite{callies2016homotopy, rogers2011higher, Fiorenza2014, ryvkin2016multisymplectic} (Theorem \ref{Relative L-infinity algebra main thm}).
    \item We prove the equivalence of the dg Lie algebra and the \(L_\infty\)-algebra associated with a relative \(n\)-plectic structure (Proposition \ref{proposition: Leibniz dg algebra}).
    \item Specific cases, such as quasi-Hamiltonian \(G\)-spaces, are shown to give rise to Lie \(2\)-algebras, including the Atiyah and Courant Lie algebras (Theorems \ref{theorem: atiyah lie 2-algebra} and \ref{theorem: courant lie 2-algebra}).
\end{itemize}

\subsection*{Chapter 5: Homotopy moment for the Lie $2$-algebras arising from a quasi-Hamiltonian $G$-space}
Building on the constructions of Chapter 4, this chapter focuses on quasi-Hamiltonian \(G\)-spaces and their associated Lie \(2\)-algebras:
\begin{itemize}

    \item We prove that quasi-Hamiltonian \(G\)-spaces give rise to relative \(3\)-plectic structures, which in turn lead to \(L_\infty\)-algebras and dg Lie algebras (Theorem \ref{Relative L-infinity algebra main thm}).
    \item Using the hemi-bracket and semi-bracket, we construct two distinct Lie \(2\)-algebras and prove their equivalence (Theorem \ref{isomorphism}).
    \item Applications to specific examples, such as group actions and moment maps, are discussed to illustrate the theory.
\end{itemize}

%% file: Chapter2.tex
\chapter{Preliminaries on Complexes and Cartan Forms}
This chapter lays the algebraic and geometric groundwork required for the development of relative \(n\)-plectic structures.
In the first part, we review homological algebra, with particular emphasis on relative (co)homology. We begin by recalling basic notions such as chain complexes, exact sequences, and homotopies, before constructing the algebraic mapping cone. 

The second part of the chapter focuses on Maurer–Cartan forms and their associated structural identities, including the Cartan 3-form and its equivariant extension. These results will be important in constructing the homotopy moment maps in Chapter~\ref{chapter 5}. Due to the technical nature of these computations, we have chosen to collect the necessary theorems, lemmas, and propositions here to simplify the proof of Theorem \ref{proposition: relative moment map} and enhance the overall clarity of the exposition.
The chapter concludes with a brief review of quasi-Hamiltonian \(G\)-spaces.

\section{Relative Homology/Cohomology} \label{section:relative cohomology}

We begin by revisiting concepts in homological algebra relevant to this thesis, starting with an overview of complexes, followed by an introduction to  relative (co)homology. 
For those readers seeking a more detailed exploration of these topics, we recommend consulting primary references such as \cite{Weibel1994}, \cite{BottTu1982}, \cite{Shahbazi_2006}, \cite{Guillemin1999}, and \cite[Chapter 21]{Loring2020}.

\subsection{Chain Complexes and Homology}\label{subsection:Chain Complexes and Homology}
Although the concepts defined here are applicable across abelian categories in general, our discussion will be confined to the category of \(R\)-modules, with \(R\) being a commutative ring. Let $\mathbf{Mod}_R$ denote the category of right $R$-modules.

\begin{defn}[Exact Sequences]\label{exact sequence}
    A sequence of \(R\)-module homomorphisms
\[ L \xrightarrow{f} M \xrightarrow{g} N \]
is said to be \emph{exact} if \(\text{Ker } g = \text{Im } f\). This condition inherently implies that \(g \circ f = 0\).

More generally, a sequence of homomorphisms
\[ \cdots \xrightarrow{f_{n+1}} M_n \xrightarrow{f_n} M_{n-1} \xrightarrow{f_{n-1}} \cdots \]
is defined as \emph{exact} if, for each index \(n\), the sequence
\[ M_{n+1} \xrightarrow{f_{n+1}} M_n \xrightarrow{f_n} M_{n-1} \]
satisfies the exactness condition, i.e., \(\text{Ker } f_n = \text{Im } f_{n+1}\).
\end{defn}

\begin{defn}[Short Exact Sequences] \label{definition: Short Exact Sequences}
An exact sequence of the form
\[ 0 \rightarrow L \xrightarrow{f} M \xrightarrow{g} N \rightarrow 0 \]
is specifically called a \emph{short exact sequence}. 
\end{defn}

This configuration highlights a situation where \(f\) is injective, \(g\) is surjective, and the image of \(f\) (i.e., \(\text{Im } f\)) matches the kernel of \(g\) (i.e., \(\text{Ker } g\)).

\begin{defn} \label{definition: chain complex}
    A \emph{chain complex} of $R$-modules $(C_\bullet, d_\bullet)$ is a family $\{C_n\}_{n\in\mathbb{Z}}$ of $R$-modules together with $R$-modules maps $d=d_n:C_n\to C_{n-1}$ called differentials such that  the composition of any two consecutive maps is the zero map. That is $d_{n}\circ d_{n+1}=0$. 
    \end{defn}
    
    For notational simplicity, the property $d_{n}\circ d_{n+1}=0$ is often represented as \(d^2 = 0\).   
    The complex may be written out as follows
    \begin{equation}\label{eq0}
    \begin{split}
        \cdots \rightarrow C_{n} \xrightarrow{\quad d_n\quad} C_{n-1} \xrightarrow{\quad d_{n-1}\quad} C_{n-2} \xrightarrow{\quad d_{n-2}\quad} C_{n-3} \rightarrow \cdots
    \end{split}
\end{equation}

\begin{defn}
    For a chain complex of $R$-modules $C_\bullet$, 
    \begin{itemize}
        \item the morphism $d_n$ are called differentials,
        \item the elements of $C_n$ are called the $n$-chains,
        \item the elements in the kernel $Z_n:=\ker(d_n:C_n\to C_{n-1})$  are called $n$-cycles,
        \item the elements in the image $B_n:=\mathrm{im}(d_{n+1}: C_{n+1}\to C_{n})$ are called the $n$-boundaries
    \end{itemize}
\end{defn}

Given that \(d_{n} \circ d_{n+1} = 0\), it follows that there are canonical inclusions

\[0 \hookrightarrow \operatorname{B}_n \hookrightarrow \operatorname{Z}_n \hookrightarrow C_n.\]

\begin{itemize}
    \item The quotient \(\operatorname{H}_n(C_{\bullet}) \coloneqq \operatorname{Z}_n / \operatorname{B}_n\), is called the degree-\(n\) \emph{chain homology} of \(C_{\bullet}\). This gives rise to a short exact sequence:
    
    \[0 \rightarrow \operatorname{B}_n \rightarrow \operatorname{Z}_n \rightarrow \operatorname{H}_n(C_{\bullet}) \rightarrow 0.\]
\end{itemize}

\begin{defn}[Cochain Complex of \(R\)-modules]
A \emph{ cochain complex} of \(R\)-modules consists of a sequence of \(R\)-module homomorphisms \(\{d^n: C^n \rightarrow C^{n+1}\}\), where each \(C^n\) is an \(R\)-module and each \(d^n\) is a differential that increases the degree by one. This structure forms a chain complex of $R$-modules in the opposite category of \(R\)-modules, denoted as \((C^\bullet, d^\bullet)\).
\end{defn}

\begin{rem}
A complex of $R$-modules is also referred to as \emph{homologically graded} if its differentials decrease the degree (chain complexes of $R$-modules), and \emph{cohomologically graded} if its differentials increase the degree  (cochain complexes of $R$-modules). 
\end{rem}

\begin{defn}[Morphism of chain complexes of $R$-modules]
   Let $(C_\bullet, d_\bullet)$ and $(D_\bullet, \delta_\bullet)$ be chain complexes of $R$-modules.  A \emph{morphism of chain complexes of $R$-modules} $f : C_\bullet \to D_\bullet $ is given by a family of morphisms $f_n : C_n \to D_n$ such that all the diagrams 
  \[  \xymatrix{ C_n \ar[r]^{d_ n} \ar[d]_{f_n} &  C_{n - 1} \ar[d]^{f_{n - 1}} \\  D_n \ar[r]^{\delta_n} &  D_{n - 1} }  \]
 commute. That is, $f_{n - 1}\circ d_n=\delta_n\circ f_n$ for all $n$.
\end{defn}

\begin{defn}
A \emph{homotopy} $h$ between a pair of morphisms of chain complexes of $R$-modules $f, g : C_\bullet \to D_\bullet $ is a collection of morphisms $h_n : C_n \to D_{n+ 1}$ such that we have
  \[  f_ n - g_ n = d_{n+ 1} \circ h_ n + h_{n - 1} \circ d_ n  \]
 for all $n$. 
 
 Two morphisms $f, g : C_\bullet \to D_\bullet $ are said to be \emph{homotopic} if a homotopy between $f$ and $g$ exists. 
 \end{defn}

\noindent
To simplify notation, we suppress subscripts on the differentials and uniformly write \( \mathrm{d} \) for all maps in the chain complex \( D \), instead of the previously used \( \mathrm{d}_n \). This aligns with the conventions adopted in later applications throughout the thesis.

\subsection{Algebraic Mapping Cone for Chain Complexes} \label{subsection:Algebraic Mapping Cone for Chain Complexes}

We now transition from the basic notions of chain complexes of $R$-modules and homotopies presented in the previous subsection to the mapping cone construction. This will ultimately lead to the definition of relative (co)homology and the algebraic model of relative differentials.

\begin{defn} \label{defn map cone}
    Let \(f_{\bullet}: C_\bullet \longrightarrow D_\bullet\) be a morphism of chain complexes of $R$-modules. The \emph{algebraic mapping cone} of \(f\), denoted by \(\operatorname{Cone}_{\bullet}(f)\), is defined as the chain complex where \(\operatorname{Cone}_n(f) = C_{n-1} \oplus D_{n}\), equipped with the differential
    \[
    \mathrm{d}(\alpha, \beta) = (f^*\beta + \mathrm{d}\alpha,\, -\mathrm{d} \beta).
    \]
\end{defn}

There are other possible choices for the sign convention in the definition of the differential; however, for this work, we adopt the above definition to ensure compatibility with the moment map on a quasi-Hamiltonian \(G\)-space.

\begin{lem}\label{lemma:d^2=0}
  The differential \(\mathrm{d}\) on \(\operatorname{Cone}_{\bullet}(f)\) satisfies \(\mathrm{d}^2 = 0\).
\end{lem}

\begin{proof}
To verify that \(\mathrm{d}^2 = 0\), we compute the action of \(\mathrm{d}\) twice on an arbitrary element \((\alpha, \beta) \in \operatorname{Cone}_n(f) = C_{n-1} \oplus D_n\):
    \begin{align*}
    \mathrm{d}^2(\alpha, \beta) &= \mathrm{d} \left(f^*\beta + \mathrm{d}\alpha, -\mathrm{d} \beta\right) \\
        &= \left(f^*(-\mathrm{d} \beta) + \mathrm{d}(f^*\beta + \mathrm{d}\alpha), -\mathrm{d}(-\mathrm{d} \beta)\right) \\
        &= \left(-f^*(\mathrm{d} \beta) + \mathrm{d}f^*\beta + \mathrm{d}^2\alpha, \mathrm{d}^2 \beta\right) \\
        &= \left(-f^*(\mathrm{d} \beta) + f^*(\mathrm{d} \beta) + 0, 0\right) \quad \text{since } \mathrm{d}^2 = 0 \text{ and } f \circ \mathrm{d} = \mathrm{d} \circ f \\
        &= (0, 0).
\end{align*}

\end{proof}
The homology of the mapping cone gives rise to the notion of \emph{relative homology} for the map \(f_\bullet\).

\begin{defn}

The \emph{relative homology} of \(f_{\bullet}\) is defined as
\[
H_n(f) := H_n(\operatorname{Cone}_{\bullet}(f)).
\]
\end{defn}

This construction fits naturally into a short exact sequence of chain complexes of $R$-modules:
\[
0 \rightarrow D_{n} \xrightarrow{j_n} \operatorname{Cone}_n(f) \xrightarrow{k_n} C_{n-1} \rightarrow 0,
\]
where \(j_n(\beta) = (0, \beta)\) and \(k_n(\alpha, \beta) = \alpha\), for each integer \(n\). These sequences are compatible with the differentials, meaning they form a short exact sequence of chain complexes of $R$-modules. This induces a long exact sequence in homology:
\begin{equation}\label{eq0}
    \begin{split}
        \cdots \rightarrow H_{n}(D) \xrightarrow{j_*} H_{n}(f) \xrightarrow{k_*} H_{n-1}(C) \xrightarrow{\delta} H_{n-1}(D) \rightarrow \cdots,
    \end{split}
\end{equation}
where \(\delta\) is the connecting homomorphism.

\begin{lem}
    The connecting homomorphism \(\delta\) is characterized by \(\delta[\gamma] = [f(\gamma)]\) for \(\gamma \in C_{n-1}\).
\end{lem}

\begin{proof}
    For \(\gamma \in C_{n-1}\), we observe that \(k_n(\gamma, 0) = \gamma\). The short exact sequence yields an element \(\gamma' \in D_{n-1}\) such that
    \[
    j_{n-1}(\gamma') = \partial(\gamma, 0) = (\partial\gamma, f(\gamma)),
    \]
    where \(\partial\) is the differential on \(\operatorname{Cone}_{\bullet}(f)\). By definition, \(\delta[\gamma] = [\gamma']\). Since \(j_{n-1}(\gamma') = (0, \gamma')\), we deduce \(f(\gamma) = \gamma'\), which implies
    \[
    \delta[\gamma] = [f(\gamma)].
    \]
\end{proof}

The mapping cone also characterizes quasi-isomorphisms:
\begin{defn}
    A chain map \(f_{\bullet}: C_{\bullet} \rightarrow D_{\bullet}\) is a \emph{quasi-isomorphism} if it induces an isomorphism in homology, i.e., \(H_{\bullet}(C) \overset{\cong}\rightarrow H_{\bullet}(D)\).
\end{defn}

\begin{cor}
    A chain map \(f_{\bullet}: C_{\bullet} \rightarrow D_{\bullet}\) is a quasi-isomorphism if and only if \(H_{\bullet}(f) = 0\).
\end{cor}

\begin{proof}
    The map \(f\) is a quasi-isomorphism if and only if the connecting homomorphism \(\delta\) in the long exact sequence \eqref{eq0} is an isomorphism. 
   
\end{proof}

An important feature of the mapping cone is that it behaves well under homotopy equivalence of chain maps.

\begin{pro}\cite{shahbazi2006prequantization}
Any homotopy between chain maps \(f, g : C_\bullet \to D_\bullet\) induces an isomorphism of chain complexes between \(\operatorname{Cone}(f)_\bullet\) and \(\operatorname{Cone}(g)_\bullet\).
\end{pro}

\begin{proof}
Suppose there exists a homotopy \( h : C_\bullet \to D_{\bullet+1} \) between \( f \) and \( g \), i.e., \( g - f = \mathrm{d} h + h \mathrm{d} \) where \(\mathrm{d} \) denotes the boundary operator in the complexes \( C_\bullet \) and \( D_\bullet \).

Define a map \( F : \operatorname{Cone}_\bullet(f) \to \operatorname{Cone}_\bullet(g) \) by setting
\[
F(\alpha, \beta) = (\alpha, -h(\alpha) + \beta).
\]
for each \((\alpha, \beta) \in \operatorname{Cone}_\bullet(f)\), where \(\alpha \in C_\bullet\) and \(\beta \in Y_{\bullet+1}\).

We verify that \( F \) is a chain map:
\begin{align*}
\mathrm{d} F(\alpha, \beta) &= \mathrm{d}(\alpha, -h(\alpha) + \beta) \\
&= (\mathrm{d} \alpha, g(\alpha) - \mathrm{d} h(\alpha) + \mathrm{d} \beta) \\
&= (\mathrm{d} \alpha, f(\alpha) + (\mathrm{d} h(\alpha) + h \mathrm{d} \alpha) - \mathrm{d} h(\alpha) + \mathrm{d} \beta) \\
&= (\mathrm{d} \alpha, f(\alpha) + h \mathrm{d} \alpha + \mathrm{d} \beta) \\
&= F(\mathrm{d} \alpha, f(\alpha) + \mathrm{d} \beta) \\
&= F \mathrm{d} (\alpha, \beta),
\end{align*}
where the third equality follows from the homotopy condition \( g - f = \mathrm{d} h + h \mathrm{d} \).

The inverse of \( F \) can be constructed as follows:
Define \( F^{-1} : \operatorname{Cone}_\bullet(g) \to \operatorname{Cone}_\bullet(f) \) by
\[
F^{-1}(\alpha, \beta) = (\alpha, h(\alpha) + \beta).
\]
A similar calculation shows that \( F^{-1} \) is also a chain map, and by the definitions, it is clear that \( F \) and \( F^{-1} \) are inverses of each other, thus establishing the isomorphism of chain complexes.
\end{proof}
  
This concludes the algebraic setup. In the next sections, we shift our attention to the differential-geometric side of the story, focusing on Lie group actions and the associated Maurer–Cartan forms. 

\section{Group Action and Generating Vector Fields}

We review the basic structure of Lie groups and Lie algebras, the notion of group actions on manifolds, and the associated infinitesimal actions given by fundamental vector fields. This prepares the ground for introducing Maurer–Cartan forms.

\begin{defn}
A \emph{Lie algebra} is a vector space $\mathfrak{g}$ equipped with a bilinear skew-symmetric map $\displaystyle [\cdot,\cdot]:\mathfrak{g}\times \mathfrak{g}\to \mathfrak{g}$ that satisfies the following identity called the Jacobi identity:
\begin{eqnarray*}
[x,[y,z]]+[z,[x,y]]+[y,[z,x]]=0 \qquad \mbox{for all $x,y,z\in \mathfrak{g}$}.
\end{eqnarray*}
\end{defn}

\begin{defn}
Let $(\mathfrak{g}, [\cdot, \cdot ]_\mathfrak{g})$ and $(\mathfrak{h},  [\cdot, \cdot ]_\mathfrak{h})$ be two Lie algebras. A \emph{homomorphism of Lie algebras} from $\mathfrak{g}$ to $\mathfrak{h}$ is a linear map $\phi : \mathfrak{g} \to \mathfrak{h}$ such that for all $x,y \in \mathfrak{g}$ we have
 $$\phi([x,y]_{\mathfrak{g}}) = [\phi(x),\phi(y)]_{\mathfrak{h}} \,.$$ 
\end{defn}

\begin{defn}
An \emph{action} of a Lie group $G$ on a manifold $M$ is a group homomorphism  
\begin{eqnarray*}
\mathcal{A}:G &\to & \mathrm{Diff}(M)\\
g & \mapsto & \mathcal{A}_g
\end{eqnarray*}
from the group $G$ into the group of diffeomorphisms $\mathrm{Diff}(M)$, such that the action map 
\begin{eqnarray*}
\mathcal{A}:G\times M&\to &M\\
(g,m) & \mapsto & \mathcal{A}_g(m):=g\cdot m
\end{eqnarray*}
is smooth.

 By a \emph{$G$-manifold}, we mean a manifold $M$ together with a smooth action $\mathcal{A}:G\times M\to M$ of the group $G$ on M. 
\end{defn}

Let $\mathfrak{X}(M)$ denote the space of vector fields on $M$. We can view vector fields as derivations on $C^\infty(M)$, the algebra of smooth functions on $M$. Considering the usual bracket $[X,Y]=X\circ Y-Y\circ X$, we can see that $\mathfrak{X}(M)$ is a Lie algebra.

\begin{defn}
Let \( G \) be a Lie group with its corresponding Lie algebra \( \mathfrak{g} \), identified with the tangent space at the identity element \( e \) of \( G \), denoted by \( T_eG \). Define the actions of \( G \) on itself:
\begin{itemize}
    \item Left multiplication \( \mathrm{L}_g \) by \( g \) is given by \( \mathrm{L}_g(h) = g \cdot h \).
    \item Right multiplication \( \mathrm{R}_g \) by \( g \) is given by \( \mathrm{R}_g(h) = h \cdot g \).
    \item The adjoint action \( \mathrm{Ad}_g \) (or conjugation) by \( g \) is given by \( \mathrm{Ad}_g(h) = g h g^{-1} \).
\end{itemize}

\end{defn}

\begin{defn}[Left- and Right-Invariant Vector Fields]
Let \( G \) be a Lie group with identity element \( e \), and let \( \mathfrak{g} = T_e G \) denote its Lie algebra.

\begin{itemize}
    \item A vector field \( X \in \mathfrak{X}(G) \) is called \emph{left-invariant} if
    \[
    (\mathrm{L}_g)_* X(h) = X(gh) \quad \text{for all } g,h \in G,
    \]
    where \( \mathrm{L}_g(h) = gh \) denotes left multiplication. For \( x \in \mathfrak{g} \), the left-invariant vector field \( x^L \) is defined by:
    \[
    x^L(g) := (\mathrm{L}_g)_* x.
    \]

    \item A vector field \( X \in \mathfrak{X}(G) \) is called \emph{right-invariant} if
    \[
    (\mathrm{R}_g)_* X(h) = X(hg) \quad \text{for all } g,h \in G,
    \]
    where \( \mathrm{R}_g(h) = hg \) denotes right multiplication. For \( x \in \mathfrak{g} \), the right-invariant vector field \( x^R \) is defined by:
    \[
    x^R(g) := (\mathrm{R}_g)_* x.
    \]
\end{itemize}
\end{defn}

\begin{defn}
Let \(G\) be a Lie group acting on a manifold \(M\), and let \(\mathfrak{g}\) be the corresponding Lie algebra of \(G\), i.e., \(\mathfrak{g} = T_e G\). The action of \(G\) on \(M\) induces an action of \(\mathfrak{g}\) on \(M\), defined as follows:
\[
\mathfrak{g} \times M \to M, \quad (x, m) \mapsto \frac{d}{dt}\bigg|_{t=0} \exp(-t x) \cdot m =: v_x(m),
\]
where \(\exp: \mathfrak{g} \to G\) is the exponential map of \(G\). The vector field \( v_x\) on \(M\) is called the \emph{generating vector field} corresponding to \(x \in \mathfrak{g}\).
\end{defn}

\begin{defn}\label{def:infinitesimal_action}
Let $\mathfrak{g}$ be a Lie algebra and $M$ a smooth manifold. An \emph{infinitesimal action} of $\mathfrak{g}$ on $M$ is a Lie algebra homomorphism
\[
\rho: \mathfrak{g} \to \mathfrak{X}(M),
\]
where $\mathfrak{X}(M)$ denotes the Lie algebra of smooth vector fields on $M$ equipped with the standard Lie bracket. Explicitly, $\rho$ satisfies:
\begin{enumerate}
    \item \textbf{Linearity}: $\rho(ax + by) = a\rho(x) + b\rho(y)$ for all $x,y \in \mathfrak{g}$, $a,b \in \mathbb{R}$;
    \item \textbf{Lie Bracket Compatibility}: $\rho([x,y]_{\mathfrak{g}}) = [\rho(x), \rho(y)]_{\mathfrak{X}(M)}$ for all $x,y \in \mathfrak{g}$.
\end{enumerate}
For $x \in \mathfrak{g}$, we typically denote $\rho(x) \in \mathfrak{X}(M)$ by $v_x$ and call it the \emph{fundamental vector field} associated to $x$.
\end{defn}

\begin{rem}
This definition is equivalent to specifying a smooth left action of the corresponding simply connected Lie group $G$ on $M$ at the infinitesimal level. The fundamental vector fields $v_x$ generate the flow of the group action.
\end{rem}

\begin{exa}
For the conjugation action of $G$ on itself, the infinitesimal action is:
\[
\rho: \mathfrak{g} \to \mathfrak{X}(G), \quad x \mapsto v_x := x^L -x^R,
\]
where $x^L$ and $x^R$ are the left- and right-invariant vector fields corresponding to $x \in \mathfrak{g}$.
\end{exa}

This sets the stage for the introduction of Maurer–Cartan forms, which will be used in Chapter~\ref{chapter 5} to construct the homotopy moment map associated with the Lie 2-algebra of observables arising from a quasi-Hamiltonian \(G\)-space.

\section{The Maurer-Cartan Form and Cartan 3-Form on Lie Groups}
\label{sec:cartan_forms}
We turn to the left- and right-invariant Maurer–Cartan forms, their defining properties, and their behaviour under conjugation. The identities established in this section will facilitate the exposition 
 of the proof of Theorem \ref{proposition: relative moment map} . 

Let \( G \) be a Lie group with Lie algebra \( \mathfrak{g} \). When considering the conjugation action of \( G \) on itself, we define for any \( x \in \mathfrak{g} \) the fundamental vector field:
\begin{equation}
     v_x = x^L - x^R
\end{equation}
where \( x^L \) and \( x^R \) are the left- and right-invariant vector fields associated to \( x \), respectively.

\begin{defn}[Maurer--Cartan $1$-forms]\label{def:MC-abstract} \cite{Sharpe1997}
Let $G$ be a Lie group with identity $e$ and Lie algebra $\mathfrak g:=T_eG$.
The left and right Maurer--Cartan $1$-forms are the $\mathfrak g$-valued $1$-forms
$\theta^L,\theta^R\in\Omega^1(G;\mathfrak g)$ defined by
\[
\theta^L_g := (dL_{g^{-1}})_g : T_gG \longrightarrow T_eG\simeq\mathfrak g,
\qquad
\theta^R_g := (dR_{g^{-1}})_g : T_gG \longrightarrow T_eG\simeq\mathfrak g,
\]
for each $g\in G$, where $L_h(k)=hk$ and $R_h(k)=kh$ denote left/right translations.
Equivalently, for any smooth curve $\gamma$ in $G$ with $\gamma(0)=g$, $\gamma'(0)=v$,
\[
\theta^L_g(v)=\left.\frac{d}{dt}\right|_{0}\big(g^{-1}\gamma(t)\big),\qquad
\theta^R_g(v)=\left.\frac{d}{dt}\right|_{0}\big(\gamma(t)g^{-1}\big).
\]
\end{defn}

\begin{pro}[Matrix-group expressions]\label{prop:MC-matrix}
Suppose $G\subset \mathrm{GL}_n(\mathbb R)$ (or $\mathrm{GL}_n(\mathbb C)$) is a matrix Lie group.
Viewing $g$ as the matrix of coordinate functions on $G$ and $dg$ as the matrix of
$1$-forms $(dg_{ij})$, one has, as $\mathfrak g$-valued $1$-forms,
\[
 \theta^L = g^{-1}\,dg \quad\text{and}\quad \theta^R = dg\,g^{-1}.
\]
In particular, $\theta^R = \Ad_g(\theta^L)=g\,\theta^L\,g^{-1}$.
\end{pro}

\begin{proof}
Let $v\in T_gG$ and choose a smooth curve $\gamma:(-\varepsilon,\varepsilon)\to G$
with $\gamma(0)=g$ and $\gamma'(0)=v$. Then, in matrix notation,
\[
(g^{-1}dg)_g(v)=g^{-1}\gamma'(0)=\left.\frac{d}{dt}\right|_{0}\big(g^{-1}\gamma(t)\big)
=(dL_{g^{-1}})_g(v)=\theta^L_g(v),
\]
and similarly
\[
(dg\,g^{-1})_g(v)=\gamma'(0)g^{-1}=\left.\frac{d}{dt}\right|_{0}\big(\gamma(t)g^{-1}\big)
=(dR_{g^{-1}})_g(v)=\theta^R_g(v).
\]
Since this holds for every $g$ and $v\in T_gG$, we obtain the stated identities of
$\mathfrak g$-valued $1$-forms. Finally,
\begin{equation}\label{equation 2.5}
    \theta^R = dg\,g^{-1} = g\,(g^{-1}dg)\,g^{-1} = \Ad_g(\theta^L).
\end{equation}

\end{proof}

\begin{rem}[Immediate consequences]
From Proposition~\ref{prop:MC-matrix} one recovers the standard properties:
$(L_h)^*\theta^L=\theta^L$, $(R_h)^*\theta^R=\theta^R$, and the Maurer--Cartan equations
$d\theta^L+\tfrac12[\theta^L,\theta^L]=0$, $d\theta^R-\tfrac12[\theta^R,\theta^R]=0$,
where $[\ ,\ ]$ is the wedge–bracket induced by the Lie bracket on $\mathfrak g$.
\end{rem}

\begin{lem}[Contractions of Maurer--Cartan Forms]\label{interior product of cartan form}
Let $x \in \mathfrak{g}$, and define $ v_x = x^L - x^R$ as the vector field associated with the infinitesimal conjugation action. Then:
\begin{align*}
\iota_{ v_x} \theta^L &= (1 - \operatorname{Ad}_{g^{-1}})x, \\\\
\iota_{ v_x} \theta^R &= (\operatorname{Ad}_g - 1)x.
\end{align*}
\end{lem}

\begin{proof}
This follows directly from the definition $ v_x = x^L - x^R$ and the identity $\operatorname{Ad}_g \theta^L = \theta^R$ established in \eqref{equation 2.5}. Using the defining properties of left and right invariant forms, we compute:
\[
\iota_{ v_x} \theta^L = \iota_{x^L} \theta^L - \iota_{x^R} \theta^L = x - \operatorname{Ad}_{g^{-1}}x = (1 - \operatorname{Ad}_{g^{-1}}) x,
\]
\[
\iota_{ v_x} \theta^R = \iota_{x^L} \theta^R - \iota_{x^R} \theta^R = \operatorname{Ad}_g x - x = (\operatorname{Ad}_g - 1)x.
\]
\end{proof}

\begin{lem}[Bracket Evaluation via Maurer--Cartan Form]
\label{lemma:MC-bracket}
Let $X, Y$ be left-invariant vector fields on a Lie group $G$. Then the left-invariant Maurer--Cartan form $\theta^L$ satisfies:
\[
\theta^L([X, Y]) = [\theta^L(X), \theta^L(Y)].
\]
\end{lem}

\begin{proof}
Since $\theta^L$ is left-invariant and acts as the identity on the Lie algebra $\mathfrak{g}$ under evaluation at the identity, it preserves Lie brackets when applied to left-invariant vector fields. That is, for any $X = x^L$, $Y = y^L$ with $x, y \in \mathfrak{g}$, we have
\[
\theta^L(X) = x, \quad \theta^L(Y) = y, \quad \text{and} \quad [X, Y] = [x, y]^L,
\]
so
\[
\theta^L([X, Y]) = \theta^L([x, y]^L) = [x, y] = [\theta^L(X), \theta^L(Y)].
\]
\end{proof}

So, the bracket structure is preserved under the Maurer–Cartan form.
This confirms that the Maurer–Cartan form not only captures the Lie algebra structure but also faithfully preserves it under the Lie bracket operation when evaluated on invariant vector fields.

\begin{pro}[Maurer--Cartan Structure Equations]  \cite[Theorem 15.3]{Loring2020}  \label{mauercartan identity}
The Maurer--Cartan forms $\theta^L$ and $\theta^R$ satisfy the following differential identities:
\[
d\theta^L = -\frac{1}{2}[\theta^L, \theta^L], \quad \text{and} \quad d\theta^R = \frac{1}{2}[\theta^R, \theta^R].
\]
\end{pro}

\begin{proof}
Let $X, Y$ be left-invariant vector fields. Using the global formula for the exterior derivative and the left-invariance of $\theta^L$, we compute:
\[
d\theta^L(X, Y) = -\theta^L([X, Y]).
\]
By Lemma~\ref{lemma:MC-bracket}, we have:
\[
-\theta^L([X, Y]) = -[\theta^L(X), \theta^L(Y)].
\]
Now observe that the bracket $[\theta^L, \theta^L]$ satisfies:
\[
[\theta^L, \theta^L](X, Y) = [\theta^L(X), \theta^L(Y)] - [\theta^L(Y), \theta^L(X)] = 2[\theta^L(X), \theta^L(Y)],
\]
since the Lie bracket is antisymmetric. Thus,
\[
d\theta^L(X, Y) = -[\theta^L(X), \theta^L(Y)] = -\frac{1}{2}[\theta^L, \theta^L](X, Y),
\]
which proves the identity. The expression for $d\theta^R$ follows analogously, or by applying the identity $\theta^R = \operatorname{Ad}_g \theta^L$ and using the $\operatorname{Ad}$-equivariance of the Lie bracket.
\end{proof}

\begin{rem}[Coordinate Verification of the Maurer--Cartan Equation]
An alternative derivation of the Maurer--Cartan identities can be obtained via explicit matrix differentials. Observe:
\[
0 = d(g^{-1}g) = dg^{-1}g + g^{-1}dg \quad \implies \quad d(g^{-1}) = -g^{-1}dg\,g^{-1}.
\]
Using this, we compute:
\[
d(g^{-1}dg) = dg^{-1} \wedge dg + g^{-1}d^2g = -g^{-1}dg\,g^{-1}dg,
\]
\[
d(dg\,g^{-1}) = -dg\,g^{-1} \wedge dg\,g^{-1} = dg\,g^{-1}dg\,g^{-1},
\]
which matches the structure of $d\theta^L$ and $d\theta^R$ derived above. These coordinate computations provide a concrete check of the abstract differential identities.
\end{rem}

We now recall the Cartan 3-form using the Maurer–Cartan structure and some of the results needed in this work.

\begin{defn}[Cartan 3-Form]
Let \( G \) be a Lie group with Lie algebra \( \mathfrak{g} \), and let \( \cdot  \) denote an Ad-invariant inner product on \( \mathfrak{g} \). The \emph{Cartan 3-form} \( \eta \in \Omega^3(G) \) is defined by
\[
\eta = \frac{1}{12} \theta^L \cdot [\theta^L, \theta^L].
\]
where \( \theta^L \in \Omega^1(G, \mathfrak{g}) \) is the left-invariant Maurer–Cartan form on \( G \).
\end{defn}

\begin{pro}[Contraction of the Cartan 3-Form]\label{prop:cartan-contraction}
Let \( X \) be a vector field on \( G \). Then the contraction of the Cartan 3-form \( \eta \) satisfies
\[
\iota_X \eta = -\frac{1}{2} \iota_X \theta^L \cdot d\theta^L.
\]
\end{pro}

\begin{proof}
We begin by computing the contraction explicitly:
\[
\begin{aligned}
\iota_X \eta &= \frac{1}{12} \left( \iota_X \theta^L \cdot [\theta^L, \theta^L] - \theta^L \cdot \iota_X [\theta^L, \theta^L] \right) \\
&= \frac{1}{12} \left( \iota_X \theta^L \cdot [\theta^L, \theta^L] - \theta^L \cdot [\iota_X \theta^L, \theta^L] + \theta^L \cdot [\theta^L, \iota_X \theta^L] \right) \\
&= \frac{1}{12} \left( \iota_X \theta^L \cdot [\theta^L, \theta^L] + \theta^L \cdot [\theta^L, \iota_X \theta^L] - \theta^L \cdot [\iota_X \theta^L, \theta^L] \right) \\
&= \frac{1}{12} \left( \iota_X \theta^L \cdot [\theta^L, \theta^L] + 2 \theta^L \cdot [\theta^L, \iota_X \theta^L] \right) \\
&= \frac{1}{4} \iota_X \theta^L \cdot [\theta^L, \theta^L] = -\frac{1}{2} \iota_X \theta^L \cdot d\theta^L,
\end{aligned}
\]
where we have used the identity $d\theta^L = -\tfrac{1}{2}[\theta^L, \theta^L]$ from the Maurer--Cartan equation.
\end{proof}

This 3-form \( \eta \) is bi-invariant, meaning it is invariant under both left and right multiplication by elements of \( G \). Bi-invariance of \( \eta \) implies that it is closed.

\begin{thm}\cite[Section 9.10.1]{Guillemin1999} \label{eta is closed}
The 3-form \(\eta\) defined above is closed, i.e., \(d\eta = 0\).
\end{thm}

\begin{proof}
Let $\theta:=\theta^L\in\Omega^1(G;\mathfrak g)$ be the left Maurer--Cartan form and
\[
\eta \;=\; \frac{1}{12}\,\langle \theta,\,[\theta,\theta]\rangle \in \Omega^3(G)
\]
the Cartan $3$-form, where $[\ ,\ ]$ denotes the wedge–bracket of $\mathfrak g$-valued forms
and $\langle-,-\rangle$ is an $\Ad$-invariant inner product on $\mathfrak g$.
We use the following standard identities:
\begin{enumerate}
\item (Maurer--Cartan) \quad $d\theta+\tfrac12[\theta,\theta]=0$.
\item (Graded derivation) \quad $d[\alpha,\beta]=[d\alpha,\beta]+(-1)^{|\alpha|}[\alpha,d\beta]$.
\item (Graded invariance/cyclicity) \quad
$\langle\alpha,[\beta,\gamma]\rangle
= (-1)^{|\alpha||\beta|}\langle[\alpha,\beta],\gamma\rangle$,
for homogeneous $\mathfrak g$-valued forms $\alpha,\beta,\gamma$.
\end{enumerate}
Since $|\theta|=1$, we compute
\[
d\eta
= \frac{1}{12}\Big(\langle d\theta,\,[\theta,\theta]\rangle
- \langle \theta,\,d[\theta,\theta]\rangle\Big)
= \frac{1}{12}\Big(\langle d\theta,\,[\theta,\theta]\rangle
- \langle \theta,\,[d\theta,\theta]-[\theta,d\theta]\rangle\Big),
\]
where we used (2) in the second step.
Applying (3) with the indicated degrees gives
\[
\langle \theta,\,[d\theta,\theta]\rangle=\langle[\theta,d\theta],\theta\rangle,
\qquad
\langle \theta,\,[\theta,d\theta]\rangle=-\,\langle[\theta,\theta],d\theta\rangle.
\]
Moreover, by (3) and graded skew-symmetry,
\[
\langle d\theta,\,[\theta,\theta]\rangle
= \langle [d\theta,\theta],\theta\rangle
= -\,\langle[\theta,d\theta],\theta\rangle.
\]
Substituting these into the expression for $d\eta$ yields the cancellation
\[
d\eta
= \frac{1}{12}\Big(-\,\langle[\theta,d\theta],\theta\rangle
- \langle[\theta,d\theta],\theta\rangle
+ \langle[\theta,\theta],d\theta\rangle\Big)
= -\,\frac{1}{12}\,\langle[\theta,\theta],d\theta\rangle.
\]
Using the Maurer--Cartan equation (1), $d\theta=-\tfrac12[\theta,\theta]$, we obtain
\[
d\eta
= \frac{1}{24}\,\langle[\theta,\theta],\,[\theta,\theta]\rangle.
\]
Finally, apply (3) once more with $\alpha=\theta$, $\beta=\theta$, $\gamma=[\theta,\theta]$:
since $|\theta|=1$,
\[
\langle[\theta,\theta],\,[\theta,\theta]\rangle
= (-1)^{|\theta||\theta|}\,\langle \theta,\,[\theta,[\theta,\theta]]\rangle
= -\,\langle \theta,\,[\theta,[\theta,\theta]]\rangle.
\]
The graded Jacobi identity implies $[\theta,[\theta,\theta]]=0$, hence
$\langle[\theta,\theta],[\theta,\theta]\rangle=0$ and therefore $d\eta=0$.
\end{proof}

\begin{pro}[Nondegeneracy and the center] \label{prop:2plectic-iff-center0}
Let \(\eta\) be the Cartan \(3\)-form above. Then for any \(X\in\mathfrak g\),
\[
\iota_{X^L}\eta=0 \quad\Longleftrightarrow\quad X\in\mathfrak z(\mathfrak g),
\]
where \(X^L\) is the left-invariant vector field generated by \(X\).
Consequently, \(\eta\) is \(2\)-plectic (i.e.\ \(\iota_v\eta=0\Rightarrow v=0\))
if and only if the center \(\mathfrak z(\mathfrak g)\) is trivial.
\end{pro}

\begin{proof}
By left-invariance, it suffices to test at the identity \(e\in G\).
Identifying \(T_eG\simeq\mathfrak g\), we have
\[
\eta_e(X,Y,Z)=\tfrac12\,\langle X,[Y,Z]\rangle.
\]
Fix \(X\in\mathfrak g\). Then \(\iota_{X^L}\eta=0\) iff
\(\eta_e(X,Y,Z)=0\) for all \(Y,Z\in\mathfrak g\), i.e.
\[
\langle X,[Y,Z]\rangle=0\quad\forall\,Y,Z.
\]
Using \(\mathrm{Ad}\)-invariance,
\(\langle X,[Y,Z]\rangle=\langle [X,Y],Z\rangle\).
Since \(\langle\cdot,\cdot\rangle\) is nondegenerate, this holds for all \(Z\)
iff \([X,Y]=0\) for all \(Y\), i.e.\ \(X\in\mathfrak z(\mathfrak g)\).
The converse is immediate: if \(X\in\mathfrak z(\mathfrak g)\),
then \([Y,Z]\) is orthogonal to \(X\) for all \(Y,Z\), so \(\iota_{X^L}\eta=0\).
\end{proof}

Next, we extend the Cartan 3-form in an equivariant way.

\begin{pro}[Closed Equivariant Extension of the Cartan 3-Form]
\label{cor:closed-equivariant-extension}
Let $x \in \mathfrak{g}$ and let $ v_x$ denote the infinitesimal vector field on $G$ generated by conjugation: $ v_x = x^L - x^R$. Define the equivariant extension of $\eta$ by
\[
\eta_G(x) := \eta - \frac{1}{2}(\theta^L + \theta^R) \cdot x.
\]
Then $\eta_G$ is equivariantly closed:
\[
(d_G \eta_G)(x) := (d - \iota_{ v_x}) \eta_G(x) = 0.
\]
\end{pro}

\begin{proof}
Using the previous proposition and the identity $\iota_{v_x} \theta^L = (1 - \operatorname{Ad}_{g^{-1}}) x$, we have:
\[
\begin{aligned}
\iota_{ v_x} \eta &= -\frac{1}{2} \iota_{ v_x} \theta^L \cdot d\theta^L \\
&= -\frac{1}{2}(1 - \operatorname{Ad}_{g^{-1}}) x \cdot d\theta^L \\
&= -\frac{1}{2} d\theta^L \cdot x + \frac{1}{2} \operatorname{Ad}_g d\theta^L \cdot x.
\end{aligned}
\]
Now, using the coordinate identities:
\[
\operatorname{Ad}_g d(g^{-1} dg) = - d(dg\, g^{-1}) \quad \implies \quad \operatorname{Ad}_g d\theta^L = -d\theta^R.
\]
It follows that
\[
\iota_{ v_x} \eta = -\frac{1}{2} d\theta^L \cdot \xi - \frac{1}{2} d\theta^R \cdot x = -d\left( \frac{\theta^L + \theta^R}{2} \right) \cdot x.
\]
Putting this into the definition of $d_G$, we get:
\[
\begin{aligned}
(d_G \eta_G)(x) &= d\eta - \frac{1}{2}d(\theta^L + \theta^R) \cdot x - \iota_{ v_x} \eta + \iota_{ v_x} \left( \frac{1}{2}(\theta^L + \theta^R) \cdot x \right) \\
&= 0 + 0 + d\left( \frac{\theta^L + \theta^R}{2} \right) \cdot x + \frac{1}{2}(\theta^L + \theta^R)( v_x) \cdot x \\
&= \frac{1}{2}(\theta^L + \theta^R)(x^L - x^R) \cdot x \\
&= \frac{1}{2}(\operatorname{Ad}_g - \operatorname{Ad}_{g^{-1}}) x \cdot x = 0,
\end{aligned}
\]
since $\operatorname{Ad}_g - \operatorname{Ad}_{g^{-1}}$ is skew-symmetric with respect to the invariant inner product, making the pairing vanish.
\end{proof}

\begin{pro}\label{key cartan identity maurer form}
Let $\eta = \frac{1}{12} \theta^L \cdot [\theta^L, \theta^L]$ denote the Cartan 3-form on a Lie group $G$ equipped with an $\operatorname{Ad}$-invariant inner product on its Lie algebra $\mathfrak{g}$. Then for any $x, y \in \mathfrak{g}$, the following identity of $1$-forms holds:
\[
\frac{1}{2}(\theta^L + \theta^R) \cdot [x, y] + \iota_{ v_x} \iota_{ v_y} \eta = \frac{1}{2} d \iota_{ v_x} (\theta^L + \theta^R) \cdot y.
\]
\end{pro}

\begin{proof}
As usual, to verify the equality of differential 1-forms, it suffices to test both sides on a spanning set of vector fields. We evaluate both sides on an arbitrary left-invariant vector field $Z = z^L$ with $z \in \mathfrak{g}$, and define $X = x$, $Y = y$. We have:
\[
\resizebox{\textwidth}{!}{$
\begin{aligned}
\frac{1}{2} \iota_Z (\theta^L + \theta^R) \cdot [X, Y] + \iota_Z \iota_{v_x} \iota_{v_y} \eta 
&= \frac{1}{2}(1 + \operatorname{Ad}_g)Z \cdot [X, Y] + \frac{1}{2} \iota_Y \theta^L \cdot [\iota_X \theta^L, \iota_Z \theta^L] \\
&= \frac{1}{2} \left( Z \cdot [X, Y] + \operatorname{Ad}_g Z \cdot [X, Y] + (1 - \operatorname{Ad}_{g^{-1}})Y \cdot \left[ (1 - \operatorname{Ad}_{g^{-1}})X, Z \right] \right) \\
&= \frac{1}{2} \left( [Z, X] \cdot \operatorname{Ad}_{g^{-1}} Y + \operatorname{Ad}_{g^{-1}} X \cdot [Y, Z] \right).
\end{aligned}
$}
\]

 Recall that $\iota_{ v_x}(\theta^L + \theta^R) = (\operatorname{Ad}_g - \operatorname{Ad}_{g^{-1}}) X$, hence:
\[
\begin{aligned}
\frac{1}{2} d \iota_{ v_x} (\theta^L + \theta^R) \cdot Y &= \frac{1}{2} d \left( (\operatorname{Ad}_g - \operatorname{Ad}_{g^{-1}}) X \cdot Y \right) \\
&= \frac{1}{2} d \left( \operatorname{Ad}_g X \cdot Y - X \cdot \operatorname{Ad}_g Y \right).
\end{aligned}
\]
We have:
\[
\begin{aligned}
\left. \frac{d}{dt} \right|_{t=0} \operatorname{Ad}_{g \exp(tZ)} X &= \operatorname{Ad}_g [Z, X], \\
\left. \frac{d}{dt} \right|_{t=0} \operatorname{Ad}_{g \exp(tZ)} Y &= \operatorname{Ad}_g [Z, Y].
\end{aligned}
\]
Hence
\[
\frac{1}{2} d \iota_{ v_x} (\theta^L + \theta^R) \cdot Y=\frac{1}{2} \left( \operatorname{Ad}_g [Z, X] \cdot Y - X \cdot \operatorname{Ad}_g [Z, Y] \right).
\]

 Finally, using the $\operatorname{Ad}$-invariance of the inner product,
\[
\begin{aligned}
\operatorname{Ad}_g [Z, X] \cdot Y &= [Z, X] \cdot \operatorname{Ad}_{g^{-1}} Y, \\
X \cdot \operatorname{Ad}_g [Z, Y] &= \operatorname{Ad}_{g^{-1}} X \cdot [Y, Z],
\end{aligned}
\]
it follows that
\begin{align*}
\frac{1}{2} \iota_Z (\theta^L + \theta^R) \cdot [X, Y] + \iota_Z \iota_{ v_x} \iota_{ v_y} \eta &= \frac{1}{2} \left( [Z, X] \cdot \operatorname{Ad}_{g^{-1}} Y + \operatorname{Ad}_{g^{-1}} X \cdot [Y, Z] \right)\\
&= \frac{1}{2} \left( \operatorname{Ad}_g [Z, X] \cdot Y - X \cdot \operatorname{Ad}_g [Z, Y] \right)\\
&= \frac{1}{2} d \iota_{ v_x} (\theta^L + \theta^R) \cdot Y.
\end{align*}
This completes the proof.
\end{proof}

\begin{lem}\label{Lie derivative of f2}
Let \( G \) be a Lie group with Lie algebra \( \mathfrak{g} \), equipped with an invariant inner product \( \cdot \). Define
\[
f_2(y, x)(g) := \frac{1}{2} \left( (\operatorname{Ad}_g - \operatorname{Ad}_{g^{-1}})(y) \cdot x \right)
\]
for \( x, y \in \mathfrak{g} \) and \( g \in G \). Let \( v_z \) be the fundamental vector field associated to the conjugation action by \( z \in \mathfrak{g} \). Then the Lie derivative of \( f_2(y, x) \in C^\infty(G) \) along \( v_z \) satisfies:
\[
\mathcal{L}_{v_z} f_2(y, x) = f_2([z, y], x) - f_2(y, [z, x]).
\]
\end{lem}

\begin{proof}
We compute the Lie derivative using the flow \( \phi_t^z(g) = \exp(tz) g \exp(-tz) \) of the conjugation action:
\[
\mathcal{L}_{v_z} f_2(y, x)(g) = \left. \frac{d}{dt} \right|_{t=0} f_2(y, x)(\phi_t^z(g)).
\]
Substitute the definition of \( f_2 \):
\[
\mathcal{L}_{v_z} f_2(y, x)(g) = \left. \frac{d}{dt} \right|_{t=0} \frac{1}{2} \left( (\operatorname{Ad}_{\phi_t^z(g)} - \operatorname{Ad}_{\phi_t^z(g)^{-1}})(y) \cdot x \right).
\]
Let \( h(t) := \phi_t^z(g) = \exp(tz) g \exp(-tz) \). Then:
\[
\resizebox{\textwidth}{!}{$
\begin{aligned}
\left. \frac{d}{dt} \right|_{t=0} \operatorname{Ad}_{h(t)}(y) &= \left. \frac{d}{dt} \right|_{t=0} h(t) y h(t)^{-1} \\
&= \left.\left( \frac{d}{dt} h(t) \right) y h(t)^{-1} \right|_{t=0} + \left. h(t) y \left( \frac{d}{dt} h(t)^{-1} \right) \right|_{t=0} \\
&= \left.\left( \frac{d}{dt} \exp(tz) g \exp(-tz) \right) y h(t)^{-1} \right|_{t=0} + \left. h(t) y \left( \frac{d}{dt} \exp(tz) g^{-1} \exp(-tz) \right) \right|_{t=0} \\
&= \left( z \exp(tz) g \exp(-tz) - \exp(tz) g z \exp(-tz) \right) y h(t)^{-1} \Big|_{t=0} \\
&\quad + h(t) y \left( z \exp(tz) g^{-1} \exp(-tz) - \exp(tz) g^{-1} z \exp(-tz) \right) \Big|_{t=0} \\
&= (zg - gz) y g^{-1} + g y (z g^{-1} - g^{-1} z) \\
&= (z g y g^{-1} - g z y g^{-1}) + (g y z g^{-1} - g y g^{-1} z) \\
&= (z \operatorname{Ad}_g(y) - \operatorname{Ad}_g(y) z) + g (y z - z y) g^{-1} \\
&= [z, \operatorname{Ad}_g(y)] + \operatorname{Ad}_g([y, z]).
\end{aligned}
$}
\]

Similarly,
\begin{align*}
    \left. \frac{d}{dt} \right|_{t=0} \operatorname{Ad}_{h(t)^{-1}}(y)
         &= [z, \operatorname{Ad}_{g^{-1}}(y)]+ \operatorname{Ad}_{g^{-1}}([y,z]).
\end{align*}

It follows that
\begin{align*}
    \mathcal{L}_{v_z} f_2(y, x)(g) &= \left. \frac{d}{dt} \right|_{t=0} \frac{1}{2} \left( (\operatorname{Ad}_{\phi_t^z(g)} - \operatorname{Ad}_{\phi_t^z(g)^{-1}})(y) \cdot x \right)\\
    &=\frac{1}{2}\left( [z, \operatorname{Ad}_g(y)]+ \operatorname{Ad}_g([y,z])-[z, \operatorname{Ad}_{g^{-1}}(y)]- \operatorname{Ad}_{g^{-1}}([y,z])\right)\cdot x
\end{align*}

Using invariance of the inner product, i.e., \( [a, b] \cdot c = -b \cdot [a, c] \), we rewrite:
\begin{align*}
[z, \operatorname{Ad}_g(y)] \cdot x &= -\operatorname{Ad}_g(y) \cdot [z, x], \\
[z, \operatorname{Ad}_{g^{-1}}(y)] \cdot x &= -\operatorname{Ad}_{g^{-1}}(y) \cdot [z, x],
\end{align*}
so that:
\begin{align*}
    \mathcal{L}_{v_z} f_2(y, x)(g) &= \frac{1}{2}\left( -\operatorname{Ad}_g(y) \cdot [z, x]+ \operatorname{Ad}_g([y,z])\cdot x +\operatorname{Ad}_{g^{-1}}(y) \cdot [z, x]- \operatorname{Ad}_{g^{-1}}([y,z])\cdot x\right)\\
    &= \frac{1}{2}\left( -\operatorname{Ad}_g(y) + \operatorname{Ad}_{g^{-1}}(y) \right)\cdot [z, x]+ \frac{1}{2}\left(  \operatorname{Ad}_g([y,z]) - \operatorname{Ad}_{g^{-1}}([y,z]) \right)\cdot x
\end{align*}

On the other hand,
\begin{align*}
f_2([z, y], x)(g) &= \frac{1}{2} \left( (\operatorname{Ad}_g - \operatorname{Ad}_{g^{-1}})([z, y]) \cdot x \right), \\
f_2(y, [z, x])(g) &= \frac{1}{2} \left( (\operatorname{Ad}_g - \operatorname{Ad}_{g^{-1}})(y) \cdot [z, x] \right).
\end{align*}

Therefore,
\begin{align*}
    \mathcal{L}_{v_z} f_2(y, x)(g)&=f_2([z, y], x)(g) - f_2(y, [z, x])(g)\\
    &=f_2([z, y], x)(g) + f_2(y, [x, z])(g).
\end{align*}

This completes the proof.
\end{proof}

We end this section with a Jacobi identity involving the Cartan 3-form.


Let \( G \) be a Lie group with Lie algebra \( \mathfrak{g} \), and let \( \eta \in \Omega^3(G) \) denote the Cartan 3-form.
Assume that \( G \) acts on itself by conjugation. The corresponding fundamental vector field for \( x \in \mathfrak{g} \) is
\[
v_x = x^L - x^R.
\]

We define the following auxiliary functions:
\begin{align*}
f_1(x) &= \frac{1}{2} \left( (\theta^L + \theta^R) \cdot x \right), \\
f_2(x, y) &= \iota_{v_x} f_1(y) = \frac{1}{2} \left( (\mathrm{Ad}_g - \mathrm{Ad}_{g^{-1}})(x) \cdot y \right), \\
f_2(x, y) &= \frac{1}{2} \langle (\mathrm{Ad}_g - \mathrm{Ad}_{g^{-1}})(x), y \rangle.
\end{align*}

\begin{lem}[Jacobiator identity for the Cartan 3-form]\label{triple contraction of cartan form}
For all \( x, y, z \in \mathfrak{g} \), the Cartan 3-form satisfies:
\[
\eta(v_x, v_y, v_z) = f_2(x, [y, z]) + f_2(y, [z, x]) + f_2(z, [x, y]).
\]
\end{lem}

\begin{proof}
By Proposition \ref{key cartan identity maurer form}, we have
\[
\iota_{v_y} \iota_{v_x} \eta = d f_2(y, x) - f_1([y, x]),
\]
so that
\[
\iota_{v_z} \iota_{v_y} \iota_{v_x} \eta = \iota_{v_z} d f_2(y, x) - \iota_{v_z} f_1([y, x]).
\]
Since \( f_2(y, x) \in C^\infty(G) \) and \( \iota_{v_z} \) acts trivially on functions:
\[
\iota_{v_z} d f_2(y, x) = \mathcal{L}_{v_z} f_2(y, x), \qquad \iota_{v_z} f_1([y, x]) = f_2([y, x], z).
\]
So:
\[
\iota_{v_z} \iota_{v_y} \iota_{v_x} \eta = \mathcal{L}_{v_z} f_2(y, x) - f_2([y, x], z).
\]

By Lemma~\ref{Lie derivative of f2}, 
\[
\mathcal{L}_{v_z} f_2(y, x) = f_2([z, y], x) - f_2(y, [z, x]).
\]

It follows that 
\[
\iota_{v_z} \iota_{v_y} \iota_{v_x} \eta = f_2([z, y], x) - f_2(y, [z, x]) - f_2([y, x], z).
\]

Finally, using antisymmetry of \( f_2 \) and the Lie bracket, we obtain:
\[
\iota_{v_z} \iota_{v_y} \iota_{v_x} \eta = f_2(x, [y, z]) + f_2(y, [x, z]) - f_2(z, [x, y]).
\]
\end{proof}

In the last section of this chapter, we turn to the theory of quasi-Hamiltonian \( G \)-spaces, where the Cartan 3-form plays a central role in defining group-valued moment maps.

\section{A brief review of quasi-Hamiltonian \( G \)-spaces}
\label{Group valued moment maps}

Throughout this exposition, we consider a compact Lie group \( G \). 
The concept of group-valued moment maps, introduced in \cite{10.4310/jdg/1214460860}, provides a framework for this generalization. Unlike traditional moment maps, group-valued moment maps accommodate settings where the symplectic structure is replaced by a weaker 2-form and where the moment map targets the Lie group itself. This leads to the following definition.

\begin{defn}[\cite{10.4310/jdg/1214460860}]
\label{def:quasiHamiltonianGSpace}
A \emph{quasi-Hamiltonian \(G\)-space} consists of a triple \((M, \omega, \Phi)\) where:
\begin{itemize}
    \item \(M\) is a \(G\)-manifold,
    \item \(\omega \in \Omega^2(M)^G\) is a \(G\)-invariant 2-form, and
    \item \(\Phi: M \to G\) is a \(G\)-equivariant map.
\end{itemize}
These data must satisfy the following conditions:
\begin{enumerate}
\item \(d\omega = - \Phi^* \eta\), 
\item \(\iota_{v_x}\omega = \frac{1}{2} \Phi^*\left( \theta^L + \theta^R\right)\cdot x  \quad \text{for all } x \in \mathfrak{g}\), and
\item at every point \(m\in M\), \( \ker \omega_m \cap \ker (d\Phi_m )= \{ 0 \} \).
\end{enumerate}

In this case, the map \(\Phi\) is called a \emph{Lie group valued moment map}.
\end{defn}

We now illustrate this definition with concrete examples, beginning with one of the most fundamental: conjugacy classes.

\begin{exa}[Conjugacy Classes as Quasi-Hamiltonian Spaces]
Let \( G \) be a compact Lie group with Lie algebra \( \mathfrak{g} \), and let \( \mathcal{C} \subset G \) be a conjugacy class. Define the moment map \( \Phi: \mathcal{C} \hookrightarrow G \) to be the inclusion map, and let the 2-form \( \omega \in \Omega^2(\mathcal{C}) \) be given at \( f \in \mathcal{C} \) by:
\[
\omega_f(v_x, v_y) = \frac{1}{2} \left( (y, \operatorname{Ad}_f x) - (x, \operatorname{Ad}_f y) \right),
\]
where \( v_x \) is the fundamental vector field associated to \( x \in \mathfrak{g} \), and \( (\cdot, \cdot) \) is an \( \operatorname{Ad} \)-invariant inner product on \( \mathfrak{g} \).

This 2-form satisfies the axioms of a quasi-Hamiltonian structure: it is \( G \)-invariant, the differential satisfies \( d\omega = -\Phi^* \eta \), and the vector fields \( v_x \) span the kernel of \( d\Phi \) transversally to the kernel of \( \omega \). See \cite[Proposition 3.1]{10.4310/jdg/1214460860} for a detailed proof.
\end{exa}

Next, we consider a higher-dimensional example that plays a central role in the fusion and reduction procedures: the double of the group.

\begin{exa}[The Double \( D(G) \)]
Let \( G \) be a compact Lie group equipped with an \( \operatorname{Ad} \)-invariant inner product on its Lie algebra \( \mathfrak{g} \). Define the \emph{double} as the triple \( (D(G), \omega, \Phi) \), where:
\begin{itemize}
    \item \( D(G) := G \times G \), with the diagonal conjugation action:
    \[
    g \cdot (a, b) := (gag^{-1}, gbg^{-1}),
    \]
    \item The moment map is defined as:
    \[
    \Phi(a, b) := aba^{-1}b^{-1},
    \]
    \item The 2-form is given by:
    \[
    \omega := \frac{1}{2} \left( \langle a^* \theta^L, b^* \theta^L \rangle - \langle a^* \theta^R, b^* \theta^R \rangle \right),
    \]
    where \( \theta^L, \theta^R \) are the Maurer–Cartan forms on \( G \), and \( \langle \cdot, \cdot \rangle \) is the invariant inner product on \( \mathfrak{g} \).
\end{itemize}

This example satisfies all the conditions in Definition~\ref{def:quasiHamiltonianGSpace}. For a full verification of its quasi-Hamiltonian properties, see \cite[Section 3.2]{10.4310/jdg/1214460860}.
\end{exa}

\vspace{1em}

In this chapter, we reviewed foundational tools from differential geometry and Lie theory—complexes, homology, Cartan calculus, the Maurer–Cartan form, and the theory of quasi-Hamiltonian \( G \)-spaces. We now turn our attention to multisymplectic geometry.

%% file: Chapter3.tex
\chapter{Basics on Multisymplectic Geometry} \label{ch1}

A manifold equipped with a closed, nondegenerate \( (n+1) \)-form is called an \emph{\( n \)-plectic manifold}, and when \( n = 1 \), we recover the familiar symplectic case.
Multisymplectic geometry extends the ideas of symplectic geometry to higher-degree differential forms. At its core, it provides a natural geometric setting for describing classical field theories in multiple dimensions, in the same way symplectic geometry underpins the mechanics of point particles. This connection is explored in detail in \cite{Ryvkin_2019, Marsden1998, Forger2005}, where multisymplectic structures are shown to play a central role in the geometric formulation of field theories. 

The transition from symplectic to multisymplectic geometry is more than just a shift in form degree—it brings with it profound structural consequences. In particular, the failure of the nondegenerate \( (n+1) \)-form to define an isomorphism, as it does in the symplectic case, introduces new algebraic structures such as \( L_\infty \)-algebras. 

The primary aim of this chapter is to introduce the key notions and structures of multisymplectic geometry that will serve as the foundation for our later developments. We begin with the definition of \( n \)-plectic structures and explore how Hamiltonian forms and vector fields generalize to this setting. We then review the language of graded vector spaces and \( L_\infty \)-algebras, which arise naturally from the failure of the classical Poisson bracket to extend to higher-degree forms. Special attention will be given to the case of Lie 2-algebras.

\section{Multisymplectic manifolds}\label{Ssec:MultiSymBG}

The aim of this section is to give an overview of multisymplectic manifolds. Recall that a symplectic manifold is a pair $(M, \omega)$, where $\omega$ is a closed nondegenerate 2-form. In this case we refer to $\omega$ as the symplectic structure.   In general, multisymplectic manifolds are smooth manifolds equipped with a
closed, nondegenerate differential form of higher degree.

\begin{defn}[\cite{ baez2004higher, Cantrijn1998}]\label{def:MultisymplecticManifold}
	A \emph{pre-$n$-plectic structure} on $M$ is a closed $(n+1)$-form $\omega\in\Omega^{n+1}(M)$. A pre-$n$-plectic manifold is a pair  $(M, \omega)$, where $M$ is a smooth manifold and $\omega$ is a \emph{pre-$n$-plectic structure} on $M$. 
 \end{defn}

 \begin{defn}
     An $(n+1)$-form $\omega\in\Omega^{n+1}(M)$ is  nondegenerate if 
 \begin{align*}
     \omega^\flat:
		{TM}&\to {\Lambda^nT^*M}\\
		{v}&\mapsto {\iota_v\omega}
 \end{align*}
 is injective. That is, for all $x\in M$, for all $v \in T_xM$, $\iota_v\omega=0 \implies v=0$.

\end{defn}

\begin{defn}[\cite{baez2004higher, Cantrijn1998}]\label{def:MultisymplecticManifold}
	An \emph{$n$-plectic (or multisymplectic) structure} on $M$ is a closed nondegenerate  $(n+1)$-form $\omega\in\Omega^{n+1}(M)$. An $n$-plectic manifold is a  pair $(M, \omega)$, where 
 \begin{itemize}
     \item $M$ is a smooth manifold, and
     \item  $\omega$ is an \emph{$n$-plectic structure} (i.e; a closed and nondegenerate $(n+1)$-form) on $M$.
 \end{itemize}

 \end{defn}

\begin{defn}\label{multisymplectic map}
	Let $(M,\omega)$ and $(M',\omega')$ be $n$-plectic manifolds. A \emph{multisymplectic map} from $(M,\omega)$ to $(M',\omega')$ is a smooth map $\varphi: M \to M'$ satisfying $\varphi^* \omega' = \omega$.
 
\end{defn}

\begin{rem}
    The collection of all $n$-plectic manifolds along with multisymplectic maps as defined in Definition \ref{multisymplectic map} form a category. 
\end{rem}

\begin{exa}
The special case where \( n = 1 \) corresponds to \emph{symplectic manifolds}. 
\end{exa}

\begin{exa}
Let \( G \) be a compact Lie group, and let \( \mathfrak{g} \) be its corresponding Lie algebra. It is known that \( \mathfrak{g} \) admits an inner product \( \cdot \) that is invariant under the adjoint representation \( \mathrm{Ad}: G \to \mathrm{Aut}(\mathfrak{g}) \). The 3-form
\[
\eta = \theta^L \cdot [\theta^L, \theta^L]
\]
is bi-invariant, and hence closed (see Proposition \ref{eta is closed}). For any \( k \in \mathbb{R} \setminus \{0\} \), we define a one-parameter family of 3-forms by the formula
\[
\omega_k(x,y,z) = k(x \cdot [y,z]) \quad \text{for all } x,y,z \in \mathfrak{g}.
\]

\noindent Using the Maurer–Cartan form \( \theta^L \), the 3-form \( \omega_k \) gives rise to a one-parameter family of left-invariant 3-forms \( \nu_k \) on \( G \) by setting
\[
\nu_k = \omega_k(\theta^L, \theta^L, \theta^L) = k \left( \theta^L \cdot [\theta^L, \theta^L] \right).
\]
In the following, we will prove that \( (G, \nu_k) \) is a 2-plectic manifold.
\end{exa}

\begin{pro}[{\cite[Proposition~2.9]{rogers2011higher}}]\label{nondegenerate eta}
If \( G \) is a compact, simple Lie group, then \( (G, \nu_k) \) is a 2-plectic manifold.
\end{pro}

\begin{proof}
By Theorem~\ref{eta is closed}, the 3-form \( \theta^L \cdot [\theta^L, \theta^L] \) is closed:
\[
\mathrm{d} \left( \theta^L \cdot [\theta^L, \theta^L] \right) = 0.
\]
It follows that
\[
\mathrm{d} \left( k \, \theta^L \cdot [\theta^L, \theta^L] \right) = k \, \mathrm{d} \left( \theta^L \cdot [\theta^L, \theta^L] \right) = 0,
\]
so \( \nu_k \) is also closed for any \( k \in \mathbb{R} \setminus \{0\} \).

It remains to prove that \( \nu_k \) is non-degenerate. Suppose \( x \in \mathfrak{g} \) is such that \( \omega_k(x, y, z) = 0 \) for all \( y, z \in \mathfrak{g} \). Since \(G\) is simple, then 
Then the center of $\mathfrak g$ is trivial:
\[
\mathfrak z(\mathfrak g)\;=\;\{x\in\mathfrak g:\ [x,y]=0\ \ \forall\,y\in\mathfrak g\}\;=\;\{0\}.
\]
It follows from Proposition \ref{prop:2plectic-iff-center0} that \(\nu_k\) is nondegenerate.

\end{proof}

\begin{exa}\cite[\S 6]{Cantrijn1998}\label{Multicotangent}
    For any smooth manifold $Q$, the associated \emph{multicotangent bundle}, denoted $M = \Lambda^n T^\ast Q$, naturally possesses an $n$-plectic structure. To see this, let $\pi_Q$ be the projection $\Lambda^n T^\ast Q \twoheadrightarrow Q$ and $\pi_M$  the projection $\Lambda^n T^\ast M \twoheadrightarrow M$.

    Define the \emph{tautological $n$-form} $\theta \in \Omega^n(M)$ as  the unique section of $\Gamma(\Lambda^n T^\ast M, M)$ that makes  the following diagram commute for any section $\alpha \in \Gamma(\Lambda^n T^\ast Q, Q)$:
    \begin{center}
    \begin{tikzcd}
        M = \Lambda^n T^\ast Q \arrow{r}{\theta} \arrow[r]{d}{\theta} & \Lambda^n T^\ast M \arrow{d}{\pi_M} \\
        Q \arrow{r}{\alpha} \arrow[swap]{u}{\alpha} & \Lambda^n T^\ast Q
    \end{tikzcd}
    \end{center}
    where $\sigma$ maps $Q$ to $M$ by pulling back along $\alpha$. This diagram commutes, implying $\theta$ pulled back by $\alpha$ equals $\alpha$. Thus $\alpha^* \theta = \alpha$. This means for any point $q \in Q$ and cotangent vector $\eta \in \Lambda^n T^\ast_q Q$, along with vectors $u_1, \dots, u_n \in T_{(q,\eta)}M$, we have:
    \[
    \theta_{(q,\eta)}(u_1, \dots, u_n) = \eta((\pi_Q)_* u_1, \dots, (\pi_Q)_* u_n),
    \]
    which can be equivalently expressed as:
    \[
    \theta|_{(q,\eta)} = (\pi_Q)^* \eta.
    \]

\noindent
Locally, if \( (q^1, \ldots, q^k) \) are coordinates on \( Q \) and \( p_I \) (for multi-indices \( I = (i_1, \ldots, i_n) \), with \( 1 \leq i_1 < \cdots < i_n \leq k \)) are the corresponding fiber coordinates, the tautological form can be expressed as:
\[
\theta_{|(q,p)} = \sum_{I} p_I \, dq^I,
\]
\noindent
where \( I = (i_1, \ldots, i_n) \) with \( 1 \leq i_1 < \cdots < i_n \leq k \), and \( dq^I := dq^{i_1} \wedge \cdots \wedge dq^{i_n} \).

    The canonical $n$-plectic form, $\omega$, is the $(n+1)$-form defined by:
    \[
    \omega := -d\theta,
    \]
    expressed in local coordinates as:
    \[
    \omega|_{(q, p)} = -\sum_{I} dp_I \wedge dq^I.
    \]
    This form is closed and nondegenerate, so that $(M, \omega)$ is an $n$-plectic manifold. 
\end{exa}

\begin{defn}\cite[Definition 2.1]{RYVKIN_2018-conserved}
Let \( (M, \omega) \) be an \(n\)-plectic manifold. We define the following:

\begin{itemize}

    \item An action of a Lie group \( G \) on \( M \) is a \emph{multisymplectic action} if every group element acts by multisymplectomorphisms. That is, for each \( g \in G \) and the induced map \( \mathcal{A}_g: M \to M \), it holds that
    \[
    \mathcal{A}_g^* \omega = \omega.
    \]

    \item A vector field \( v \in \mathfrak{X}(M) \)  is a \emph{multisymplectic vector field} or an \emph{infinitesimal symmetry} of \( (M, \omega) \) if it satisfies
    \[
    \mathcal{L}_v \omega = 0,
    \]
    where \( \mathcal{L}_v \) denotes the Lie derivative with respect to \( v \).

\end{itemize}
\end{defn}

\section{Multisymplectic vs. Symplectic Geometry}

Multisymplectic geometry, extending the principles of symplectic geometry, exhibits unique complexities due to its general structure and higher-degree forms. The first main difference between symplectic geometry and multisymplectic geometry arises in the Darboux Theorem, a well-known result in symplectic geometry, which states as follows.
\begin{thm}[Darboux's Theorem]\label{thm:darboux}\cite[Theorem 8.1]{dasilva2005symplectic}
Let \((M, \omega)\) be a symplectic manifold, and let \(p\) be any point in \(M\). Then there exists a local chart \((\mathcal{U}, x_1, \ldots, x_n, y_1, \ldots, y_n)\) centered at \(p\) such that:
\[
\omega = \sum_{i=1}^n dx_i \wedge dy_i.
\]
This local representation, often referred to as \emph{Darboux coordinates}, illustrates that locally, the symplectic structure \(\omega\) can always be simplified to a standard form, reflecting the intrinsic geometric properties of symplectic manifolds that are locally indistinguishable from \(\mathbb{R}^{2n}\) equipped with its standard symplectic form.
\end{thm}

Here are some keys differences.

 \subsection*{Differences in Dimensionality and Local Structure}

\paragraph{Absence of Even Dimensionality:}
    \emph{Symplectic manifolds} are inherently even-dimensional because they are defined by a nondegenerate closed 2-form, necessitating an even number of dimensions to pair up the coordinates effectively for the form \[\omega =\displaystyle \sum_{i=1}^n dx^i \wedge dx^{n+i}.\]
       
        \emph{Multisymplectic (or $n$-plectic) manifolds} do not require even dimensionality. They are characterized by a closed form of a higher degree (greater than 2), which allows them to adopt various dimensional configurations.

    \paragraph{Non-existence of Darboux Coordinates:}
    In \emph{symplectic geometry}, the Darboux theorem \cite[Theorem 8.1]{dasilva2005symplectic}  guarantees that it is always possible to find local coordinates (Darboux coordinates) around any point on a symplectic manifold where the symplectic form appears in a standard simplified form $\displaystyle \omega = \sum_{i=1}^n dx^i \wedge dx^{n+i}$.
    
         \emph{Multisymplectic geometry} does not generally support the existence of such simplifying local coordinate systems.        
Properties in symplectic geometry do not generally extend to multisymplectic manifolds \( (M, \omega) \). The following proposition demonstrates the failure of the Darboux theorem in a specific multisymplectic context:

\section{Why Darboux's Theorem Fails in Multisymplectic Geometry} \label{sec:darboux-failure}

One of the foundational results in symplectic geometry is the well-known \textbf{Darboux Theorem}. It states that, locally, all symplectic manifolds are geometrically the same.

When we move to \textbf{multisymplectic geometry}, where the symplectic 2-form \( \omega \) is replaced by a closed non-degenerate \( k \)-form for \( k \geq 3 \), the story changes significantly.
In contrast to the symplectic case \( (k=2) \), where all non-degenerate forms are locally equivalent under linear change of coordinates (i.e., they belong to a single \( \mathrm{GL}(2n, \mathbb{R}) \)-orbit), for \( k \geq 3 \), there are \emph{multiple} orbits. These orbits correspond to different \textit{linear types} of non-degenerate forms.

To understand the failure of Darboux's Theorem in this setting, consider the classification of non-degenerate \( 3 \)-forms on \( \mathbb{R}^6 \). According to results from Bryant \cite{bryant2005remarksgeometrycomplex6manifolds}  and Hitchin \cite{bryant2005remarksgeometrycomplex6manifolds}, the space of non-degenerate 3-forms on \( \mathbb{R}^6 \) decomposes into multiple distinct \( \mathrm{GL}(6, \mathbb{R}) \)-orbits, called \emph{linear types} of 3-forms. That is, not all non-degenerate 3-forms are related by a linear change of coordinates.

\begin{defn}
Two non-degenerate \( k \)-forms \( \alpha, \beta \in \Lambda^k V^* \) on a real vector space \( V \) have the same \emph{linear type} if there exists a linear isomorphism \( g \in \mathrm{GL}(V) \) such that \( g^* \beta = \alpha \).
\end{defn}

For example, on \( \mathbb{R}^6 \), there exist three distinct \( \mathrm{GL}(6, \mathbb{R}) \)-orbits of non-degenerate 3-forms. These orbits can be distinguished using an algebraic invariant:

\begin{thm}[\cite{bryant2005remarksgeometrycomplex6manifolds, Cantrijn1998} ] \label{classification theorem}
Let \( V \) be a six-dimensional real vector space, \( \Omega \in \Lambda^6 V^* \setminus \{0\} \) a volume form, and \( \{e^1, \dots, e^6\} \) an ordered basis of \( V^* \). Let \( \alpha \in \Lambda^3 V^* \) be non-degenerate. Then there exists a unique scalar \( \lambda_\alpha \in \mathbb{R} \) such that
\[
\mathrm{tr}(J_\alpha^2) = \lambda_\alpha \cdot (\Omega \otimes \Omega) \in (\Lambda^6 V^*)^{\otimes 2}.
\]
The value and sign of \( \lambda_\alpha \) are independent of the choice of volume form and determine the \emph{linear type} of \( \alpha \):
\begin{itemize}
    \item[(i)] \( \lambda_\alpha > 0 \iff \alpha \sim \alpha_{(i)} \),
    \item[(ii)] \( \lambda_\alpha < 0 \iff \alpha \sim \alpha_{(ii)} \),
    \item[(iii)] \( \lambda_\alpha = 0 \iff \alpha \sim \alpha_{(iii)} \),
\end{itemize}
where \( \sim \) denotes \( \mathrm{GL}(6, \mathbb{R}) \)-equivalence, and the canonical representatives are given by:
\begin{align*}
\alpha_{(i)} &= e^1 \wedge e^2 \wedge e^3 + e^4 \wedge e^5 \wedge e^6, \\
\alpha_{(ii)} &= e^1 \wedge e^3 \wedge e^5 - e^1 \wedge e^4 \wedge e^6 - e^2 \wedge e^3 \wedge e^6 - e^2 \wedge e^4 \wedge e^5, \\
\alpha_{(iii)} &= e^1 \wedge e^5 \wedge e^6 - e^2 \wedge e^4 \wedge e^6 + e^3 \wedge e^4 \wedge e^5.
\end{align*}
\end{thm}

Using this classification (Theorem \ref{classification theorem}), Ryvkin constructed a smooth family of 3-forms on \( \mathbb{R}^6 \) whose linear type changes from point to point:

\begin{pro}[{\cite[Proposition 2.3]{ryvkin2016multisymplectic}}]
Let \( f : \mathbb{R}^6 \to \mathbb{R} \) be a smooth function of \( x_2, x_4, x_5 \), and define:
\[
\omega_f = dx^1 \wedge dx^3 \wedge dx^5 - dx^1 \wedge dx^4 \wedge dx^6 - dx^2 \wedge dx^3 \wedge dx^6 + f(x) \cdot dx^2 \wedge dx^4 \wedge dx^5.
\]
Then \( \omega_f \) is a non-degenerate closed 3-form on \( \mathbb{R}^6 \), i.e., a multisymplectic form. Its linear type at a point \( p \) depends on the value of \( f(p) \):
\[
\omega_f|_p \sim
\begin{cases}
\alpha_{(i)} & \text{if } f(p) > 0, \\
\alpha_{(ii)} & \text{if } f(p) < 0, \\
\alpha_{(iii)} & \text{if } f(p) = 0.
\end{cases}
\]
\end{pro}

This variable behavior of linear type contradicts the existence of Darboux charts, since any local coordinate transformation preserving \( \omega \) must preserve its linear type. Therefore, the multisymplectic manifold \( (\mathbb{R}^6, \omega_f) \) does \emph{not} admit an atlas of charts in which the multisymplectic form has constant coefficients, nor are neighborhoods around distinct points necessarily multisymplectomorphic.

\begin{cor}[{\cite[Corollary  2.4]{ryvkin2016multisymplectic}}]
Let \( f(x) = x_2 \). Then \( (\mathbb{R}^6, \omega_f) \) has no neighborhood around the origin in which the linear type of \( \omega_f \) is constant. Thus, there is no Darboux chart around 0. Moreover, the group of multisymplectomorphisms \( \mathrm{Diff}_{\omega_f}(\mathbb{R}^6) \) does not act transitively.
\end{cor}

While no general Darboux theorem exists in multisymplectic geometry, it is noteworthy that for some special classes of \( n \)-plectic manifolds, local normal form theorems \emph{do} exist. We refer the reader to \cite{Gracia2024} for a detailed and comprehensive treatment of such cases. Moreover, as noted in \cite{ryvkin2016multisymplectic}, if the necessary condition of \emph{locally constant linear type} is satisfied, and further structural constraints are imposed (e.g., flat connections, homogeneity), restricted Darboux-like theorems can be established (cf. also \cite{Cantrijn1998, echeverria2012invariant, martin1988darboux}).

 \section{Hamiltonian Structures}  
In the symplectic case, Hamiltonian vector fields are characterized by the property that the interior product with the symplectic form yields an exact 1-form. In the multisymplectic case, we deal instead with an \((n+1)\)-form and define Hamiltonian vector fields via their contraction with \(\omega\) producing an exact \(n\)-form.

A vector field \(v\) is multisymplectic if and only if the interior product of \(v\) with \(\omega\) is a closed form:
\[
d(\iota_v \omega) = 0.
\]
This condition is a direct consequence of Cartan’s identity:
\[
\mathcal{L}_v \omega = d(\iota_v \omega) + \iota_v (d\omega),
\]
and the assumption that \(\omega\) is closed (\(d\omega = 0\)). It simplifies to:
\[
\mathcal{L}_v \omega = d(\iota_v \omega).
\]
Therefore, the vanishing of the Lie derivative \(\mathcal{L}_v \omega\) implies that \(\iota_v \omega\) is closed.

Just as in symplectic geometry, we are particularly interested in vector fields for which this closed form is also exact. These are known as Hamiltonian vector fields in the multisymplectic context, as formally defined below.

\begin{defn}[Hamiltonian Vector Fields on Pre-\(n\)-Plectic Manifolds]\label{def:multisymplectic_Hamiltonian_vector_field}
Let \((M, \omega)\) be a pre-\(n\)-plectic manifold. An \((n-1)\)-form \(\alpha \in \Omega^{n-1}(M)\) is said to be \emph{Hamiltonian} if there exists a vector field \(v \in \mathfrak{X}(M)\) such that
\[
d\alpha = -\iota_v \omega.
\]
In this case, \(v\) is called the \emph{Hamiltonian vector field} associated to \(\alpha\). We denote the space of all Hamiltonian \((n-1)\)-forms by \(\Omega_{\mathrm{Ham}}^{n-1}(M, \omega)\) and the corresponding vector fields by \(\mathfrak{X}_{\mathrm{Ham}}(M, \omega)\). The identity
\[
d\alpha + \iota_v \omega = 0
\]
is known in the literature as the \emph{Hamilton--De Donder--Weyl equation} \cite{Forger2005, ryvkin2016multisymplectic}.

In the case where \(\omega\) is only pre-\(n\)-plectic, the vector field \(v\) solving this equation need not be unique. However, if \(\omega\) is nondegenerate—that is, if \((M, \omega)\) is \(n\)-plectic—then the uniqueness of the Hamiltonian vector field is guaranteed. In that case, we obtain a canonical map:
\begin{align*}
\Omega_{\mathrm{Ham}}^{n-1}(M, \omega) &\to \mathfrak{X}_{\mathrm{Ham}}(M, \omega), \\
\alpha &\mapsto v_\alpha
\end{align*}
\end{defn}

\noindent   where \(v_\alpha\) is the unique Hamiltonian vector corresponding to \(\alpha\). This structure provides the foundation for the algebraic framework of observables in multisymplectic geometry. To express these ideas rigorously, we now introduce the language of graded vector spaces, which underlies the construction of \(L_\infty\)-algebras discussed later.

\section{Graded Vector Spaces} \label{section: graded vector space}

In multisymplectic geometry, especially when describing higher algebraic structures such as \(L_\infty\)-algebras, we need to work within the framework of graded vector spaces. These spaces allow for the natural encoding of degrees associated with forms, brackets, and multilinear operations.

\begin{defn}[Graded vector space]
A $\mathbb{Z}$-\emph{graded vector space} is a vector space \(V\) decomposed as a direct sum
\[
V=\bigoplus_{k\in\mathbb{Z}} V^k ,
\]
where each \(V^k\) is the subspace of \emph{homogeneous elements of degree \(k\)}.
For a homogeneous element \(x\in V^k\) we call \(k\) its \emph{degree} and write
\(\deg(x)=k\) (equivalently, \( |x|=k \)).

If \(x_1, \dots, x_n \in V\) are homogeneous elements, we define their (total) degree as
\[
\deg(x_1 \otimes \cdots \otimes x_n) = \sum_{i=1}^n \deg(x_i).
\]
where \(\deg(x_i)\) denotes the degree of \(x_i\).
\end{defn}

Any element \(v \in V\) can be uniquely written as a finite sum \(v = \sum v^k\) where \(v^k \in V^k\) and \(v^k = 0\) for all but finitely many \(k\). We also define a grading shift for such spaces:

\begin{defn}
Given a graded vector space \(V\), its \(k\)-shifted version \(V[k]\) is defined by:
\[
V[k]^i := V^{i+k} \quad \text{for all } i \in \mathbb{Z}.
\]
\end{defn}

To work with multilinear maps on graded vector spaces, it is necessary to keep track of signs arising from permutations of graded elements. This leads us to the notion of Koszul signs and unshuffles.

\begin{defn}[Koszul sign for a swap]
For homogeneous $y,z$, the \emph{Koszul (graded) commutativity rule} is
\[
  y\,z \;=\; (-1)^{\,\deg(y)\, \deg(z)}\, z\,y .
\]

\end{defn}

\begin{defn}[Koszul sign for a permutation]
Let $x_1,\dots,x_n$ be homogeneous and let $\sigma\in S_n$.
The \emph{Koszul sign} $\varepsilon(\sigma;x_1,\dots,x_n)\in\{\pm1\}$ is the
(unique) scalar such that
\[
  x_1\cdots x_n \;=\; \varepsilon(\sigma;x_1,\dots,x_n)\; x_{\sigma(1)}\cdots x_{\sigma(n)},
\]
obtained by reordering via adjacent transpositions and the swap rule above.
\end{defn}

\begin{defn}
A permutation \(\sigma \in \mathcal{S}_{p+q}\) is called a \((p, q)\)-\emph{unshuffle} if 
\[
\sigma(1) < \cdots < \sigma(p) \quad \text{and} \quad \sigma(p+1) < \cdots < \sigma(p+q).
\]
The set of all \((p, q)\)-unshuffles is denoted \(\mathrm{Sh}(p, q)\).
\end{defn}

We illustrate this with several examples:
\begin{itemize}
\item \(\mathrm{Sh}(1, 1) = \{\mathrm{id}, (12)\}\),
\item \(\mathrm{Sh}(2, 1) = \{\mathrm{id}, (23), (123)\}\),
\item \(\mathrm{Sh}(1, 2) = \{\mathrm{id}, (12), (132)\}\),
\item \(\mathrm{Sh}(2, 2) = \{\mathrm{id}, (23), (13)(24), (123), (1243), (243)\}\),
\item \(\mathrm{Sh}(3, 1) = \{\mathrm{id}, (34), (234), (1234)\}\),
\item \(\mathrm{Sh}(1, 3) = \{\mathrm{id},  (12), (132), (1432)\}\),
  \item \(\mathrm{Sh}(p+q, 0) = \{\mathrm{id}\}\).   
\end{itemize}

Finally, we define skew-symmetric maps on graded vector spaces:

\begin{defn}
Let \(f: V^{\otimes n} \to W\) be a multilinear map between graded vector spaces. Then \(f\) is said to be \emph{skew-symmetric} if for all \(\sigma \in \mathcal{S}_n\), we have
\[
f(v_{\sigma(1)}, \dots, v_{\sigma(n)}) = (-1)^{\operatorname{sgn}(\sigma)} \cdot \epsilon(\sigma; v_1, \dots, v_n) \cdot f(v_1, \dots, v_n),
\]
where \(\operatorname{sgn}(\sigma) \in \{0,1\}\) is the parity of the permutation \(\sigma\), and \(\epsilon(\sigma; v_1, \dots, v_n)\) is the Koszul sign arising from exchanging graded elements \(v_1, \dots, v_n\) according to \(\sigma\).

\end{defn}

With these algebraic tools in place, we are now ready to explore one of the key algebraic structures arising in multisymplectic geometry: \(L_\infty\)-algebras.

\section{\Linfty-Algebras} \label{section: L_infinity algebra}

With the language of graded vector spaces in place, we are now ready to define and study $L_\infty$-algebras. These structures generalize Lie algebras by allowing for higher-order operations that satisfy generalized Jacobi identities up to coherent homotopies. This is especially relevant in multisymplectic geometry, where the structure of observables forms an $L_\infty$-algebra rather than a Lie algebra.

\begin{defn} \cite{Lada1993} \label{Linfty algebra} 
A \emph{Lie $L_{\infty}$-algebra} is a graded vector space $L$
equipped with a collection
\[\left \{l_{k}: L^{\otimes k} \to L| 1 \leq k < \infty \right\}\]
of
skew-symmetric multilinear maps with $\deg{l_{k}}=k-2$ such that the following identity holds for $1 \leq m < \infty$:
\begin{align} \label{generalized jacobi identity}
   \sum_{\substack{i+j = m+1, \\ \sigma \in \mathrm{Sh}(i,m-i)}}
  (-1)^{\operatorname{sgn}(\sigma)}\epsilon(\sigma)(-1)^{i(j-1)} l_{j}
   (l_{i}(x_{\sigma(1)}, \dots, x_{\sigma(i)}), x_{\sigma(i+1)},
   \ldots, x_{\sigma(m)})=0.
\end{align}
\end{defn}

The structure maps \( l_k : L^{\otimes k} \to L \) that appear in an \( L_\infty \)-algebra are commonly referred to as \emph{multibrackets} or \emph{higher brackets}. These maps are graded skew-symmetric and encode the higher-order operations that generalize the Lie bracket and Jacobi identity up to coherent homotopies.

 $L_\infty$-algebras can be further constrained by requiring the underlying graded vector space to be concentrated in low degrees.

\begin{defn}[Lie $n$-algebra]  \cite{Lada1993, rogers2011higher} \label{Ln algebra}
Let $(L, \{l_k\})$ be an $L_\infty$-algebra, where $L = \bigoplus_{i \in \mathbb{Z}} L_i$ is a graded vector space. We say that $L$ is \emph{concentrated in degrees $0,1,\dots,n-1$} if
\[
L_i = 0 \quad \text{for all } i \notin \{0,1,\dots,n-1\}.
\]
Then $(L, \{l_k\})$ is called a \emph{Lie $n$-algebra} if and only if its underlying graded vector space is concentrated in degrees $0$ through $n-1$.
\end{defn}

\begin{defn}[$L_\infty$-morphism identities in components]
Let $(L,\{\ell_k\}_{k\ge1})$ and $(L',\{\ell'_k\}_{k\ge1})$ be $L_\infty$-algebras.
An $L_\infty$-morphism $F_\bullet:L\to L'$ is a family of graded maps
$F_n:\wedge^n L\to L'$ of degree $|F_n|=1-n$ satisfying, for all $n\ge1$ and
homogeneous $x_1,\dots,x_n\in L$,
\begin{align*}\label{eq:L-infty-morphism}
&\sum_{i=1}^{n}\ \sum_{\sigma\in\mathrm{Sh}(i,n-i)}
(-1)^{i(n-i)}\,\epsilon(\sigma;x)\,
F_{\,n-i+1}\!\Big(\,\ell_i(x_{\sigma(1)},\dots,x_{\sigma(i)}),\ x_{\sigma(i+1)},\dots,x_{\sigma(n)}\Big)
\\
&= \sum_{k\ge1}\sum_{\substack{n_1+\cdots+n_k=n\\ n_j\ge1}}\ 
\sum_{\sigma\in\mathrm{Sh}(n_1,\dots,n_k)}
\frac{\epsilon(\sigma;x)}{k!}\,
\ell'_k\!\Big(F_{n_1}(x_{\sigma(1)},\dots,x_{\sigma(n_1)}),\ \dots,\ 
F_{n_k}(\dots)\Big).
\end{align*}
Here $\mathrm{Sh}(i,n-i)$ denotes the set of $(i,n-i)$-unshuffles, more generally
$\mathrm{Sh}(n_1,\dots,n_k)$ the set of $(n_1,\dots,n_k)$-unshuffles, and
$\epsilon(\sigma;x)$ is the Koszul sign determined by permuting
$(x_1,\dots,x_n)$ by $\sigma$ in the graded-antisymmetric setting.

\medskip
Equivalently (and often more useful), the first few cases are:

\begin{itemize}
\item[$n=1$] (chain map)
\[
F_1\ell_1(x)\;=\;\ell'_1 F_1(x).
\]

\item[$n=2$] (bracket up to $F_2$-homotopy)
\[
\begin{aligned}
F_1\ell_2(x_1,x_2)\;-\;\ell'_2(F_1x_1,F_1x_2)
&=\ \ell'_1 F_2(x_1,x_2)\;+\;F_2(\ell_1x_1,x_2)\;-\;(-1)^{|x_1|}\,F_2(x_1,\ell_1x_2).
\end{aligned}
\]

\item[$n=3$] (Jacobiator compatibility up to $F_2,F_3$)
\[
\begin{aligned}
&F_1\ell_3(x_1,x_2,x_3)\;-\;\ell'_3(F_1x_1,F_1x_2,F_1x_3) \\
&=\;\sum_{\sigma\in\mathrm{Sh}(2,1)}\!\!\epsilon(\sigma;x)\,
\Big\{
F_2\big(\ell_2(x_{\sigma(1)},x_{\sigma(2)}),x_{\sigma(3)}\big)
-\ell'_2\big(F_2(x_{\sigma(1)},x_{\sigma(2)}),F_1x_{\sigma(3)}\big)
\Big\} \\
&\quad\ +\ \ell'_1 F_3(x_1,x_2,x_3)\;-\;\sum_{j=1}^{3}(-1)^{|x_1|+\cdots+|x_{j-1}|}\,
F_3(x_1,\dots,\ell_1x_j,\dots,x_3).
\end{aligned}
\]
\end{itemize}

\noindent
 
\end{defn}

\begin{defn}\cite[Definition 3.6]{callies2016homotopy}\label{def:Linf-quasi-iso}
A morphism $(f_k)\!:\,(L,\ell_k)\to(L',\ell'_k)$ of $L_\infty$-algebras is an
\emph{$L_\infty$-quasi-isomorphism} if the induced morphism of complexes
\[
f_1:\ (L,\ell_1)\longrightarrow (L',\ell'_1)
\]
induces an isomorphism on cohomology:
\[
H^\ast(f_1):\ H^\ast(L)\xrightarrow{\ \cong\ } H^\ast(L').
\]
\end{defn}

Just as every symplectic manifold naturally gives rise to a Poisson algebra structure on the space of smooth functions, multisymplectic (or \(n\)-plectic) manifolds give rise to higher analogues of Poisson algebras. The appropriate algebraic structures in this setting are \(L_\infty\)-algebras, which encode higher-order brackets and relations generalizing those of Lie algebras. 

In a foundational result, Rogers~\cite{rogers2011higher} showed that every \(n\)-plectic manifold canonically defines an \(L_\infty\)-algebra of observables, thereby extending the symplectic case to the multisymplectic setting.

\begin{thm}[\emph{\cite[Theorem 3.14]{rogers2011higher}, see also \cite[Theorem 4.7]{callies2016homotopy}}]\label{def:RogersAlgebra1}
	Given an $n$-plectic manifold $(M,\omega)$, there exists an $L_\infty$-algebra $L_{\infty}(M,\omega)=(L,\{l_k\}_{k\geq 1})$ with 
    
	\begin{itemize}
		\item 
		the underlying graded vector space $L$, where
	\begin{equation*}
		L_i=\begin{cases}
			\Omega_{\mathrm{Ham}}^{n-1}(M,\omega) 
			& \quad~\text{if } i=0
			\\
			\Omega^{n-1+i}(M) 
		 	&  \quad~\text{if } 1-n \leq i\leq -1
		 	\\			
			0 & \quad ~\text{otherwise,}
		\end{cases}
	\end{equation*}

	\item 
		$n+1$ nontrivial multibrackets $\lbrace l_k : L^{\wedge k} \to L\rbrace_{1\leq k\leq n+1}$, given by
	\begin{displaymath}
		l_1(\alpha) = 
		\begin{cases}
			0 & \quad\text{if~} \deg(\alpha) = 0	
			\\
			\mathrm{d} \alpha & \quad \text{if~} \deg(\alpha) \leq -1,
		\end{cases}
	\end{displaymath}
	and, for $ 2 \leq k \leq n+1$, as
	\begin{displaymath}
		l_k(\alpha_1,\dots,\alpha_k) = 
		\begin{cases}
			\epsilon(k) \iota_{ v_{\alpha_k}}\ldots\iota_{ v_{\alpha_1}}\omega
			& \quad\text{if~} \deg(\alpha_i)=0 \text{ for } 1\leq i \leq k
			\\
			0 & \quad\text{otherwise,}		
		\end{cases}
	\end{displaymath}
	\end{itemize}

\noindent		where $ v_{\alpha_k}$ denotes any Hamiltonian vector field corresponding to $\alpha_k\in \Omega^{n-1}_{\mathrm{Ham}}(M,\omega)$ and $\epsilon(k) = - (-1)^{\frac{k(k+1)}{2}}$ is the total Koszul sign.         
  
\end{thm}

This construction not only establishes the existence of an $L_\infty$-algebra associated to any $n$-plectic manifold, but also provides an explicit model for the algebra of observables in multisymplectic geometry. The graded vector space described above can be viewed as a truncation of the de Rham complex, with degrees inverted, and the higher brackets generalize the Poisson bracket of symplectic geometry to a homotopy Lie structure. This motivates the following definition.

\begin{defn}[\cite{rogers2011higher}, cf. also \cite{callies2016homotopy}]
	The \emph{$L_\infty$-algebra of observables} $L_\infty(M,\omega)$ 
	of the (pre)-$n$-plectic manifold $(M,\omega)$ consists of a chain complex $L_\bullet$

\[\begin{array}{ccccccccccccccc}
    0 & \to & L_{n-1} & \to & L_{n-2}  & \to & \cdots & \to & L_{k-2} & \to & \cdots & \to & L_{1} & \to & L_0\\
     &  & \parallel &  & \parallel &  &  &  & \parallel &  &  &  & \parallel &  & \parallel \\
      &  & \Omega^{0}(M)& \to & \Omega^{0}(M)  & \to & \cdots & \to & \Omega^{k-2}(M) & \to & \cdots & \to & \Omega^{n-2}(M) & \to & \Omega^{n-1}_{\textrm{Ham}}(M)\\ 
\end{array}\]

	which is a truncation of the de-Rham complex with inverted grading,
	endowed with $n$ (skew-symmetric) multibrackets $(2 \leq k \leq n+1)$
	\begin{equation}
		\begin{tikzcd}[column sep= small,row sep=0ex]
				[\cdot,\dots,\cdot]_k \colon& \Lambda^k\left(\Omega^{n-1}_{\textrm{Ham}}\right) 	\arrow[r]& 				\Omega^{n+1-k} \\
				& \sigma_1\wedge\dots\wedge\sigma_k 	\ar[r, mapsto]& 	\epsilon(k)\iota_{v_{\sigma_k}}\dots\iota_{v_{\sigma_1}}\omega 
		\end{tikzcd}		
	\end{equation}
	where $v_{\sigma_k}$ is any Hamiltonian vector field associated to $\sigma_k\in \Omega^{n-1}_{\textrm{Ham}}$ and $\epsilon(k) := - (-1)^{\frac{k(k+1)}{2}}$ is the Koszul sign.
\end{defn}

\vspace{1em}
When one fixes a form $\omega$ on a manifold $M$ it is natural to highlight the group actions preserving this extra structure, also known as ``symmetries''.

\section{Lie \Twoalgebras}

We now focus on Lie \(2\)-algebras, which are \(L_\infty\)-algebras concentrated in degrees 0 and 1. One of the goals of this thesis is to construct a Lie \(2\)-algebra associated to a quasi-Hamiltonian \(G\)-space. To lay the groundwork, we survey two equivalent formulations of Lie \(2\)-algebras that appear in the literature and briefly present some canonical examples, including the Atiyah and Courant Lie \(2\)-algebras.

\subsection{Revisiting Two Equivalent Definitions of Lie \Twoalgebras}

Lie \(2\)-algebras can be understood in at least two equivalent ways. The first interpretation arises from categorification and defines a Lie \(2\)-algebra as a semistrict 2-vector space equipped with a skew-symmetric bracket satisfying a Jacobi identity up to coherent isomorphism, as introduced by Baez and Crans in \cite{baez2004higher}.

\begin{defn}\cite{baez2004higher} \label{def-Lie-alegra1}
A \emph{Lie 2-algebra} is defined by the following data:

\begin{enumerate}
    \item A $2$-term chain complex of vector spaces
    \[
    L_{\bullet} = (L_1 \stackrel{d}{\longrightarrow} L_0)
    \]
    where $d: L_1 \to L_0$ is the differential.
    
    \item A skew-symmetric chain map, denoted by
    \[
    [\cdot,\cdot]: L_{\bullet} \otimes L_{\bullet} \to L_{\bullet},
    \]
    and referred to as the \emph{bracket}.
    
    \item A skew-symmetric chain map
    \[
    J: L_{\bullet} \otimes L_{\bullet} \otimes L_{\bullet} \to L_{\bullet},
    \]
    which serves as a transformation between two chain maps. Specifically, it maps:
    \[
    \begin{array}{ccl}  
        L_{\bullet} \otimes L_{\bullet} \otimes L_{\bullet} & \to & L_{\bullet}   \\
        x \otimes y \otimes z & \longmapsto & [x,[y,z]],  
    \end{array}
    \]
    to 
    \[
    \begin{array}{ccl}  
        L_{\bullet} \otimes L_{\bullet} \otimes L_{\bullet}& \to & L_{\bullet}   \\
        x \otimes y \otimes z & \longmapsto & [[x,y],z] + [y,[x,z]]. 
    \end{array}
    \]
    This homotopy is termed the \emph{Jacobiator}.
    
    \item The Jacobiator must also fulfill a compatibility condition given by:
    \begin{align}\label{big_J}
    &[x,J(y,z,w)] + J(x,[y,z],w) + J(x,z,[y,w]) + [J(x,y,z),w] + [z,J(x,y,w)] \nonumber \\
    &= J(x,y,[z,w]) + J([x,y],z,w)  + [y,J(x,z,w)] + J(y,[x,z],w) + J(y,z,[x,w]).
    \end{align}
\end{enumerate}

\end{defn}

The second formulation describes Lie \(2\)-algebras more algebraically, in terms of 2-term \(L_\infty\)-algebras. These consist of a chain complex of vector spaces \( L_1 \xrightarrow{d} L_0 \), together with skew-symmetric multibrackets satisfying higher Jacobi identities. This formulation is particularly useful in computations and constructions. A thorough exposition of this approach is provided by Noohi in \cite{noohi2013integrating}, and a concise summary can also be found in \cite[Section 4.1]{krepski2022multiplicative}.

\begin{defn}\cite{noohi2013integrating}\label{def-Lie-alegra2}
A \emph{$2$-term $L_\infty$-algebra} $\mathbb{L}$ is defined as a $2$-term chain complex of vector spaces $L_{1}\xrightarrow{\mathrm{d}}L_0$ along with:
\begin{itemize}
\item three skew-symmetric bilinear maps $[\cdot, \cdot]: L_0\times L_0\to L_0$, $[\cdot, \cdot]: L_0\times L_{1}\to L_{1}$, and $[\cdot, \cdot]: L_1\times L_{0}\to L_{1}$;
\item a skew-symmetric trilinear map $\left\langle \cdot, \cdot, \cdot\right\rangle: L_0\times L_0\times L_0\to L_{1}$,
\end{itemize}
satisfying the following conditions for all $w,x,y,z\in L_0$ and $a,b\in L_{1}$:
\begin{enumerate} 
\item  $\mathrm{d} [x,a]=[x,\mathrm{d} a]$
;
\item $\mathrm{d} [a,b]=[a,\mathrm{d} b]$;
\item  $[[x,y],z]+[[y,z],x]+[[z,x],y]=-\mathrm{d} \left\langle x,y,z\right\rangle$;
\item  $[[x,y],a]+[[y,a],x]+[[a,x],y]=-\left\langle x,y,da\right\rangle$;
\item  $\left\langle [w,x],y,z\right\rangle-\left\langle [w,y],x,z\right\rangle +\left\langle [w,z],x,y\right\rangle+\left\langle [x,y],w,z\right\rangle+\left\langle [y,z],w,x\right\rangle-\left\langle [x,z],w,y\right\rangle-
[\left\langle w,x,y\right\rangle,z]-[\left\langle w,y,z\right\rangle,x]+[\left\langle w,x,z\right\rangle,y]+[\left\langle x,y,z\right\rangle,w]=0$.
\end{enumerate}
 
\end{defn}

As a final equivalent viewpoint, we recall a third definition inspired by chain complexes equipped with homotopies encoding the failure of strict Lie brackets.

\begin{defn}\label{def-Lie-alegra3}\cite{rogers20132plectic}
A  \emph{Lie 2-algebra} is a 2-term chain complex of vector spaces
$L = (L_1\stackrel{d}\rightarrow L_0)$ equipped with the following structure:
\begin{itemize}
\item a chain map $[\cdot,\cdot]: L \otimes L\to L$ called the 
\emph{bracket};
\item a chain homotopy $S : L\otimes L \to L$
from the chain map
\[     \begin{array}{ccl}  
     L \otimes L &\to& L   \\
     x \otimes y &\longmapsto& [x,y]  
  \end{array}
\]
to the chain map
\[     \begin{array}{ccl}  
     L \otimes L & \to & L   \\
     x \otimes y & \longmapsto & -[y,x]  
  \end{array}
\]
called the \emph{alternator};
\item an antisymmetric chain homotopy $J : L \otimes L \otimes L
  \to L$ 
from the chain map
\[     \begin{array}{ccl}  
     L \otimes L \otimes L & \to & L   \\
     x \otimes y \otimes z & \longmapsto & [x,[y,z]]  
  \end{array}
\]
to the chain map
\[     \begin{array}{ccl}  
     L \otimes L \otimes L& \to & L   \\
     x \otimes y \otimes z & \longmapsto & [[x,y],z] + [y,[x,z]]  
  \end{array}
\]
called the  \emph{Jacobiator}.
\end{itemize}
In addition, the following equations are required to hold:
\begin{equation*}
\begin{array}{c}
  [x,J(y,z,w)] + J(x,[y,z],w) +
  J(x,z,[y,w]) + [J(x,y,z),w] \\ + [z,J(x,y,w)] 
  = J(x,y,[z,w]) + J([x,y],z,w) \\ + [y,J(x,z,w)] + J(y,[x,z],w) + J(y,z,[x,w]),
\end{array}
\end{equation*}
\begin{equation*}
    J(x,y,z)+J(y,x,z)=-[S(x,y),z],
\end{equation*}
\begin{equation*}
    J(x,y,z)+J(x,z,y)=[x,S(y,z)]-S([x,y],z)-S(y,[x,z]),
\end{equation*}
\begin{equation*}
    {S(x,[y,z])} = S([y,z],x).
\end{equation*}
\end{defn}

\vspace{1cm}

These three definitions---categorified Lie algebras, 2-term $L_\infty$-algebras, and homotopy Lie brackets---are all equivalent up to isomorphism. This equivalence is established in \cite[Proposition 8]{baez2004higher}, providing a flexible toolkit for working with Lie $2$-algebras depending on the geometric or algebraic context.

We now explore the classification of Lie $2$-algebras based on the presence of nontrivial Jacobiators or alternators.

\begin{defn}\cite{baez2004higher}\label{definition:hemistrict-al}
A Lie 2-algebra is called \emph{ hemistrict} if its Jacobiator is the identity
chain homotopy. Similarly, a Lie 2-algebra is called   \emph{semistrict} if its alternator is the identity chain homotopy.
\end{defn}

Having reviewed the abstract definitions, we now turn to Lie 2-algebras arising from 2-plectic manifolds. 

\subsection{Lie 2-algebras associated with a 2-plectic manifold} \label{Lie 2-lagebra of a 2-plecticmanifold}

We know from Theorem~\ref{def:RogersAlgebra1} that any \(2\)-plectic manifold naturally gives rise to a Lie \(2\)-algebra. Notably, the higher bracket structure appearing in Theorem~\ref{def:RogersAlgebra1} corresponds to the \emph{semi-bracket}, as defined in Definition~\ref{definition: semi-bracket1}. 

In the specific case of \(2\)-plectic manifolds, there exist (at least) two distinct—but canonically associated—Lie \(2\)-algebra structures: the \emph{hemi-strict} and the \emph{semi-strict} Lie \(2\)-algebras. Both structures share the same underlying 2-term chain complex of forms and vector fields, but they differ in the choice of bracket and Jacobiator. The hemi-strict Lie \(2\)-algebra uses the \emph{hemi-bracket}, which is strictly antisymmetric but fails to satisfy the Jacobi identity strictly, whereas the semi-strict version employs the \emph{semi-bracket}, which fails to be strictly antisymmetric but satisfies the Jacobi identity up to coherent homotopy. 
This provides equivalent but non-isomorphic ways of encoding observables on a 2-plectic manifold via Lie \(2\)-algebras.

\begin{defn}
\label{definition: semi-bracket1} \cite{rogers2011higher}
Let \((M, \omega)\) be a 2-plectic manifold, and let \(\Omega_{\mathrm{Ham}}^1(M)\) denote the space of Hamiltonian 1-forms on \(M\).
The \emph{semi-bracket}, denoted \(\left\{\cdot, \cdot\right\}_s\), is defined by 
\begin{align*}
\left\{\cdot, \cdot\right\}_s : \Omega_{\mathrm{Ham}}^1(M)\times \Omega_{\mathrm{Ham}}^1(M) &\to \Omega_{\mathrm{Ham}}^1(M)\\
(\alpha, \beta) &\mapsto \iota_{v_\beta}\iota_{v_\alpha}\omega
\end{align*}
\end{defn}

\noindent 
This is indeed a bracket since one can check that the following properties are satisfied. 

\begin{itemize}
  \item \textit{Hamiltonian Property:} The semi-bracket of Hamiltonian forms should itself be a Hamiltonian form. Formally, this means 
  \[
  d\left\{\alpha, \beta\right\}_s = -\iota_{[v_\alpha,v_\beta]} \omega,
  \]
  which implies that the Hamiltonian vector field of the semi-bracket is the Lie bracket of the original Hamiltonian vector fields, i.e., \(v_{\left\{\alpha, \beta\right\}_s} = [v_\alpha,v_\beta]\).

  \item \textit{Antisymmetry:} The semi-bracket must be antisymmetric, meaning 
  \[
  \left\{\alpha, \beta\right\}_s = -\left\{\beta, \alpha\right\}_s.
  \]

  \item \textit{Jacobi Identity up to an Exact \(1\)-form:} The semi-bracket should satisfy a generalized Jacobi identity:
  \[
  \left\{\alpha, \left\{\beta, \gamma\right\}_s\right\}_s + \left\{\gamma, \left\{\alpha, \beta\right\}_s\right\}_s + \left\{\beta, \left\{ \gamma, \alpha\right\}_s\right\}_s=-dJ(\alpha,\beta,\gamma)
  ,
  \]
  where 
the Jacobiator is represented by the linear map
\begin{align*}
J : \Omega^1_{\mathrm{Ham}}(M) \otimes \Omega^1_{\mathrm{Ham}}(M) \otimes \Omega^1_{\mathrm{Ham}}(M) &\to   C^\infty(M) \\
\left(\alpha, \beta, \gamma)\right) &\mapsto \left(0, -\iota_{v_\alpha}\iota_{v_\beta}\iota_{v_\gamma}\omega\right).
\end{align*}
\end{itemize}

For a \(2\)-plectic manifold \( (M, \omega) \), one can construct a semi-strict Lie \(2\)-algebra \( L(M, \omega)_{s} \) as proven in a theorem by Rogers and Baez \cite{Baez_2009}.

\begin{thm}\cite[Theorem 4.4]{Baez_2009}
\label{semistrict}
Let \((M, \omega)\) be a 2-plectic manifold. Then there exists a semi-strict Lie \(2\)-algebra \(L(M, \omega)_{s}\) with the following characteristics:
\begin{itemize}
\item $L_0=\Omega^1_{\mathrm{Ham}}(M)$;
\item \(L_1=C^\infty(M)\);
\item the differential  \(d:C^\infty(M)\to \Omega^1_{\mathrm{Ham}}\) is the usual exterior differential;
\item the bracket operation is the semi-bracket \(\left\{\cdot, \cdot\right\}_s\) as defined in Definition \ref{definition: semi-bracket1};
\item the alternator is the identity chain homotopy map, hence given by the bilinear map
\begin{align*}
S:\Omega^1_{\mathrm{Ham}}(M)\times \Omega^1_{\mathrm{Ham}}(M) & \to C^\infty(M)\\
(\alpha, \beta) & \mapsto 0;
\end{align*}
\item the Jacobiator is represented by the trilinear map
\begin{align*}
J : \Omega^1_{\mathrm{Ham}}(M) \otimes \Omega^1_{\mathrm{Ham}}(M) \otimes \Omega^1_{\mathrm{Ham}}(M) &\to   C^\infty(M) \\
\left(\alpha, \beta, \gamma)\right) &\mapsto -\iota_{v_\alpha}\iota_{v_\beta}\iota_{v_\gamma}\omega.
\end{align*}

\end{itemize}
\end{thm}

\begin{exa}[Poisson-Lie 2-Algebra in the Context of (Pre)-2-Plectic Geometry] \cite{djounvouna2023infinitesimalsymmetriesbundlegerbes}
Consider a manifold \( M \) equipped with a closed 3-form, denoted by \( \chi \in \Omega^3(M) \). The \emph{Poisson-Lie 2-algebra of observables}, \( \mathbb{L}(M,\chi) \), is constructed as follows:
\begin{itemize}
\item $L_1=C^\infty(M)$, the set smooth functions on $M$;
\item $L_0= \{(x,\beta) \in \mathfrak{X}(M) \times \Omega^1(M) \, | \, \iota_x \chi = -d\beta \}$;
\item $\mathrm{d}: C^\infty(M) \to \{(x,\beta) \in \mathfrak{X}(M) \times \Omega^1(M) \, | \, \iota_x \chi = -\mathrm{d}\beta \}$ such that $\mathrm{d}(f)=(0, \mathrm{d}f)$, where $v_f$ is the the Hamiltonian vector field corresponding to $f$;
\item the bracket $[\cdot, \cdot ]:L_i\otimes L_j\to L_{i+j}$ with $i+j=0,1$ is given by 
\[
[({u}_1,\beta_1),({u}_2,\beta_2)]=([{u}_1,{u}_2],\iota_{{u}_2}\iota_{{u}_1}\chi),
\]
in degree $0$, and $0$ otherwise;
\item the Jacobiator $J:L_0\otimes L_0 \otimes L_0 \to L_{1}$ is given by
\[
J\left(({u}_1,\beta_1), ({u}_2,\beta_2), ({u}_3,\beta_3)\right) = - \iota_{{u}_3}\iota_{{u}_2}\iota_{{u}_1} \chi.
\]
\end{itemize}
\end{exa}

In parallel to the semi-strict construction, another bracket---the hemi-bracket---gives rise to a different Lie $2$-algebra structure, as outlined below.

\begin{defn}\cite{rogers20132plectic}
\label{definition:hemi-bracket1}
Let \((M, \omega)\) be a 2-plectic manifold, and let \(\Omega_{\mathrm{Ham}}^1(M)\) denote the space of Hamiltonian 1-forms on \(M\). For each \(\alpha \in \Omega_{\mathrm{Ham}}^1(M)\), let \(v_\alpha\) be the unique Hamiltonian vector field associated to \(\alpha\). The \textbf{hemi-bracket} is the bilinear map
\[
\{\cdot, \cdot\}_h : \Omega_{\mathrm{Ham}}^1(M) \times \Omega_{\mathrm{Ham}}^1(M) \to \Omega_{\mathrm{Ham}}^1(M)
\]
defined by
\[
\{\alpha, \beta\}_h := \mathcal{L}_{v_\alpha} \beta,
\]
where \(\mathcal{L}_{v_\alpha}\) denotes the Lie derivative along \(v_\alpha\).
\end{defn}

The hemi-bracket is closely related to the semi-bracket defined earlier. In fact, they differ by an exact 1-form:
\[
\{\alpha, \beta\}_h = \{\alpha, \beta\}_s + d\iota_{v_\alpha} \beta.
\]
This identity follows from Cartan’s magic formula for the Lie derivative, which states that \(\mathcal{L}_v = \iota_v d + d\iota_v\). Applying this to \(\beta\), we compute:
\begin{align*}
\{\alpha, \beta\}_h &= \mathcal{L}_{v_\alpha} \beta \\
&= \iota_{v_\alpha} d\beta + d\iota_{v_\alpha} \beta \\
&= \iota_{v_\alpha}(-\iota_{v_\beta} \omega) + d\iota_{v_\alpha} \beta \quad \text{(since } d\beta = -\iota_{v_\beta} \omega\text{)} \\
&= \{\alpha, \beta\}_s + d\iota_{v_\alpha} \beta.
\end{align*}

Next, we verify that the hemi-bracket defines a valid bracket on \(\Omega^1_{\mathrm{Ham}}(M)\). First, it satisfies the Hamiltonian condition:
\[
d\{\alpha, \beta\}_h = -\iota_{[v_\alpha, v_\beta]} \omega.
\]
This shows that the bracket of two Hamiltonian 1-forms is again a Hamiltonian 1-form, whose associated vector field is the Lie bracket \([v_\alpha, v_\beta]\).

Second, the hemi-bracket is antisymmetric up to an exact term. Indeed, define the bilinear map
\[
S : \Omega^1_{\mathrm{Ham}}(M) \times \Omega^1_{\mathrm{Ham}}(M) \to C^\infty(M)
\]
by
\[
S(\alpha, \beta) := -(\iota_{v_\alpha} \beta + \iota_{v_\beta} \alpha).
\]
Then we have the relation
\[
\{\alpha, \beta\}_h + dS(\alpha, \beta) = -\{\beta, \alpha\}_h.
\]

Finally, the hemi-bracket satisfies the Jacobi identity strictly:
\[
\{\alpha, \{\beta, \gamma\}_h\}_h = \{\{\alpha, \beta\}_h, \gamma\}_h + \{\beta, \{\alpha, \gamma\}_h\}_h.
\]
This allows us to construct a hemistrict Lie 2-algebra as follows.

\begin{thm}\cite[Theorem 4.3]{Baez_2009}
\label{thm:hemistrict_lie2}
Let \((M, \omega)\) be a 2-plectic manifold. Then there exists a hemistrict Lie \(2\)-algebra \(L(M, \omega)_h\) with the following structure:
\begin{itemize}
    \item \(L_0 = \Omega^1_{\mathrm{Ham}}(M)\),
    \item \(L_1 = C^\infty(M)\),
    \item The differential \(d : C^\infty(M) \to \Omega^1_{\mathrm{Ham}}(M)\) is the exterior derivative,
    \item The bracket is the hemi-bracket \(\{\cdot, \cdot\}_h\),
    \item The alternator is the bilinear map \(S(\alpha, \beta) = -(\iota_{v_\alpha} \beta + \iota_{v_\beta} \alpha)\),
    \item The Jacobiator is the identity chain homotopy, i.e., \(J(\alpha, \beta, \gamma) = 0\).
\end{itemize}
\end{thm}

\subsection{Morphisms of Semistrict Lie \Twoalgebras}

Since Lie \(2\)-algebras are precisely \(L_\infty\)-algebras concentrated in degrees \(0\) and \(1\), their morphisms admit a simplified structure compared to general \(L_\infty\)-morphisms. In this section, we describe morphisms between Lie \(2\)-algebras in the 2-term \(L_\infty\) language.

\begin{defn}\cite[Definition 23]{baez2004higher}
\label{homo of Lie 2-algebras}
Let \(L = (L_{\bullet}, [\cdot,\cdot], J)\) and \(L' = (L_{\bullet}', [\cdot,\cdot]^{\prime}, J')\) be two Lie \(2\)-algebras. A \emph{morphism} from \(L\) to \(L'\) consists of the following data:

\begin{enumerate}
    \item A chain map 
    \[
    \phi_{\bullet}: L_{\bullet} \to L_{\bullet}',
    \]
    
    \item A chain homotopy
    \[
    \varphi: L_{\bullet} \otimes L_{\bullet} \to L_{\bullet}',
    \]
    which serves as a homotopy between the two bracket structures. That is, it satisfies
    \[
    \phi_{\bullet} \left( [x,y] \right) - \left [ \phi_{\bullet}(x), \phi_{\bullet}(y) \right]^{\prime} = d' \varphi(x,y) + \varphi(d x, y) + (-1)^{|x|} \varphi(x, d y).
    \]
    
    \item A coherence condition that controls the compatibility with the Jacobiators:
    \begin{align}
    \label{coherence}
    \begin{split}
    \phi_1(J(x,y,z)) &- J^{\prime}(\phi_0(x), \phi_0(y), \phi_0(z)) = \varphi(x, [y,z]) - \varphi([x,y], z) - \varphi(y, [x,z]) \\
    &\quad - [\varphi(x,y), \phi_0(z)]^{\prime} + [\phi_0(x), \varphi(y,z)]^{\prime} - [\phi_0(y), \varphi(x,z)]^{\prime}.
    \end{split}
    \end{align}
\end{enumerate}

\end{defn}

When the chain homotopy \(\varphi\) vanishes identically, we say that the morphism is \textbf{strict}. In this case, the brackets and Jacobiators are preserved exactly, not merely up to homotopy.

\subsection*{Atiyah and Courant Lie \(2\)-Algebras Associated to Pre-2-Plectic Manifolds}

We now turn to important geometric examples of Lie \(2\)-algebras that arise naturally from pre-2-plectic manifolds. Two particularly noteworthy constructions are the Atiyah and Courant Lie \(2\)-algebras, introduced in \cite{Fiorenza2014}, and connected via a natural sequence of \(L_\infty\)-morphisms. 

\begin{defn}\cite[Definition/Proposition 5.2.1]{Fiorenza2014}
Let \((M,\omega)\) be a pre-2-plectic manifold. The \emph{Atiyah Lie \(2\)-algebra}, denoted \(\mathfrak{atiyah}(M,\omega)\), is the graded vector space
\[
\mathfrak{atiyah}(M,\omega)_0 = \mathfrak{X}(M), \quad \mathfrak{atiyah}(M,\omega)_1 = \Omega^0(M),
\]
equipped with the following nontrivial brackets for all \(f \in \Omega^0(M)\), and \(v, v_i \in \mathfrak{X}(M)\), \(i = 1,2,3\):
\[
\begin{aligned}
    & \bra f \ket^\mathfrak{a}_1 = 0, \\
    & \bra v_1, v_2 \ket^\mathfrak{a}_2 = [v_1, v_2], \\
    & \bra v, f \ket^\mathfrak{a}_2 = \mathcal{L}_v f, \\
    & \bra v_1, v_2, v_3\ket^\mathfrak{a}_3 = -\iota_{v_1} \iota_{ v_2 } \iota_{v_3} \omega.
\end{aligned}
\]
All other brackets vanish for degree reasons.
\end{defn}

\begin{defn}\cite[ Definition/Proposition 5.2.1]{Fiorenza2014}
Let \((M,\omega)\) be a pre-2-plectic manifold. The \emph{Courant Lie \(2\)-algebra}, denoted \(\mathfrak{courant}(M,\omega)\), is defined on the graded vector space
\[
\mathfrak{courant}(M, \omega)_0 = \mathfrak{X}(M) \oplus \Omega^1(M), \quad \mathfrak{courant}(M, \omega)_1 = \Omega^0(M),
\]
with the following nontrivial brackets:
\[
\begin{aligned}
    & \bra f \ket^\mathfrak{c}_1 = d f, \\
    & \bra (v , \theta), f \ket^\mathfrak{c}_2 = \frac{1}{2} \iota_v d f, \\
    & \bra (v_1 , \theta_1), (v_2 , \theta_2)\ket^\mathfrak{c}_2 = \left( [v_1, v_2], \mathcal{L}_{v_1}\theta_2 - \mathcal{L}_{v_2}\theta_1 - \frac{1}{2} d(\iota_{v_1}\theta_2 - \iota_{v_2}\theta_1) - \iota_{v_1}\iota_{v_2} \omega \right), 
\end{aligned}
\]

\[
\begin{aligned}
          &\bra (v_1 , \theta_1), (v_2 , \theta_2), (v_3 , \theta_3) \ket^\mathfrak{c}_3 
= -\frac{1}{6} \left(
\langle \bra (v_1 , \theta_1), (v_2 , \theta_2) \ket^\mathfrak{c}_2 , (v_3 , \theta_3) \rangle
+ \langle \bra (v_2 , \theta_2), (v_3 , \theta_3) \ket^\mathfrak{c}_2 , (v_1 , \theta_1) \rangle \right. \\
& \hspace{6.5cm} + \left. \langle \bra (v_3 , \theta_3), (v_1 , \theta_1) \ket^\mathfrak{c}_2 , (v_2 , \theta_2) \rangle
\right)
\end{aligned}
\]

Here, \(\langle \cdot, \cdot \rangle\) denotes the natural symmetric pairing on \(TM \oplus T^*M\), given by
\[
\langle (v_1 , \theta_1), (v_2 , \theta_2) \rangle := \iota_{v_1} \theta_2 + \iota_{v_2} \theta_1.
\]
\end{defn}

These two Lie \(2\)-algebras are related via a sequence of \(L_\infty\)-morphisms.

\begin{pro}\cite[Proposition 5.2.3]{Fiorenza2014}
\label{Atiyah-Courant}
Let \((M,\omega)\) be a pre-2-plectic manifold. Then there exists a natural sequence of \(L_\infty\)-morphisms:
\[
L_\infty(M, \omega) \xrightarrow{\phi} \mathfrak{courant}(M, \omega) \xrightarrow{\psi} \mathfrak{atiyah}(M, \omega),
\]
with the nontrivial components of \(\phi\) given by:
\[
\phi_1(v , \theta) = (v , \theta), \quad \phi_1(f) = f, \quad \phi_2((v_1 , \theta_1), (v_2 , \theta_2)) = -\frac{1}{2} (\iota_{v_1}\theta_2 - \iota_{v_2}\theta_1),
\]
and the components of \(\psi\) defined as:
\[
\psi_1(v , \theta) = v, \quad \psi_1(f) = f, \quad \psi_2((v_1 , \theta_1), (v_2 , \theta_2)) = -\frac{1}{2} (\iota_{v_1}\theta_2 - \iota_{v_2} \theta_1).
\]
\end{pro}

\begin{proof}
The result follows from two key sources. The morphism \(\phi\) is established in \cite[Theorem 7.1]{rogers20132plectic}, while the morphism \(\psi\) is constructed in \cite[Proposition 5.2.3]{Fiorenza2014}.
\end{proof}

%% file: Chapter4.tex
\chapter{Observables of Relative \nplectic Structures}

Recall that a quasi-Hamiltonian \( G \)-space is defined as a manifold equipped with a \( G \)-action, a 2-form \( \omega \), and a moment map \( \Phi \), which satisfy certain compatibility conditions as described in Definition \ref{def:quasiHamiltonianGSpace} (see \cite{10.4310/jdg/1214460860} and \cite{meinrenken2012twisted} for more details about quasi-Hamiltonian \(G\)-spaces). The 2-form \(\omega\) is not closed (\(\mathrm{d}\omega=-\Phi^* \eta\)) nor nondegenerate; however taking into account the Cartan bi-invariant 3-form \(\eta\) that arises naturally on any compact Lie group, we get a closed (Proposition \ref{closed}) and non-degenerate (Theorem \ref{Non-deg}) relative 3-form \((\omega, \eta)\). On the other hand, 2-plectic manifolds are smooth manifolds endowed with a closed and non-degenerate 3-form (see \cite{rogers2011higher}, \cite{rogers20132plectic}, \cite{callies2016homotopy}, \cite{Baez_2009}, \cite{Cantrijn1998}). So quasi-Hamiltonian \(G\)-spaces and 2-plectic manifolds share some similarities, mainly because they are endowed with closed and non-degenerate 3-forms. Following the work of Rogers, a 2-plectic manifold gives rise to a Lie 2-algebra of observables (see \cite{rogers2011higher} and \cite{rogers20132plectic}). So it is natural to ask the following question: Is there a Lie 2-algebra of observables corresponding to a quasi-Hamiltonian \(G\)-space?

In general, it is known from \cite[Theorem 3.14]{rogers2011higher} that an \(n\)-plectic manifold gives rise to an \(L_\infty\)-algebra of observables. This result is even true when we have only a pre-\(n\)-plectic structure (see \cite[Theorem 4.7]{CohomologicalFramework2015}). Now assume that \((N, \, \omega_N)\) is an \(n\)-plectic manifold and \((M,\, \omega_M)\) is a smooth manifold along with an \(n\)-form \(\omega_M\) (not necessarily closed or nondegenerate).  Let \(F: M\to N\) be a smooth map such that 
\(F^*\omega_N=-\mathrm{d}\omega_M\) and 
\[\ker(\omega_M)_m\cap \ker(\mathrm{d}_mF)=\{0\}\quad \text{for all}\quad m\in M,\]

\noindent Using the mapping cone, we get a relative closed and nondegenerate $(n+1)$-form \((\omega_M, \, \omega_N) \).  In this chapter, we prove that the relative closed and non-degenerate \(n\)-form \((\omega_M, \, \omega_N) \) leads to an \( L_\infty \)-algebra of observables (Theorem \ref{Relative L-infinity algebra main thm} and Theorem \ref{theorem:Lie algebra of obs for pre-n-plectic}), and to a Leibniz differential graded 
 algebra (Proposition \ref{proposition: Leibniz dg algebra}), similar to the way \( n \)-plectic manifolds do, as shown in \cite[Theorem 3.14]{rogers2011higher}. As a consequence of this result, a quasi-Hamiltonian \( G \)-space naturally gives rise to a Lie 2-algebra, Atiyah (Theorem \ref{theorem: atiyah lie 2-algebra}) and Courant (Theorem \ref{theorem: courant lie 2-algebra}) Lie 2-algebras.

The construction of \(L_\infty\)-algebras in \(n\)-plectic geometry is essentially based on Cartan's Calculus and Poisson structures. To establish the connection between relative structures and \(L_\infty\)-algebras, we introduce essential concepts in Section  \ref{section:relative cartan magic formulas}, including Relative Cartan's Calculus and Relative Structures. We define key operations such as the Lie derivative, contraction, and differential within the context of relative differential forms and verify that they satisfy Cartan's formulae, similar to those in classical differential calculus. This enables us to introduce relative Hamiltonian forms and vector fields in the relative setting, which is key to explicitly constructing \(L_\infty\)-algebras on a complex consisting of these forms and smooth functions.

\section{Relative Cartan's Calculus and Relative \nplectic Structures} \label{section:relative cartan magic formulas}
We note that some constructions in this section are standard in the literature on relative de Rham cohomology, particularly the cone complex \( \Omega^\bullet(F) = \mathrm{Cone}(F^*) \) and the associated differential. However, the formulation and verification of Cartan calculus identities in this relative setting, as well as the systematic development of relative \( n \)-plectic structures and their associated \( L_\infty \)-algebras, represent a new contribution. To our knowledge, a fully developed theory of relative symplectic or multisymplectic geometry has not previously appeared in this precise form.

  Let \(F: M\to N\) be a smooth map between smooth manifolds. Taking the pullback of the map $F: M\to N$, it follows from Definition \ref{defn map cone} that we can define the algebraic mapping cone of the cochain map $F^*:\Omega^\bullet(N)\to \Omega^\bullet(M)$, denoted $\Omega^\bullet( F):=\mathrm{Cone}^\bullet(F^*)$ so that
 $\displaystyle\Omega^\bullet(F)=\bigoplus_{n\geq 0}\Omega^n(F)$ and 
\begin{eqnarray*}
\Omega^{n}(F)&=&\Omega^{n-1}(M)\oplus\Omega^{n}(N)
\end{eqnarray*}
with differential $\mathrm{d}$, defined by
\begin{eqnarray*}
\mathrm{d}(\alpha,\, \beta)&=&\left(F^*\beta+ \mathrm{d}\alpha, \, -\mathrm{d}\beta\right).
\end{eqnarray*}
We will refer to this differential as the \emph{relative differential} or the \emph{de Rham relative differential}.  
An element of $\Omega^{n}(F)$ will be called a relative $n$-form, and $\Omega^{n}(F)$ is the space of relative $n$-forms.

\begin{defn}
\label{defintion: closed relative n-form}
A relative $n$-form $(\alpha,\beta)\in \Omega^{n}(F)$  is  \emph{closed} if
\[
    \mathrm{d}(\alpha,\beta)=(0,0),
\]
or equivalently:
\begin{equation*}
    (F^*\beta + \mathrm{d}\alpha, -\mathrm{d}\beta) = (0,0).
\end{equation*}

\end{defn}

The space of relative closed \(n\)-forms will be denoted by \(\Omega_{\mathrm{cl}}^n(F)\).   As an example of a relative closed 3-form, we have \( (\omega,\,  \eta)\) from a quasi-Hamiltonian \(G\)-space \( (M,  \omega, \Phi) \) as proved in the next proposition.

\begin{pro}\label{closed}
If \( (M,  \omega, \Phi) \) is a quasi-Hamiltonian \( G \)-space, then the relative 3-form \( (\omega,\,  \eta) \in \Omega^2(M) \oplus \Omega^3(G) = \Omega^3(\Phi) \) is closed.
\end{pro}

\begin{proof}
Indeed, 
\[\mathrm{d}(\omega, \eta)=\left(\Phi^*\eta + \mathrm{d}\omega,\, -\mathrm{d}\eta\right).\]
By the properties of a quasi-Hamiltonian \( G \)-space, we have 
\begin{equation*}
\mathrm{d}\omega = - \Phi^*\eta,
\end{equation*}
which implies 
\begin{equation*}
\Phi^*\eta + \mathrm{d}\omega = 0.
\end{equation*}
By Theorem \ref{eta is closed}, we have \( \mathrm{d}\eta = 0 \). It follows that
\begin{align*}
\mathrm{d}(\omega,\eta) &= \left(\Phi^*\eta + \mathrm{d}\omega, -\mathrm{d}\eta\right) = (0,0).
\end{align*}
Hence, \( (\omega,\, \eta) \) is closed as desired. 
\end{proof}

\begin{defn}\label{def:relative interior and Lie}
  Let $(\alpha,\beta)\in \Omega^{n}(F)$ be a relative  $n$-form, and let $(u, v)\in \mathfrak{X}(M) \times \mathfrak{X}(N)$ be a pair of vector fields. The \emph{ interior multiplication or contraction} of $(\alpha,\beta)$ with $(u, v)$ is the relative $(n-1)$-form $\iota_{(u,v)}(\alpha,\beta)$ defined by
\begin{eqnarray*}
\iota_{(u,v)}(\alpha,\beta)&=&\left(\iota_u\alpha, -\iota_v\beta\right),
\end{eqnarray*}
where \(\iota_u \alpha\) and \(\iota_v \beta\) denote the standard contractions of \(\alpha\) with \(u\) and \(\beta\) with \(v\) respectively.

The \emph{Lie derivative} of $(\alpha,\beta)$ at $(u,v)$ is  the relative $n$-form $\mathcal{L}_{(u,v)}(\alpha,\beta)$ defined by 
\begin{eqnarray*}
\mathcal{L}_{(u,v)}(\alpha,\beta)&=&\left(\mathcal{L}_{u}\alpha,\mathcal{L}_{v}\beta\right),
\end{eqnarray*}
where \(\mathcal{L}_u \alpha\) and \(\mathcal{L}_v \beta\) denote the standard Lie derivatives of \(\alpha\) with respect to \(u\) and \(\beta\) with respect to \(v\), respectively.
\end{defn}

\begin{defn}\label{symplectic vector}
    A pair of vector fields \((u, v)\) is called \emph{multisymplectic} if it preserves \((\omega_M, \omega_N)\): \[\mathcal{L}_{(u, v)}(\omega_M, \omega_N)=0.\]
\end{defn}

\begin{defn}
Let \((u, v)\) and \((a, b)\) be pairs of vector fields in \(\mathfrak{X}(M) \times \mathfrak{X}(N)\). The \emph{Lie bracket} of these pairs, denoted by \( [(u, v), (a, b)] \), is defined as:  
\[
[(u, v), (a, b)] = \left([u, a], [v, b]\right),
\]
where \([u, a]\) and \([v, b]\) are the standard Lie brackets of vector fields on \(M\) and \(N\), respectively.
\end{defn}

This bracket satisfies the usual properties of a Lie bracket: bilinearity, antisymmetry, and the Jacobi identity. Consequently, \(\mathfrak{X}(M) \times \mathfrak{X}(N)\) equipped with this bracket forms a Lie algebra. For simplicity, we denote \(\mathfrak{X}(M) \times \mathfrak{X}(N)\) by \(\mathfrak{X}(F)\).

\begin{defn}\label{relative nondegenerate}
 
 Let \( F: M \to N \) be a smooth map between manifolds. A relative \( k \)-form \( (\alpha, \beta) \in \Omega^k(F) := \Omega^{k-1}(M) \oplus \Omega^k(N) \) is said to be \emph{nondegenerate} if for all \( x \in M \), and for all pairs \( (u, v) \in T_x M \times T_{F(x)} N \) satisfying \( v = dF_x(u) \), the condition
\[
\iota_{(u, v)}(\alpha, \beta) = \left( \iota_u \alpha(x),\, \iota_v \beta(F(x)) \right) = (0, 0)
\]
implies \( (u, v) = (0, 0) \).
	
\end{defn}

\bigskip 

\begin{defn}\label{def:relativeMultisymplecticManifold}
	Let \(F: M\to N\) be a smooth map between smooth manifolds. A relative \emph{$n$-plectic (or multisymplectic) structure}  is a closed nondegenerate  $(n+1)$-form $(\alpha, \beta)\in\Omega^{n+1}(F)$. A relative $n$-plectic manifold is the data of   
 \begin{itemize}
     \item a \(F: M\to N\) be a smooth map between smooth manifolds, and
     \item  a  relative $n$-plectic structure $(\alpha, \beta)\in\Omega^{n+1}(F)$.
 \end{itemize}

 \end{defn}

\bigskip 

\begin{center}
\fbox{
\parbox[c]{12.6cm}{\begin{center}
 From now on,  \(F: M\to N\) will be a smooth map between smooth manifolds along with a closed non-degenerate relative  $(n+1)$-form $(\omega_M, \omega_N)\in\Omega^{n+1}(F)$.
\end{center}
}}
 \end{center}

 \bigskip 

\begin{thm}[Nondegeneracy of the Relative Structure]\label{Non-deg}
Let \( F: M \to N \) be a smooth map, and let \( (\omega_M, \omega_N) \in \Omega^{n+1}(F) \) be a relative \( (n+1) \)-form such that:
\[
\ker(\omega_M)_m \cap \ker(dF_m) = \{0\} \quad \text{for all } m \in M,
\]
and \( \omega_N \in \Omega^{n+1}(N) \) is nondegenerate in the usual \( n \)-plectic sense.

Then the relative form \( (\omega_M, \omega_N) \) is nondegenerate in the sense of Definition~\ref{relative nondegenerate}.
\end{thm}

\begin{proof}
Let \( m \in M \), and suppose \( (u, v) \in T_m M \times T_{F(m)} N \) with \( v = dF_m(u) \), and
\[
\iota_{(u, v)}(\omega_M, \omega_N) = ( \iota_u \omega_M,\, \iota_v \omega_N ) = (0, 0).
\]
Since \( \iota_v \omega_N = 0 \) and \( \omega_N \) is nondegenerate, it follows that \( v = 0 \).

From \( v = dF_m(u) = 0 \), we deduce \( u \in \ker(dF_m) \). Also, \( \iota_u \omega_M = 0 \), so \( u \in \ker(\omega_M)_m \). Hence, \( u \in \ker(\omega_M)_m \cap \ker(dF_m) = \{0\} \), which implies \( u = 0 \).

Therefore, \( (u, v) = (0, 0) \), and the relative form is nondegenerate as claimed.
\end{proof}

The following theorem states that the relative 3-form $(\omega, \eta)$ arising from a quasi-Hamiltonian \(G\)-space is nondegenerate.
\begin{thm}\label{Non-degeneracy general}
Let \((M, \omega, \Phi)\) be a quasi-Hamiltonian \(G\)-space. Then the relative 3-form \((\omega, \eta)\) is non-degenerate in the sense of Definition~\ref{relative nondegenerate}.
\end{thm}

\begin{proof}
By the definition of a quasi-Hamiltonian \(G\)-space, we have the condition \( \ker (\omega_m) \cap \ker(\mathrm{d}_m \Phi) = \{0\} \) for all \(m\in M\). Moreover, the bi-invariant 3-form \(\eta\) is nondegenerate, as shown in Proposition \ref{nondegenerate eta}. 

Thus, by Theorem \ref{Non-deg}, it follows that the relative 3-form \((\omega, \eta)\) is nondegenerate.
\end{proof}

When the relative \((n-1)\)-form \((\alpha,\, \beta)\) is Hamiltonian, we expect the corresponding Hamiltonian vector field to be unique, similar to the case of \(n\)-plectic geometry. This is indeed the case.

\begin{lem}\label{Uniqueness of hamiltonian vector}
   The pair of Hamiltonian vectors associated with a relative Hamiltonian \(n\)-form \((\alpha,\, \beta)\) is unique. 
\end{lem}

\begin{proof}
Suppose \((u, v)\) and \((x, y)\) are two Hamiltonian vector fields corresponding to \((\alpha,\, \beta)\). Then, we have:
\[
\mathrm{d}(\alpha,\, \beta) = -\iota_{(u,\, v)}(\omega_M,\, \omega_N) = -\iota_{(x,\, y)}(\omega_M,\, \omega_N),
\]
which implies:
\[
\iota_{(u,\, v)}(\omega_M,\, \omega_N) = \iota_{(x,\, y)}(\omega_M,\, \omega_N),
\]
or equivalently:
\[
\iota_u \omega_M = \iota_x \omega_M \quad \text{and} \quad \iota_v \omega_N = \iota_y \omega_N.
\]

By the nondegeneracy of \(\omega_N\), it follows that \(v = y\).  Also,  \( v = \mathrm{d}F(u) \)  and \( y = \mathrm{d}F(x) \) since \(u \sim_F v\) and \(x \sim_F y\). The equality becomes 
\(\mathrm{d}F(u)=\mathrm{d}F(x)\) so that \(u-x\in \ker(\omega_M)_m\cap \ker(\mathrm{d}_mF)=\{0\}\). Therefore, \(u-x=0\), so that \(u=x\). Hence, \((u, v) = (x, y)\), showing that the Hamiltonian vector field is unique as desired.
\end{proof}

\begin{defn} \label{hamiltonian}
 A relative $(n-1)$-form $(\alpha,\,\beta)\in \Omega^{n-1}(F)$ is \emph{ Hamiltonian} if there exists a pair \((u,v) \in \mathfrak{X}(F)\) of \(F\)-related vectors $u \in \mathfrak{X}(M)$ and $v \in \mathfrak{X}(N)$ such that
 \[
\mathrm{d}(\alpha,\, \beta)=-\iota_{(u,\, v)}(\omega_M,\, \omega_N),
\]
or equivalently,
\[
\left(F^*\beta+\mathrm{d}\alpha, \, -\mathrm{d}\beta\right)=(-\iota_{u} \omega_M, \, \iota_{v} \omega_N).
\]

\noindent In this case, we say that $(u, v)$ is a pair of \emph{ Hamiltonian vector fields} corresponding to the relative $(n-1)$-form $(\alpha,\beta)$. We denote the set of such Hamiltonian relative $(n-1)$-forms by $\Omega_{\mathrm{Ham}}^{n-1}(F)$, and the set of Hamiltonian pairs of vector fields by $\mathfrak{X}_{\mathrm{Ham}}(F)$. Both of these sets are vector spaces.
\end{defn}

The brackets, Lie derivative, and interior product, as defined above, satisfy Cartan's formulae in differential geometry as expected.

\begin{pro}[Relative Cartan's formulas]\label{relative cartan magic formula}
Let \(F: M\to N\) be a smooth map between smooth manifolds. Let  $(u, v), \, (a,b)\in \mathfrak{X}(F)$. Then 
\begin{enumerate}
\item $\mathcal{L}_{(u,v)}\circ\mathcal{L}_{(a,b)}-\mathcal{L}_{(a,b)}\circ\mathcal{L}_{(u,v)}=\mathcal{L}_{[(u,v),(a,b)]}$.
\item $\iota_{(u,v)}\circ\iota_{(a,b)}+\iota_{(a,b)}\circ\iota_{(u,v)}=0$.
\item If $u$ is $F$-related to $v$, then $$\mathrm{d}\circ \mathcal{L}_{(u,v)}- \mathcal{L}_{(u,v)}\circ \mathrm{d}=0.$$
\item $\mathcal{L}_{(u,v)}\circ \iota_{(a,b)}-\iota_{(a,b)}\circ \mathcal{L}_{(u,v)}=\iota_{[(u,v),(a,b)]}$

\item If $u$ is $F$-related to $v$, then $$\iota_{(u,v)}\circ \mathrm{d}+ \mathrm{d}\circ\iota_{(u,v)}= \mathcal{L}_{(u,v)}.$$
\end{enumerate}

\end{pro}

\begin{proof}
\begin{enumerate}

\item We have:
    \begin{align*}
        \left(\mathcal{L}_{(u,v)} \circ \mathcal{L}_{(a,b)} - \mathcal{L}_{(a,b)} \circ \mathcal{L}_{(u,v)}\right)(\alpha,\beta) &= \left(\mathcal{L}_u\mathcal{L}_a \alpha, \, \mathcal{L}_v\mathcal{L}_b\beta\right)-\left(\mathcal{L}_a\mathcal{L}_u \alpha, \, \mathcal{L}_b\mathcal{L}_v\beta\right)\\
&=\left(\left(\mathcal{L}_u\mathcal{L}_a-\mathcal{L}_a\mathcal{L}_u \right) \alpha,\, \left(\mathcal{L}_v\mathcal{L}_b-\mathcal{L}_b\mathcal{L}_v\right)\beta\right)\\
&=\left(\mathcal{L}_{[u,a]}\alpha, \mathcal{L}_{[v,b]}\beta\right)\\
&=\mathcal{L}_{\left([u,a], [v,b]\right)}\left(\alpha, \beta\right)\\
&=\mathcal{L}_{[(u,v),(a,b)]}(\alpha,\beta).
    \end{align*}

Hence, $\mathcal{L}_{(u,v)}\circ\mathcal{L}_{(a,b)}-\mathcal{L}_{(a,b)}\circ\mathcal{L}_{(u,v)}=\mathcal{L}_{[(u,v),(a,b)]}$ as desired.

\item We have:
\begin{eqnarray*}
\left(\iota_{(u,v)}\circ\iota_{(a,b)}+\iota_{(a,b)}\circ\iota_{(u,v)}\right)(\alpha,\beta)&=&\iota_{(u,v)}\left(\iota_a\alpha,-\iota_b\beta\right)+\iota_{(a,b)}\left(\iota_u\alpha,-\iota_v\beta\right)\\
&=&\left(\iota_u\iota_a\alpha,-\iota_v(-\iota_b\beta)\right)+\left(\iota_a\iota_u\alpha,-\iota_b(-\iota_v\beta)\right)\\
&=&\left(\iota_u\iota_a\alpha,\iota_v\iota_b\beta\right)+\left(\iota_a\iota_u\alpha,\iota_b\iota_v\beta\right)\\
&=&\left(\left(\iota_u\iota_a+\iota_a\iota_u\right)\alpha,\left(\iota_v\iota_b+\iota_b\iota_v\right)\beta\right)\\
&=&(0,0).
\end{eqnarray*}

\item We have:
\begingroup
\small 
\begin{align*}
\left(\mathrm{d}\circ \mathcal{L}_{(u,v)} - \mathcal{L}_{(u,v)}\circ \mathrm{d}\right)(\alpha,\beta) 
&= \mathrm{d}\left(\mathcal{L}_u\alpha, \mathcal{L}_v\beta\right) - \mathcal{L}_{(u,v)}\left(F^*\beta+\mathrm{d}\alpha, -\mathrm{d}\beta\right) \\
&= \left(F^*\mathcal{L}_v\beta + \mathrm{d}\mathcal{L}_u\alpha, -\mathrm{d}\mathcal{L}_v\beta\right) 
  - \left(\mathcal{L}_u\left(F^*\beta + \mathrm{d}\alpha\right), \mathcal{L}_v(-\mathrm{d}\beta)\right) \\
&= \left(F^*\mathcal{L}_v\beta + \mathrm{d}\mathcal{L}_u\alpha, -\mathrm{d}\mathcal{L}_v\beta\right) 
  - \left(\mathcal{L}_u F^*\beta + \mathcal{L}_u\mathrm{d}\alpha, -\mathcal{L}_v\mathrm{d}\beta\right) \\
&= \left(F^*\mathcal{L}_v\beta + \mathrm{d}\mathcal{L}_u\alpha - \mathcal{L}_u F^*\beta - \mathcal{L}_u\mathrm{d}\alpha, 
     -\mathrm{d}\mathcal{L}_v\beta + \mathcal{L}_v\mathrm{d}\beta\right) \\
&= \left(0, 0\right).
\end{align*}
\endgroup

Hence, $\mathrm{d}\circ \mathcal{L}_{(u,v)}- \mathcal{L}_{(u,v)}\circ \mathrm{d}=0$ as desired. 


\item 
\begin{eqnarray*}
\left(\mathcal{L}_{(u,v)}\circ \iota_{(a,b)}-\iota_{(a,b)}\circ \mathcal{L}_{(u,v)}\right)(\alpha,\beta)&=&\mathcal{L}_{(u,v)}\left(\iota_a\alpha,-\iota_b\beta\right)-\iota_{(a,b)}\left(\mathcal{L}_u\alpha,\mathcal{L}_v\beta\right)\\
&=&\left(\mathcal{L}_u\iota_a\alpha,-\mathcal{L}_v\iota_b\beta\right)-\left(\iota_a\mathcal{L}_u\alpha,-\iota_b\mathcal{L}_v\beta\right)\\
&=&\left(\mathcal{L}_u\iota_a\alpha-\iota_a\mathcal{L}_u\alpha,-\left(\mathcal{L}_v\iota_b\beta-\iota_b\mathcal{L}_v\beta\right)\right)\\
&=&\left(\iota_{[u,a]}\alpha,-\iota_{[v,b]}\beta\right)\\
&=&\iota_{([u,a],[v,b])}(\alpha,\beta)\\
&=&\iota_{[(u,v),(a,b)]}(\alpha,\beta).
\end{eqnarray*}


\item We have:
\begin{eqnarray*}
\left(\iota_{(u,v)}\circ \mathrm{d}+ \mathrm{d}\circ\iota_{(u,v)}\right)(\alpha,\beta)&=&\iota_{(u,v)}\left( F^*\beta+\mathrm{d}\alpha,-\mathrm{d}\beta\right)+ \mathrm{d}\left(\iota_u\alpha,-\iota_v\beta\right)\\
&=&\left( \iota_u\left(F^*\beta+\mathrm{d}\alpha\right),-\iota_v(-\mathrm{d}\beta)\right)\\
&& + \left(F^*\left(-\iota_v\beta\right)+ \mathrm{d}\iota_u\alpha,-\mathrm{d}(-\iota_v\beta)\right)\\
&=&\left( \iota_u F^*\beta+\iota_u\mathrm{d}\alpha,\iota_v\mathrm{d}\beta\right)+ \left(-F^*\iota_v\beta+ \mathrm{d}\iota_u\alpha,\mathrm{d}\iota_v\beta\right)\\
&=&\left( \iota_u F^*\beta+\iota_u\mathrm{d}\alpha-F^*\iota_v\beta+ \mathrm{d}\iota_u\alpha,\iota_v\mathrm{d}\beta+\mathrm{d}\iota_v\beta\right)\\
&=&\left( \iota_u\mathrm{d}\alpha+ \mathrm{d}\iota_u\alpha,\iota_v\mathrm{d}\beta+\mathrm{d}\iota_v\beta\right)\\
&=&\left(\mathcal{L}_u\alpha,\mathcal{L}_v\beta\right)\\
 &=& \mathcal{L}_{(u,v)}(\alpha,\beta).
\end{eqnarray*}
\end{enumerate}
This completes the proof of the proposition.
\end{proof}

Similar to the multisymplectic case, one can define bracket structures on observables in the \emph{relative} setting. Given a relative \(n\)-plectic structure, one can construct analogues of the semi-bracket and hemi-bracket.

\begin{defn}
\label{bracket_def}
  Let \((f,\alpha), (g,\beta) \in \Omega_{\mathrm{Ham}}^{n-1}(F)\). The \emph{relative hemi-bracket} \(\{(f,\alpha), (g,\beta)\}_\mathrm{h}\) is the relative \((n-1)\)-form given by
\[  
\{(f,\alpha), (g,\beta)\}_{\mathrm{h}} = \mathcal{L}_{(u, \, v)}(g,\beta)
\]
where \((u,\, v)\) is any pair of Hamiltonian vector fields corresponding to \((f, \, \alpha)\).

The \emph{relative semi-bracket} \(\{(f,\alpha), (g,\beta)\}_\mathrm{s}\) is the relative \((n-1)\)-form given by
\[  
\{(f,\alpha), (g,\beta)\}_\mathrm{s} =\iota_{(u_1,\, v_1)}\iota_{(u_2,\, v_2)}\left(\omega_M, \omega_N\right)= \left( \iota_{u_{2}}\iota_{u_{1}}\omega_M,\, \iota_{v_{2}}\iota_{v_{1}}\omega_N \right)
\]
where \((u_1,\, v_1)\), and \((u_2,\, v_2)\) are any pairs of Hamiltonian vector fields corresponding to \((f, \, \alpha)\), and  \((g, \, \beta)\) respectively.
\end{defn}

The following proposition shows that the \emph{relative semi-bracket} satisfies properties analogous to those of the semi-bracket in the absolute multisymplectic setting: it is skew-symmetric and defines a Hamiltonian form. In particular, this bracket equips the space of relative Hamiltonian \((n{-}1)\)-forms with a structure compatible with the underlying relative \(n\)-plectic geometry.

\begin{pro}\label{proposition:bracket}
Let \((f,\alpha), (g,\beta) \in \Omega_{\mathrm{Ham}}^{n-1}(F)\) be relative Hamiltonian \((n{-}1)\)-forms, and let \((u_1, v_1)\) and \((u_2, v_2)\) be corresponding Hamiltonian vector fields. Then the semi-bracket \(\{\,\cdot\,,\,\cdot\,\}_{\mathrm{s}}\) satisfies the following properties:
\begin{enumerate}
    \item \textbf{Skew-symmetry:}
    \[
    \{(f, \alpha), (g, \beta)\}_{\mathrm{s}} = -\{(g, \beta), (f, \alpha)\}_{\mathrm{s}}.
    \]

    \item \textbf{Hamiltonianity:} The semi-bracket of Hamiltonian forms is again Hamiltonian, with associated Hamiltonian vector field given by the Lie bracket of the corresponding vector fields:
    \[
    \mathrm{d} \{(f, \alpha), (g, \beta)\}_{\mathrm{s}} = -\iota_{[(u_1, v_1), (u_2, v_2)]}(\omega_M, \omega_N),
    \]
    and hence
    \[
    v_{\{(f, \alpha), (g, \beta)\}_{\mathrm{s}}} = [(u_1, v_1), (u_2, v_2)] = \left([u_1, u_2], [v_1, v_2]\right).
    \]
\end{enumerate}
\end{pro}

\begin{proof}
\begin{enumerate}
\item Let \((u_1, v_1)\) and \((u_2, v_2)\) be Hamiltonian vectors corresponding to  $(f, \alpha)$ and $(g, \beta)$ respectively. We have:
\begin{align*}
\{(f,\alpha), (g,\beta)\}_\mathrm{s} &= \left( \iota_{u_2}\iota_{u_1}\omega_M, \, \iota_{v_2}\iota_{v_1}\omega_N \right)\\
&= \left( -\iota_{u_1} \iota_{u_2}\omega_M,\, -\iota_{v_1}\iota_{v_2}\omega_N \right)\\
&= -\left( \iota_{u_1} \iota_{u_2}\omega_M, \, \iota_{v_1}\iota_{v_2}\omega_N \right)\\
&= -\{(g,\beta), (f,\alpha)\}_\mathrm{s}.
\end{align*}

\item 
\begin{align*}
\mathrm{d}\{(f,\alpha), (g,\beta)\}_\mathrm{s} &= \mathrm{d} \left( \iota_{u_2}\iota_{u_1}\omega_M, \, \iota_{v_2}\iota_{v_1}\omega_N \right)\\
&= \left( F^*\left(\iota_{v_2}\iota_{v_1}\omega_N\right) + \mathrm{d}\left(\iota_{u_2}\iota_{u_1}\omega_M\right), \, -\mathrm{d}\left(\iota_{v_2}\iota_{v_1}\omega_N\right) \right)
\end{align*}

But
\begin{align*}
\mathrm{d}\left(\iota_{v_{2}}\iota_{v_{1}}\omega_N\right) &= \left(\mathcal{L}_{v_{2}}-\iota_{v_2}\mathrm{d}\right) \left(\iota_{v_{1}}\omega_N\right)\\
&= \mathcal{L}_{v_{2}}\left(\iota_{v_{1}}\omega_N\right) - \iota_{v_2}\mathrm{d} \left(\iota_{v_{1}}\omega_N\right)\\
&= \mathcal{L}_{v_{2}}\left(\iota_{v_{1}}\omega_N\right) - \iota_{v_2}\mathrm{d} \left(-\mathrm{d} \beta\right) \quad\quad \text{since \(\iota_{v_1}\omega_N=-\mathrm{d} \beta\)}\\
&= \mathcal{L}_{v_{2}}\iota_{v_{1}}\omega_N \quad\quad \text{since \(\mathrm{d}^2=0\)}\\
&= \iota_{[v_2,v_1]}\omega_N + \iota_{v_1}\mathcal{L}_{v_2}\omega_N \quad\quad \text{by Cartan's magic formula}\\
&= \iota_{[v_2,v_1]}\omega_N + \iota_{v_1}\left(\mathrm{d}\iota_{v_2}\omega_N + \iota_{v_2}\mathrm{d}\omega_N\right)\\
&= \iota_{[v_2,v_1]}\omega_N + \iota_{v_1}\left(\mathrm{d}\iota_{v_2}\omega_N + 0\right) \quad\quad \text{since \(\mathrm{d}\omega_N=0\)}\\
&= \iota_{[v_2,v_1]}\omega_N - \iota_{v_1}\left(\mathrm{d}^2\beta\right) \quad\quad \text{since \(\iota_{v_2}\omega_N=-\mathrm{d} \beta\)}\\
&= \iota_{[v_2,v_1]}\omega_N \quad\quad \text{since \(\mathrm{d}^2=0\)}\\
&= -\iota_{[v_1,v_2]}\omega_N
\end{align*}

and

\begin{align*}
\mathrm{d}\left(\iota_{u_{2}}\iota_{u_{1}}\omega_M\right) &= \left(\mathcal{L}_{u_{2}} - \iota_{u_2}\mathrm{d}\right) \left(\iota_{u_{1}}\omega_M\right)\\
&= \mathcal{L}_{u_{2}}\left(\iota_{u_{1}}\omega_M\right) - \iota_{u_2}\mathrm{d}\left(\iota_{u_{1}}\omega_M\right)\\
&= \mathcal{L}_{u_{2}}\left(\iota_{u_{1}}\omega_M\right) + \iota_{u_2}\mathrm{d}\left(F^*\omega_N + \mathrm{d} \alpha\right) \quad \text{since \(\iota_{u_1}\omega_M = -F^*\beta - \mathrm{d} \alpha\)}\\
&= \mathcal{L}_{u_{2}}\left(\iota_{u_{1}}\omega_M\right) + \iota_{u_2}\mathrm{d}F^*\beta \quad \text{since \(\mathrm{d}^2 = 0\)}\\
&= \left(\iota_{[u_2, u_1]} + \iota_{u_1}\mathcal{L}_{u_2}\right)\omega_M + \iota_{u_2}\mathrm{d} F^*\beta \quad \text{by Cartan's magic formula}\\
&= \iota_{[u_2, u_1]}\omega_M + \iota_{u_1}\mathcal{L}_{u_2}\omega_M + \iota_{u_2}\mathrm{d} F^*\beta\\
&= \iota_{[u_2, u_1]}\omega_M + \iota_{u_1}\left(\mathrm{d}\iota_{u_2}\omega_M + \iota_{u_2}\mathrm{d}\omega_M\right) + \iota_{u_2}\mathrm{d} F^*\beta\\
&= \iota_{[u_2, u_1]}\omega_M + \iota_{u_1}\mathrm{d}\iota_{u_2}\omega_M + \iota_{u_1}\iota_{u_2}\mathrm{d}\omega_M + \iota_{u_2}\mathrm{d} F^*\beta\\
&= \iota_{[u_2, u_1]}\omega_M + \iota_{u_1}\mathrm{d}\left(-F^*\delta - \mathrm{d} g\right) - \iota_{u_1}\iota_{u_2}\mathrm{d}F^*\omega_N + \iota_{u_2}\mathrm{d} F^*\beta\\
&= \iota_{[u_2, u_1]}\omega_M - \iota_{u_1}\mathrm{d} F^*\delta - \iota_{u_1}\mathrm{d}^2 g - \iota_{u_1}\iota_{u_2}\mathrm{d} F^*\omega_N + \iota_{u_2}\mathrm{d} F^*\beta\\
&= \iota_{[u_2, u_1]}\omega_M - \iota_{u_1} F^*\mathrm{d}\delta - \iota_{u_1}\iota_{u_2}\mathrm{d} F^*\omega_N + \iota_{u_2} F^*\mathrm{d}\beta\\
&= \iota_{[u_2, u_1]}\omega_M - \iota_{u_1} F^*\left(-\iota_{v_2}\omega_N\right) + \iota_{u_1}\iota_{u_2}\mathrm{d} F^*\omega_N - \iota_{u_2} F^*\left(-\iota_{v_1}\omega_N\right)\\
&= \iota_{[u_2, u_1]}\omega_M + \iota_{u_1} F^*\iota_{v_2}\omega_N - \iota_{u_1}\iota_{u_2} F^*\omega_N - \iota_{u_2} F^*\iota_{v_1}\omega_N\\
&= \iota_{[u_2, u_1]}\omega_N + \iota_{u_1}\iota_{u_2} F^*\omega_N - \iota_{u_1}\iota_{u_2} F^*\omega_N - \iota_{u_2} F^*\iota_{v_1}\omega_N\\
&= \iota_{[u_2, u_1]}\omega_M - \iota_{u_2}\iota_{u_1} F^*\omega_N\\
&= -\iota_{[u_1, u_2]}\omega_M - \iota_{u_2}\iota_{u_1} F^*\omega_N.
\end{align*}
Hence,
\[
\mathrm{d}\left(\iota_{u_{2}}\iota_{u_{1}}\omega_M\right) + \iota_{u_2}\iota_{u_1} F^*\omega_N = -\iota_{[u_1, u_2]}\omega_M.
\]

Therefore,
\begin{align*}
\mathrm{d}\{(f,\alpha), (g,\beta)\}_\mathrm{s} &= \mathrm{d} \left( \iota_{u_{2}}\iota_{u_{1}}\omega_M, \iota_{v_{2}}\iota_{v_{1}}\omega_N\right)\\
&= \left(F^*\left(\iota_{v_{2}}\iota_{v_{1}}\omega_N\right) + \mathrm{d}\left(\iota_{u_{2}}\iota_{u_{1}}\omega_M\right), -\mathrm{d}\left(\iota_{v_{2}}\iota_{v_{1}}\omega_N\right)\right)\\
&= \left(- \iota_{[u_1, u_2]}\omega_M, \iota_{[v_1, v_2]}\omega_N\right)\\
&= \left( -\iota_{[u_1, u_2]}\omega_M, \iota_{[v_1, v_2]}\omega_N\right)\\
&= -\left( \iota_{[u_1, u_2]}\omega_M, -\iota_{[v_1, v_2]}\omega_N\right)\\
&= -\iota_{[(u_1, v_1), (u_2, v_2)]}(\omega_M, \omega_N).
\end{align*}
\end{enumerate}

\end{proof}

\begin{lem}\label{LocHam}
Let \((f, \beta) \in \Omega_{\mathrm{Ham}}^{n-1}(F)\) be a relative Hamiltonian form with corresponding pair of vector fields \((u, v)\), where \(u \in \mathfrak{X}(M)\), \(v \in \mathfrak{X}(N)\). Then the pair \((u, v)\) preserves the relative multisymplectic structure \((\omega_M, \omega_N)\); that is,
\[
\mathcal{L}_{(u, v)}(\omega_M, \omega_N) := \left( \mathcal{L}_u \omega_M,\, \mathcal{L}_v \omega_N \right) = (0, 0).
\]
\end{lem}

\begin{proof}
By Proposition \ref{relative cartan magic formula}, we have
\[
\mathcal{L}_{(u, v)}(\omega_M, \omega_N)  = \mathrm{d} \iota_{(u, v)} (\omega_M,\omega_N) + \iota_{(u, v)} \mathrm{d} (\omega_M,\omega_N).
\]
\noindent
Since $(u, v)$ is a Hamiltonian vector field corresponding to  $(f,\beta)$, it satisfies:
\[
\iota_{(u, v)} (\omega_M,\omega_N) =- \mathrm{d}(f, \beta).
\]
\noindent
Since $(\omega_M, \omega_N)$ is closed, we have
\[
\mathrm{d} (\omega_M,\omega_N) = 0
\]
\noindent
so that
\[
\mathcal{L}_{u, v}(\omega_M, \omega_N)  = \mathrm{d}(-\mathrm{d}(f, \beta) )=-\mathrm{d}^2 (f, \beta) = (0,0)
\]
 as desired.
\end{proof}

The (semi)-bracket and the hemi-bracket are related by the following proposition.

\begin{pro}\label{hemi bracket-semi bracket}
Let \((f,\alpha), (g,\beta) \in \Omega_{\mathrm{Ham}}^{n-1}(F)\) be relative Hamiltonian \((n{-}1)\)-forms, with corresponding Hamiltonian vector fields \((u_1, v_1)\) and \((u_2, v_2)\), respectively. Then the hemi-bracket and semi-bracket are related by
\[
\{(f,\alpha), (g,\beta)\}_{\mathrm{h}} = \{(f,\alpha), (g,\beta)\}_{\mathrm{s}} + \mathrm{d} \iota_{(u_1, v_1)}(g, \beta),
\]
where \(\iota_{(u_1, v_1)}(g, \beta) := \left( \iota_{u_1} g,\, \iota_{v_1} \beta \right)\) and \(\mathrm{d}\) denotes the relative de Rham differential.
\end{pro}

\begin{proof}
We have:
\begin{align*}
\{(f,\alpha), (g,\beta)\}_\mathrm{h} &=  \mathcal{L}_{(u_1, \, v_1)}(g,\beta)\\
&=\left( \iota_{(u_1,v_1)}\mathrm{d} + \mathrm{d}\iota_{(u_1,v_1)}\right)(g,\beta)\\
&= \iota_{(u_1,v_1)}\mathrm{d}(g,\beta) + \mathrm{d}\iota_{(u_1,v_1)}(g,\beta)\\
&= -\iota_{(u_1,v_1)}\iota_{(u_2,v_2)}(\omega_M,\omega_N) + \mathrm{d}\iota_{(u_1,v_1)}(g,\beta)\\
&= \{(f,\alpha), (g,\beta)\}_\mathrm{s} + \mathrm{d}\iota_{(u_1,v_1)}(g,\beta).
\end{align*}
\end{proof}

We observe that the hemi-bracket and the semi-bracket differ by the term \( \mathrm{d}\iota_{(v_\alpha, v_\beta)}(\gamma, \delta) \), where \((v_\alpha, v_\beta)\) are the associated Hamiltonian vector fields and \((\gamma, \delta)\) the second Hamiltonian form. Consequently, the two brackets differ by an exact form.

\begin{cor}\label{derivative of hemi bracket}
Let \((f,\alpha), (g,\beta) \in \Omega_{\mathrm{Ham}}^{n-1}(F)\) be relative Hamiltonian \((n{-}1)\)-forms, with corresponding Hamiltonian vector fields \((u_1, v_1)\) and \((u_2, v_2)\), respectively. Then the exterior differential of the hemi-bracket satisfies
\[
\mathrm{d} \{(f,\alpha), (g,\beta)\}_{\mathrm{h}} = -\iota_{[(u_1, v_1),\, (u_2, v_2)]} (\omega_M, \omega_N).
\]
In particular, the hemi-bracket of relative Hamiltonian forms is itself a Hamiltonian \((n{-}1)\)-form, with associated Hamiltonian vector field given by the Lie bracket \([ (u_1, v_1), (u_2, v_2) ]\).
\end{cor}

\begin{proof}
    By Proposition \ref{hemi bracket-semi bracket}, we have 
    \[
\{(f,\alpha), (g,\beta)\}_\mathrm{h} = \{(f,\alpha), (g,\beta)\}_\mathrm{s} + \mathrm{d}\iota_{(u_1,v_1)}(g,\beta)
\]
 so that
 \begin{align*}
     \mathrm{d}\{(f,\alpha), (g,\beta)\}_\mathrm{h}& = \mathrm{d}\{(f,\alpha), (g,\beta)\}_\mathrm{s} + \mathrm{d}^2\iota_{(u_1,v_1)}(g,\beta)\\
     &=\mathrm{d}\{(f,\alpha), (g,\beta)\}_\mathrm{s}\qquad \text{since \(\mathrm{d}^2=0\)}\\
     &=-\iota_{[(u_1, v_1), (u_2, v_2)]}(\omega_M, \omega_N) \quad \text{by Proposition \ref{proposition:bracket}}.
 \end{align*}
\end{proof}

Similar to the semi-bracket in the absolute multisymplectic setting, the \emph{relative semi-bracket} satisfies the Jacobi identity only up to an exact form. The following proposition makes this precise.

\begin{pro}\label{Jacobi-identity-up-toexact-form}
Let \((f, \alpha),\, (g, \beta),\, (k, \gamma) \in \Omega^{n-1}_{\mathrm{Ham}}(F)\) be relative Hamiltonian \((n{-}1)\)-forms, with corresponding Hamiltonian vector fields \((u_1, v_1),\, (u_2, v_2),\, (u_3, v_3)\), respectively. 

Then the relative semi-bracket \( \{\cdot, \cdot\}_{\mathrm{s}} \) satisfies the Jacobi identity up to an exact form:
\begin{align*}
&\{(f, \alpha), \{(g, \beta), (k, \gamma)\}_{\mathrm{s}}\}_{\mathrm{s}} 
- \{\{(f, \alpha), (g, \beta)\}_{\mathrm{s}}, (k, \gamma)\}_{\mathrm{s}}  - \{(g, \beta), \{(f, \alpha), (k, \gamma)\}_{\mathrm{s}}\}_{\mathrm{s}} \\
&= \mathrm{d} J\left((f, \alpha), (g, \beta), (k, \gamma)\right),
\end{align*}
where the \emph{Jacobiator} is given by
\[
J\left((f, \alpha), (g, \beta), (k, \gamma)\right) := 
-\iota_{(u_1, v_1)} \iota_{(u_2, v_2)} \iota_{(u_3, v_3)} (\omega_M, \omega_N).
\]
\end{pro}

\begin{proof}
Let \((u_1, v_1)\) and \((u_2, v_2)\) be Hamiltonian vector fields corresponding to  $(f, \alpha)$ and $(g, \beta)$ respectively.
We have:
\begin{align*}
\{(f, \alpha), (g, \beta)\}_\mathrm{s} &=\left(\iota_{u_2} \iota_{u_1} \omega_M, \, \iota_{v_2} \iota_{v_1} \omega_N\right)\\
&=-\left(\iota_{u_1} \iota_{u_2}  \omega, \, \iota_{v_1} \iota_{v_2} \eta\right)\\
&=-\left(\iota_{u_1} (-F^*\beta-dg), \, \iota_{v_1} (-d \beta)\right)\\
&=\left(\iota_{u_1} F^*\beta+\iota_{u_1}dg, \, \iota_{v_1} d \beta\right)\\
&=\iota_{(u_1, v_1)}\left( F^*\beta+ dg, \,  -d \beta\right)\\
&=\iota_{(u_1, v_1)} \mathrm{d}\left( g, \,  \beta \right)\\
\end{align*}

Now, we can calculate the three terms involved in the Jacobi identity as follows:
\begin{align*}
\{\{(f, \alpha), (g, \beta)\}, (k, \gamma)\}_\mathrm{s} &=\iota_{v_{\{(f, \alpha), (g, \beta)\}}} d (k, \gamma)=\iota_{[(u_1, v_1), (u_2, v_2)]} d (k, \gamma),
\end{align*}
\begin{align*}
\{(g, \beta), \{(f, \alpha), (k, \gamma)\}_\mathrm{s}\}_\mathrm{s} &=\iota_{(u_2, v_2)}\mathrm{d}\{(f, \alpha), (k, \gamma)\}_\mathrm{s}=\iota_{(u_2, v_2)}\mathrm{d}\iota_{(u_1, v_1)} \mathrm{d}(k, \gamma)
\end{align*}
and 
\begin{align*}
\{(f, \alpha), \{(g, \beta), (k, \gamma)\}_\mathrm{s}\}_\mathrm{s} &=\iota_{(v_f, v_\alpha)} \mathrm{d} \{(g, \beta), (k, \gamma)\}_\mathrm{s}=\iota_{(u_1, v_1)} \mathrm{d} \iota_{(u_2, v_2)} \mathrm{d} (k, \gamma).\\
\end{align*}

It follows that
\begingroup
\small
\begin{align*}
&\{(f, \alpha), \{(g, \beta), (k, \gamma)\}_\mathrm{s}\}_\mathrm{s} 
- \{\{(f, \alpha), (g, \beta)\}_\mathrm{s}, (k, \gamma)\}_\mathrm{s} 
- \{(g, \beta), \{(f, \alpha), (k, \gamma)\}_\mathrm{s}\}_\mathrm{s} \\
&= \iota_{(u_1, v_1)} \mathrm{d} \iota_{(u_2, v_2)} \mathrm{d} (k, \gamma)
 - \iota_{[(u_1, v_1), (u_2, v_2)]} \mathrm{d} (k, \gamma)
 - \iota_{(u_2, v_2)} \mathrm{d} \iota_{(u_1, v_1)} \mathrm{d} (k, \gamma) \\
&= \left( \iota_{(u_1, v_1)} \mathrm{d} \iota_{(u_2, v_2)} 
    - \iota_{[(u_1, v_1), (u_2, v_2)]}
    - \iota_{(u_2, v_2)} \mathrm{d} \iota_{(u_1, v_1)} \right) \mathrm{d}(k, \gamma) \\
&= \left( \iota_{(u_1, v_1)} \mathrm{d} \iota_{(u_2, v_2)} 
   - \mathcal{L}_{(u_1, v_1)} \iota_{(u_2, v_2)} 
   + \iota_{(u_2, v_2)} \mathcal{L}_{(u_1, v_1)} 
   - \iota_{(u_2, v_2)} \mathrm{d} \iota_{(u_1, v_1)} \right) \mathrm{d}(k, \gamma) \\
&\qquad \text{(since } \iota_{[(u_1, v_1), (u_2, v_2)]} 
= \mathcal{L}_{(u_1, v_1)} \iota_{(u_2, v_2)} 
- \iota_{(u_2, v_2)} \mathcal{L}_{(u_1, v_1)}\text{)} \\
&= \left( -\mathrm{d} \iota_{(u_1, v_1)} \iota_{(u_2, v_2)} 
         + \iota_{(u_2, v_2)} \iota_{(u_1, v_1)} \mathrm{d} \right) \mathrm{d}(k, \gamma) \\
&= -\mathrm{d} \iota_{(u_1, v_1)} \iota_{(u_2, v_2)} \mathrm{d}(k, \gamma) 
  + \iota_{(u_2, v_2)} \iota_{(u_1, v_1)} \mathrm{d}^2(k, \gamma) \\
&= -\mathrm{d} \iota_{(u_1, v_1)} \iota_{(u_2, v_2)} \mathrm{d}(k, \gamma)
\qquad \text{(since } \mathrm{d}^2 = 0\text{)} \\
&= \mathrm{d} \iota_{(u_1, v_1)} \iota_{(u_2, v_2)} \iota_{(u_3, v_3)} (\omega_M, \omega_N)
\qquad \text{(since } \mathrm{d}(k, \gamma) = -\iota_{(u_3, v_3)} (\omega_M, \omega_N)\text{)} \\
&= \mathrm{d} J\left((f, \alpha), (g, \beta), (k, \gamma)\right).
\end{align*}
\endgroup
where \(J\) is the map 
\begin{align*}
J : \Omega^{n-1}_{\mathrm{Ham}}(F) \otimes \Omega^{n-1}_{\mathrm{Ham}}(F) \otimes \Omega^{n-1}_{\mathrm{Ham}}(F) &\to \Omega^{n-3}_{\mathrm{Ham}}(F)  \\
\left((f, \alpha), ( g,  \beta), (k,\gamma)\right) &\mapsto -\iota_{(u_1, v_1)}\iota_{ (u_2, v_2)}\iota_{(u_3, v_3)}  (\omega_M,\omega_N)
\end{align*}
Therefore, the expression satisfies the Jacobi identity up to an exact form, as desired.
\end{proof}

We now state a result that extends a well-known identity for the exterior derivative of iterated contractions with Hamiltonian vector fields to the setting of relative multisymplectic geometry. This is the relative analogue of Lemma \cite[Lemma 3.6]{rogers2011higher}.

\begin{lem}\label{lemma derivative}
Let \((F, \omega_M, \omega_N)\) be a relative \(n\)-plectic structure, and let \((u_1, v_1), \ldots, (u_m, v_m) \in \mathfrak{X}_{\mathrm{Ham}}(F)\) be Hamiltonian vector fields for \(m \geq 2\). Then the exterior derivative of the iterated contraction satisfies:
\begin{align*}
   & \mathrm{d}\iota_{(u_m,v_m)} \cdots \iota_{(u_1,v_1)}(\omega_M, \omega_N) \\
   &\qquad =
(-1)^m \sum_{1 \leq i < j \leq m} (-1)^{i+j} \iota_{(u_m,v_m)} \cdots \widehat{\iota_{(u_i,v_i)}} \cdots \widehat{\iota_{(u_j,v_j)}} \cdots \iota_{(u_1,v_1)}\iota_{([u_i, u_j], [v_i, v_j])}(\omega_M, \omega_N).
\end{align*}
where the hats indicate omission of the corresponding contraction.
\end{lem}

\begin{proof} 
We prove this by induction. For all \(m\geq 2\), let \(P(m)\) be the statement 
\begin{align*}
   & \mathrm{d}\iota_{(u_m,v_m)} \cdots \iota_{(u_1,v_1)}(\omega_M, \omega_N) \\
   &\qquad =
(-1)^m \sum_{1 \leq i < j \leq m} (-1)^{i+j} \iota_{(u_m,v_m)} \cdots \widehat{\iota_{(u_i,v_i)}} \cdots \widehat{\iota_{(u_j,v_j)}} \cdots \iota_{(u_1,v_1)}\iota_{([u_i, u_j], [v_i, v_j])}(\omega_M, \omega_N).
\end{align*} 

\paragraph{Base case:} For $m=2$, we have by definition of the bracket that
\[
\iota_{ (u_{2}, v_{2})} \iota_{(u_{1}, v_{1})}(\omega_M, \omega_N)=\left\{(f_{1}, \alpha_{1}), (f_{2}, \alpha_{2})\right\},
\]
where $(f_{1},\alpha_{1}),(f_{2},\alpha_{2})$
are any Hamiltonian $(n-1)$-forms whose Hamiltonian vector
fields are $(u_{1}, v_{1}),(u_{2}, v_{2})$, respectively.

It follows that
\[\mathrm{d}\iota_{ (u_{2}, v_{2})} \iota_{(u_{1}, v_{1})}(\omega_M, \omega_N)=\mathrm{d}\left\{(f_{1}, \alpha_{1}), (f_{2}, \alpha_{2})\right\},
\]
By Proposition \ref{proposition:bracket}, 
\[
\mathrm{d}\left\{(f_{1}, \alpha_{1}), (f_{2}, \alpha_{2})\right\} = \iota_{[(u_2,v_2), (u_1,v_1)]}(\omega_M,\omega_N).
\]

Hence,
\[\mathrm{d}\iota_{ (u_{2}, v_{2})} \iota_{(u_{1}, v_{1})}(\omega_M, \omega_N)=\mathrm{d}\left\{(f_{1}, \alpha_{1}), (f_{2}, \alpha_{2})\right\}=(-1)^2\iota_{[(u_2,v_2), (u_1,v_1)]}(\omega_M,\omega_N).
\]

\paragraph{Inductive step:} Assume that \(P(m)\) holds for some \(m \geq 2\). That is, 
\begin{align*}
   & \mathrm{d}\iota_{(u_m,v_m)} \cdots \iota_{(u_1,v_1)}(\omega_M, \omega_N) \\
   &\qquad =
(-1)^m \sum_{1 \leq i < j \leq m} (-1)^{i+j} \iota_{(u_m,v_m)} \cdots \widehat{\iota_{(u_i,v_i)}} \cdots \widehat{\iota_{(u_j,v_j)}} \cdots \iota_{(u_1,v_1)}\iota_{([u_i, u_j], [v_i, v_j])}(\omega_M, \omega_N).
\end{align*}

Let us prove that  \(P(m+1)\) holds. That is, we want to prove that
\[
\resizebox{\textwidth}{!}{$
\begin{aligned}
   & \mathrm{d}\iota_{(u_{m+1},v_{m+1})} \cdots \iota_{(u_1,v_1)}(\omega_M, \omega_N) \\
   &\qquad =
   (-1)^{m+1} \sum_{1 \leq i < j \leq m+1} (-1)^{i+j} \iota_{(u_{m+1},v_{m+1})} \cdots 
   \widehat{\iota_{(u_i,v_i)}} \cdots \widehat{\iota_{(u_j,v_j)}} \cdots \iota_{(u_1,v_1)}
   \iota_{([u_i, u_j], [v_i, v_j])}(\omega_M, \omega_N).
\end{aligned}
$}
\]

 We have
 \begin{align*}
    \iota_{(u_{m+1},v_{m+1})} \cdots \iota_{(u_1,v_1)}(\omega_M, \omega_N)&=\iota_{(u_{m+1}, v_{m+1})} \iota_{(u_{m},v_{m})} \cdots \iota_{(u_1,v_1)}(\omega_M, \omega_N)\\
    &=\iota_{(u_{m+1}, v_{m+1})} (\alpha, \beta),
 \end{align*}
 with \((\alpha, \beta )=\iota_{(u_{m},v_{m})} \cdots \iota_{(u_1,v_1)}(\omega_M, \omega_N)\).

It follows by Cartan's formula (Proposition \ref{relative cartan magic formula}, part (5)),
\begin{align}\label{step1}
    \mathrm{d}\iota_{(u_{m+1},v_{m+1})} \cdots \iota_{(u_1,v_1)}(\omega_M, \omega_N)&=\mathrm{d}\iota_{(u_{m+1}, v_{m+1})} (\alpha, \beta)\nonumber\\
    &=\mathcal{L}_{(u_{m+1}, v_{m+1})}(\alpha, \beta)- \iota_{(u_{m+1}, v_{m+1})} \mathrm{d} (\alpha, \beta) 
 \end{align}
But, by part (4) of Proposition \ref{relative cartan magic formula},  we have

\begin{align*}
\mathcal{L}_{(u_{m+1}, v_{m+1})}(\alpha, \beta)
&= \mathcal{L}_{(u_{m+1}, v_{m+1})} \iota\left( (u_1, v_1) \wedge \cdots \wedge (u_{m+1}, v_{m+1}) \right)(\omega_M, \omega_N) \\
&= \iota\left( [(u_{m+1}, v_{m+1}), (u_1, v_1)] \wedge \cdots \wedge (u_m, v_m) \right)(\omega_M, \omega_N) \\
&\quad + \iota\left( (u_1, v_1) \wedge \cdots \wedge (u_m, v_m) \right) \mathcal{L}_{(u_{m+1}, v_{m+1})}(\omega_M, \omega_N) \\
&= \iota\left( [(u_{m+1}, v_{m+1}), (u_1, v_1)] \wedge \cdots \wedge (u_m, v_m) \right)(\omega_M, \omega_N) 
\end{align*}

Moreover, 
\begin{align*}
 &[(u_{m+1}, v_{m+1}),(u_{1}, v_{1}) \wedge \cdots \wedge (u_{m}, v_{m})] & \\
 &=\sum_{i=1}^{m}
(-1)^{i+1} [(u_{m+1}, v_{m+1}),(u_{i}, v_{i})] \wedge (u_{1}, v_{1}) \wedge \cdots \wedge \widehat{(u_{i}, v_{i})}
\wedge \cdots \wedge (u_{m}, v_{m}).   &
\end{align*}

It follows that
\begingroup
\small
\begin{align*}
&\mathcal{L}_{(u_{m+1}, v_{m+1})} \, 
\iota\big((u_{1}, v_{1}) \wedge \cdots \wedge (u_{m}, v_{m})\big)(\omega_M, \omega_N) \\
&\quad = \iota\left( \left[(u_{m+1}, v_{m+1}), 
(u_{1}, v_{1}) \wedge \cdots \wedge (u_{m}, v_{m}) \right] \right)(\omega_M, \omega_N) \\
&\quad = \sum_{i=1}^{m} (-1)^{i} 
\, \iota\Big( [(u_{i}, v_{i}), (u_{m+1}, v_{m+1})] \wedge 
(u_{1}, v_{1}) \wedge \cdots \wedge \widehat{(u_{i}, v_{i})} \wedge \cdots \wedge (u_{m}, v_{m}) \Big)(\omega_M, \omega_N).
\end{align*}
\endgroup

By Equation \ref{step1} and using the hypothesis of induction, we get
\[
\resizebox{\textwidth}{!}{$
\begin{aligned}
&
\mathrm{d} \iota\big((u_{1}, v_{1}) \wedge \cdots \wedge (u_{m+1}, v_{m+1})\big) (\omega_M, \omega_N) \\
&
= \sum_{i=1}^{m}
(-1)^{i} \iota\big([(u_{i}, v_{i}), (u_{m+1}, v_{m+1})] \wedge (u_{1}, v_{1}) \wedge \cdots \wedge \widehat{(u_{i}, v_{i})}
\wedge \cdots \wedge (u_{m}, v_{m})\big) (\omega_M, \omega_N) \\
& \quad -(-1)^{m} \sum \limits_{1 \leq i < j \leq m} (-1)^{i+j} \iota_{(u_{m+1}, v_{m+1})} \iota\big([(u_{i}, v_{i}), (u_{j}, v_{j})] \wedge (u_{1}, v_{1}) \wedge \cdots \\
& \quad \wedge \widehat{(u_{i}, v_{i})} \wedge \cdots \wedge \widehat{(u_{j}, v_{j})} \wedge \cdots \wedge (u_{m}, v_{m})\big)
(\omega_M, \omega_N) \\
& = (-1)^{m+1} \left( \sum_{i=1}^{m}
(-1)^{i+m+1} \iota\big([(u_{i}, v_{i}), (u_{m+1}, v_{m+1})] \wedge (u_{1}, v_{1}) \wedge \cdots \wedge \widehat{(u_{i}, v_{i})}
\wedge \cdots \wedge (u_{m}, v_{m})\big) (\omega_M, \omega_N) \right. \\
& \quad \left. + \sum \limits_{1 \leq i < j \leq m} (-1)^{i+j} \iota\big([(u_{i}, v_{i}), (u_{j}, v_{j})] \wedge (u_{1}, v_{1}) \wedge \cdots \wedge
\widehat{(u_{i}, v_{i})} \wedge \cdots \wedge \widehat{(u_{j}, v_{j})} \wedge \cdots \wedge (u_{m+1}, v_{m+1})\big)
\omega \right) \\
& = (-1)^{m+1} \sum_{1 \leq i < j \leq m+1} (-1)^{i+j} \iota\big([(u_{i}, v_{i}), (u_{j}, v_{j})] \wedge (u_{1}, v_{1}) \wedge \cdots \wedge
\widehat{(u_{i}, v_{i})} \wedge \cdots \wedge \widehat{(u_{j}, v_{j})} \wedge \cdots \\
& \hspace{10cm} \wedge (u_{m+1}, v_{m+1})\big)\, \omega.
\end{aligned}
$}
\]

\end{proof}

\begin{rem}\label{simplicity of notation}
    
For simplicity of notation, we will adopt the following notation throughout the remainder of this thesis:

\begin{align*}\label{space}
\widetilde{\Omega^{n-1}_{\mathrm{Ham}}(F)} &= \left\{ (u, v)\oplus (f,\beta)  \in  \mathfrak{X}_{\mathrm{rel}}(F) \oplus \Omega^{n-1}_{\mathrm{Ham}}(F) \, \Bigg| \, \begin{pmatrix} F^*\beta + df \\ -d\beta \end{pmatrix} = -\begin{pmatrix} \iota_u\omega_M \\ -\iota_v\omega_N \end{pmatrix}
 \right\},
\end{align*}
where $ u\sim_{F} v$ means $u$ and $v$ are $F$-related.
\end{rem}

\begin{pro}\label{brackets F-related}
Let \( F: M \to N \) be a smooth map between smooth manifolds. Suppose \( u_1, u_2 \in \mathfrak{X}(M) \) and \( v_1, v_2 \in \mathfrak{X}(N) \) are vector fields such that \( u_1 \sim_F v_1 \) and \( u_2 \sim_F v_2 \); that is, they are \(F\)-related. Then their Lie brackets are also \(F\)-related:
\[
[u_1, u_2] \sim_F [v_1, v_2].
\]
\end{pro}

\begin{proof}
 Since \( u_1 \sim_F v_1 \) and \( u_2 \sim_F v_2 \), it follows that \( v_1(f) \circ F = u_1(f) \circ F \) and \( v_2(f) \circ F = u_2(f) \circ F \) for any smooth function \( f \) on \( N \). We need to check that:
    \[
    [v_1, v_2](f) \circ F = [u_1, u_2](f \circ F).
    \]
    
    We compute:
    \begin{align*}
        [v_1, v_2](f) \circ F &= \left(v_1(v_2(f)) - v_2(v_1(f))\right) \circ F \\
        &= v_1(v_2(f)) \circ F - v_2(v_1(f)) \circ F \\
        &= v_1(v_2(f) \circ F) - v_2(v_1(f) \circ F) \\
        &= u_1(v_2(f) \circ F) - u_2(v_1(f) \circ F) \\
        &= u_1(u_2(f \circ F)) - u_2(u_1(f \circ F)) \\
        &= [u_1, u_2](f \circ F).
    \end{align*}
    
    Therefore, \( [u_1, u_2] \sim_F [v_1, v_2] \).

\end{proof}

\begin{pro}\label{well-defined}
Consider the map
\[
\{\cdot,\cdot\}: \widetilde{\Omega^{n-1}_{\mathrm{Ham}}(F)}\times \widetilde{\Omega^{n-1}_{\mathrm{Ham}}(F)} \to \widetilde{\Omega^{n-1}_{\mathrm{Ham}}(F)}
\]
defined by
\[
\left\{ (u_1, v_1)\oplus (f_1,\beta_1), (u_2, v_2)\oplus (f_2,\beta_2) \right\}
=
\left( [u_1, u_2], [v_1,v_2]\right)\oplus \left(
\iota_{u_2}\iota_{u_1}\omega_M, \iota_{v_2}\iota_{v_1}\omega_N\right)
\]
This map is a skewsymmetric bilinear map and satisfies the Jacobi identity up to an exact form.
\end{pro}
Having established the necessary tools of Cartan calculus in the relative multisymplectic setting, we are now prepared to construct the \(L_\infty\)-algebra of relative observables. 
The goal of the next section is to define a higher Lie algebra structure on the space of relative Hamiltonian \((n{-}1)\)-forms, extending the construction of Rogers \cite{rogers2011higher} to the relative case.

\section{\Linfty-algebras of Relative Observables}

Since the relative Cartan formulas introduced in Section \ref{section:relative cartan magic formulas} satisfy analogous properties to those found in differential geometry, it is natural to expect relative counterparts to the results of \cite{Rogers_2011, callies2016homotopy}. The approach presented here adapts the methodology of \cite[Theorem 3.14]{rogers2011higher}, substituting the de Rham differential with the relative differential and extending the notion of observables to their relative counterparts.

The following theorem is the relative version of \cite[Theorem 3.14]{rogers2011higher}.

\begin{thm} \label{Relative L-infinity algebra main thm}
Given a relative $n$-plectic structure  $(F,\omega_M, \omega_N)$, there is a Lie $n$-algebra

\[
\resizebox{\textwidth}{!}{$
\begin{array}{ccccccccccccccc}
    0 & \longrightarrow & L_{n-1} & \longrightarrow & L_{n-2}  & \longrightarrow & \cdots & \longrightarrow & L_{k-2} & \longrightarrow & \cdots & \longrightarrow & L_{1} & \longrightarrow & L_0 \\
     &  & \parallel &  & \parallel &  &  &  & \parallel &  &  &  & \parallel &  & \parallel \\
     &  & \Omega^{0}(F) &  & \Omega^{1}(F)  &  &  &  & \Omega^{k-2}(F) &  &  &  & \Omega^{n-2}(F) &  & \Omega^{n-1}_{\textrm{Ham}}(F) \\
\end{array}
$}
\]

denoted $\mathrm{L}_\infty(F,\omega_M, \omega_N)=(\mathrm{L},\{l_{k} \})$ with 
\begin{itemize}
    \item 
underlying graded vector space \(L_0=\Omega_{\mathrm{Ham}}^{n-1}(F)\) and \(L_i=\Omega^{n-1-i}(F) \) for  \(0 \leq i < n-1\)

and 
\item maps  $\left \{l_{k} : \mathrm{L}^{\otimes k} \to \mathrm{L}\, \, |\, \,  1
  \leq k < \infty \right\}$ defined as
\[ 
l_{1}(f, \alpha)=\mathrm{d}(f, \alpha)=(F^*\alpha+\mathrm{d}f, -\mathrm{d}\alpha),
\]
if $\deg(f, \alpha)>0$ and 
\begingroup
\small 
\begin{align*}
&l_k\left((f_1,\alpha_1), \ldots, (f_k,\alpha_k)\right)
\\
&\qquad \qquad= \begin{cases}
0 & \text{if } \deg\left((f_1,\alpha_1)\otimes \cdots \otimes (f_k,\alpha_k)\right) > 0, \\[0.5em]
(-1)^{\frac{k}{2}+1} \iota_{(u_k, v_k)} \cdots \iota_{(u_1, v_1)} (\omega_M, \omega_N) 
& \begin{aligned}
&\text{if } \deg\left((f_1,\alpha_1)\otimes \cdots \otimes (f_k,\alpha_k)\right) = 0 \\
&\text{and } k \text{ is even},
\end{aligned} \\[0.5em]
(-1)^{\frac{k-1}{2}} \iota_{(u_k, v_k)} \cdots \iota_{(u_1, v_1)} (\omega_M, \omega_N) 
& \begin{aligned}
&\text{if } \deg\left((f_1,\alpha_1)\otimes \cdots \otimes (f_k,\alpha_k)\right) = 0 \\
&\text{and } k \text{ is odd}.
\end{aligned}
\end{cases}
\end{align*}
\endgroup

for $k>1$, where $(u_i, v_i)$ is the unique Hamiltonian vector field
associated to $(f_i, \alpha_{i}) \in \Omega_{\mathrm{Ham}}^{n-1}(F)$.
\end{itemize}
\end{thm}

\begin{proof}
Let us show that the maps $\{l_{k}\}$
are well-defined, skew-symmetric, and $\deg(l_{k})=k-2$.

\paragraph{Let us show that $\deg(l_{k})=k-2$:} Because $\iota_{(u_k, v_k)}\cdots \iota_{(u_1, v_1)}  (\omega_M, \omega_N) \in
\Omega^{n+1-k}(F)=\Omega^{n-1-(k-2)}(F)=\mathrm{L}_{k-2}$, then $\deg(l_k)=k-2$. Since \((\omega_M, \omega_N)\) is an \((n+1)\)-form, then \(\iota_{(u_k, v_k)}\cdots \iota_{(u_1, v_1)}  (\omega_M, \omega_N)=0\) for \(k>n+1\) so that $l_{k} = 0 $ for $k>n+1$.

\paragraph{Let us show that the maps $\{l_{k}\}$
are well-defined:} We need to show that for any \((f_1,\alpha_1)\otimes \ldots \otimes (f_k,\alpha_k)\in L^{\otimes k}\), \(l_{k}\left((f_1,\alpha_1), \ldots, (f_k,\alpha_k) \right)\) is a Hamiltonian \((n-1)\)-form whenever \(\deg l_{k}\left((f_1,\alpha_1), \ldots, (f_k,\alpha_k) \right)=0.\) Since \(l_{k}\) has degree \(k-2\), then
\begin{align*}
  \deg l_{k}\left((f_1,\alpha_1), \ldots, (f_k,\alpha_k) \right)=k-2+  \deg \left((f_1,\alpha_1), \ldots, (f_k,\alpha_k) \right)  
\end{align*}

Thus,
\begin{align*}
  \deg l_{k}\left((f_1,\alpha_1), \ldots, (f_k,\alpha_k) \right)=0 \implies   \deg \left((f_1,\alpha_1), \ldots, (f_k,\alpha_k) \right)=2-k .
\end{align*}
Since \(L_i=0\) for \(i<0\), then it suffices to prove it for $k=1, 2$. 

\paragraph{For $k=1$:} If $(f,\alpha)\in L_0$, we have $l_1$ is of degree $-1$, so $l_1(f,\alpha)=0$ since $L_{-1}=0$. If $(f,\alpha)\in L_1$, then
\[ 
l_{1}(f, \alpha)=\mathrm{d}(f, \alpha),
\]
so that \(l_{1}(f, \alpha)\) is an exact $(n-1)$-form, hence it is Hamiltonian with corresponding Hamiltonian vector field being the pair $(0,0)$.

\paragraph{For $k=2$:} we have $l_2$ is of degree $0$, and for $\deg{(f_1,\alpha_1)\otimes(f_2,\alpha_2)}=0$, 
\begin{align*}
l_{2}\left((f_1,\alpha_1), (f_2,\alpha_2) \right)
& = \iota_{(u_2, v_2)} \iota_{(u_1, v_1)}  (\omega_M, \omega_N).
\end{align*}
Since $\deg{(f_1,\alpha_1)\otimes(f_2,\alpha_2)}=0$, then $\deg{(f_1,\alpha_1)}=0$ and $\deg{(f_2,\alpha_2)}=0$ so that $(f_1,\alpha_1)$ and $(f_2,\alpha_2)$ are Hamiltonian $(n-1)$-forms. By Proposition \ref{proposition:bracket}, their bracket $\{(f_1,\alpha_1), (f_2,\alpha_2)\}$ is Hamiltonian. But
$$\{(f_1,\alpha_1), (f_2,\alpha_2)\}=\iota_{(u_2, v_2)} \iota_{(u_1, v_1)}  (\omega_M, \omega_N) =l_{2}\left((f_1,\alpha_1), (f_2,\alpha_2) \right).$$
Hence $l_{2}\left((f_1,\alpha_1), (f_2,\alpha_2) \right)$ is Hamiltonian as desired. This proves that the maps $l_k$ are well-defined.

\paragraph{Let us show that the maps $l_k$ are skew-symmetric:}
If $(f_1,\alpha_1)\otimes \cdots\otimes (f_k,\alpha_k)\in \mathrm{L}^{\otimes \bullet}$ has
degree 0, then for all $\sigma \in \mathrm{S}_{k}$, $(f_{\sigma(1)},\alpha_{\sigma(1)})\otimes \cdots\otimes (f_{\sigma(k)},\alpha_{\sigma(k)})\in \mathrm{L}^{\otimes \bullet}$ has
degree 0. Also, since $(\omega_M, \omega_N) $ is skew-symmetric, we have
\[\iota_{(u_{\sigma(k)}, v_{\sigma(k)})}\cdots \iota_{(u_{\sigma(1)}, v_{\sigma(1)})}  (\omega_M, \omega_N)=(-1)^{\operatorname{sgn}(\sigma)}\iota_{(u_k, v_k)}\cdots \iota_{(u_1, v_1)}  (\omega_M, \omega_N) .\]

It follows that:
\[
\resizebox{\textwidth}{!}{$
\begin{aligned}
& l_{k}\left((f_{\sigma(1)},\alpha_{\sigma(1)}), \cdots, (f_{\sigma(k)},\alpha_{\sigma(k)}) \right) = \\
& \qquad\qquad =
\begin{cases}
(-1)^{\frac{k}{2}+1} \iota_{(u_{\sigma(k)}, v_{\sigma(k)})} \cdots \iota_{(u_{\sigma(1)}, v_{\sigma(1)})} (\omega_M, \omega_N) 
& \text{if } \deg\big((f_1,\alpha_1) \otimes \cdots \otimes (f_k,\alpha_k)\big) = 0 \\
& \text{and } k \text{ even}, \\
(-1)^{\frac{k-1}{2}} \iota_{(u_{\sigma(k)}, v_{\sigma(k)})} \cdots \iota_{(u_{\sigma(1)}, v_{\sigma(1)})} (\omega_M, \omega_N)
& \text{if } \deg\big((f_1,\alpha_1) \otimes \cdots \otimes (f_k,\alpha_k)\big) = 0 \\
& \text{and } k \text{ odd},
\end{cases} \\
& \qquad\qquad =
\begin{cases}
(-1)^{\frac{k}{2}+1} (-1)^{\operatorname{sgn}(\sigma)} \iota_{(u_k, v_k)} \cdots \iota_{(u_1, v_1)} (\omega_M, \omega_N)
& \text{if } \deg\big((f_1,\alpha_1) \otimes \cdots \otimes (f_k,\alpha_k)\big) = 0 \\
& \text{and } k \text{ even}, \\
(-1)^{\frac{k-1}{2}} (-1)^{\operatorname{sgn}(\sigma)} \iota_{(u_k, v_k)} \cdots \iota_{(u_1, v_1)} (\omega_M, \omega_N)
& \text{if } \deg\big((f_1,\alpha_1) \otimes \cdots \otimes (f_k,\alpha_k)\big) = 0 \\
& \text{and } k \text{ odd},
\end{cases} \\
& \qquad\qquad = (-1)^{\operatorname{sgn}(\sigma)} l_{k}((f_1,\alpha_1), \cdots, (f_k,\alpha_k)).
\end{aligned}
$}
\]
Furthermore, since \(\deg(f_i, \alpha_{i})=0\) for each \(i\), we have \(\epsilon(\sigma)=1\). This proves that $l_{k}$ is skew-symmetric.

\paragraph{Let us show that \(l_k\) satisfy the generalized Jacobi identity (Equation \ref{generalized jacobi identity} in Definition
\ref{Linfty algebra}):}
$1 \leq m < \infty :$

\begin{align*} 
   \sum_{\substack{i+j = m+1, \\ \sigma \in \mathrm{Sh}(i,m-i)}}
  (-1)^{\operatorname{sgn}(\sigma)}\epsilon(\sigma)(-1)^{i(j-1)} l_{j}
   (l_{i}(x_{\sigma(1)}, \dots, x_{\sigma(i)}), x_{\sigma(i+1)},
   \ldots, x_{\sigma(m)})=0.
\end{align*}

\begin{itemize}
    \item 

If $m=1$, then it is satisfied since $l_{1}$
is the relative differential, and the left-hand side of the Equation \ref{generalized jacobi identity} becomes
\begin{align*} 
   \sum_{\substack{i+j = 2, \\ \sigma \in \mathrm{Sh}(i,1-i)}}
  (-1)^{\operatorname{sgn}(\sigma)}\epsilon(\sigma)(-1)^{i(j-1)} l_{j}
   (l_{i}(x_{\sigma(1)}, \dots, x_{\sigma(i)}), x_{\sigma(i+1)},
   \ldots, x_{\sigma(m)})
\end{align*}
and $i=j=1$ so that $l_j=l_i=l_1=\mathrm{d}$ is the relative differential. Since $\mathrm{Sh}(1,0)= \{\mathrm{id}\}$, the equation above becomes,
 \(l_1\circ l_1=\mathrm{d}^2=0,\) which is always true because $\mathrm{d}$
 is the relative deRham differential. 

\item If $m=2$, then Equation \ref{generalized jacobi identity} becomes
\begin{align*} 
   \sum_{\substack{i+j = 3, \\ \sigma \in \mathrm{Sh}(i,2-i)}}
  (-1)^{\operatorname{sgn}(\sigma)}\epsilon(\sigma)(-1)^{i(j-1)} l_{j}
   (l_{i}(x_{\sigma(1)}, \dots, x_{\sigma(i)}), x_{\sigma(i+1)},
   \ldots, x_{\sigma(m)})=0.
\end{align*}

In this case, we have $(i,j)=(1, 2), (2,1)$. We have $\mathrm{Sh}(1,1)=\{\mathrm{id}, (12)\}$ and  $\mathrm{Sh}(2,0)=\{\mathrm{id}\}$. The equality we must prove in this case can be written as:
\begin{equation}\label{case m=2}
    l_{1}(l_{2}\left((f_1, \alpha_1),\, (f_2, \alpha_2)\right) - l_{2}\left(l_{1}(f_1, \alpha_1),\, (f_2, \alpha_2)\right) - l_{2}(l_{1}((f_2, \alpha_2)), (f_1, \alpha_1))=0.
\end{equation}

That is,
\[
l_{1}(l_{2}\left((f_1, \alpha_1),\, (f_2, \alpha_2)\right) - l_{2}\left(l_{1}(f_1, \alpha_1),\, (f_2, \alpha_2)\right) - (-1)^{\deg{(f_1, \alpha_1)}}l_{2}((f_1, \alpha_1),l_{1}((f_2, \alpha_2)))=0.
\]

\paragraph{Case 1:}
If  $\deg{(f_1,\alpha_1)\otimes  (f_2,\alpha_2)} \geq 2$, then  $\deg{(l_{1}(f_1, \alpha_1), (f_2, \alpha_2))} \geq 1$ since \(l_1\) is of degree \(-1\). It follows by definition of \(l_2\) that \[l_{2}\left(l_{1}(f_1, \alpha_1),\, (f_2, \alpha_2)\right)=(0, 0).\] A similar argument shows that 
\[l_{2}((f_1, \alpha_1),l_{1}((f_2, \alpha_2)))=(0,0).\]

Since $\deg{(f_1,\alpha_1)\otimes  (f_2,\alpha_2)} \geq 2$, then \(l_{2}\left((f_1, \alpha_1),\, (f_2, \alpha_2)\right) =(0,0)\) so that
 \[l_{1}(l_{2}\left((f_1, \alpha_1),\, (f_2, \alpha_2)\right) =l_1(0,0)=(0, 0).\]

It follows that Equation \ref{case m=2} holds trivially.

\paragraph{Case 2:}
If $\deg{(f_1,\alpha_1) }= 0$ and $\deg{  (f_2,\alpha_2)} = 1$, then \(l_1(f_1,\alpha_1)=(0, 0)\). 
It follows \[l_{2}\left(l_{1}(f_1, \alpha_1),\, (f_2, \alpha_2)\right)=l_{2}\left((0, 0),\, (f_2, \alpha_2)\right)=(0, 0).\] 

Since  $\deg{  (f_2,\alpha_2)} = 1>0$, then \(l_{1}((f_2, \alpha_2)))=\mathrm{d}(f_2, \alpha_2)\) is an exact \((n-1)\)-form, so it is Hamiltonian with corresponding Hamiltonian vector field the pair \((0,0)\). It follows by definition of \(l_2\) that  
 \[l_{2}\left((f_1, \alpha_1),\, l_{1}(f_2, \alpha_2)\right)=\iota_{(0, 0)}\iota_{(u_1, v_1)}  (\omega_M, \omega_N)=(0, 0).\]

 Since $\deg{(f_1,\alpha_1)\otimes  (f_2,\alpha_2)} =1>0$, then \(l_{2}\left((f_1, \alpha_1),\, (f_2, \alpha_2)\right) =(0,0)\) so that
 \[l_{1}(l_{2}\left((f_1, \alpha_1),\, (f_2, \alpha_2)\right) =l_1(0,0)=0.\]
Hence, Equation \ref{case m=2} holds trivially.

\paragraph{Case 3:}
If $\deg{(f_1,\alpha_1) }= 1$ and $\deg{  (f_2,\alpha_2)} = 0$, then \(l_1(f_2,\alpha_2)=(0, 0)\). 
It follows \[l_{2}\left((f_1, \alpha_1),\, l_{1}(f_2, \alpha_2)\right)=l_{2}\left((f_1, \alpha_1),\,(0, 0)\right)=(0, 0).\] 

Since $\deg{(f_1,\alpha_1) }= 1$, then \(l_1(f_1,\alpha_1)=\mathrm{d}(f_1,\alpha_1)\) is an exact \((n-1)\)-form, so it is Hamiltonian with corresponding Hamiltonian vector field the pair \((0,0)\). It follows by definition of \(l_2\) that  
 \[l_{2}\left(l_{1}(f_1, \alpha_1),\, (f_2, \alpha_2)\right)=\iota_{(u_2, v_2)}\iota_{(0, 0)}  (\omega_M, \omega_N)=(0, 0).\]

 Since $\deg{(f_1,\alpha_1)\otimes  (f_2,\alpha_2)} =1>0$, then \(l_{2}\left((f_1, \alpha_1),\, (f_2, \alpha_2)\right) =(0,0)\) so that
 \[l_{1}(l_{2}\left((f_1, \alpha_1),\, (f_2, \alpha_2)\right) =l_1(0,0)=0.\]
Again, Equation \ref{case m=2} holds trivially.

 \paragraph{Case 4:} 
If $\deg{(f_1,\alpha_1) }= 0$ and $\deg{  (f_2,\alpha_2)} = 0$, then \(l_{1}(f_1, \alpha_1)=0\) and  \(l_1(f_2,\alpha_2)=(0, 0)\). 

It follows that
 \[l_{2}\left(l_{1}(f_1, \alpha_1),\, (f_2, \alpha_2)\right)=l_{2}\left((0, 0),\,(f_2, \alpha_2)\right)=(0, 0)\] 
 and
 \[l_{2}\left((f_1, \alpha_1),\, l_{1}(f_2, \alpha_2)\right)=l_{2}\left((f_1, \alpha_1),\,(0, 0)\right)=(0, 0)\]

 Since    $\deg{(f_1,\alpha_1)\otimes  (f_2,\alpha_2)}=0$, then
 \begin{align*}
l_{2}\left((f_1,\alpha_1), (f_2,\alpha_2) \right)
& = \iota_{(u_2, v_2)} \iota_{(u_1, v_1)}  (\omega_M, \omega_N)   
\end{align*}
Since $\deg{(f_1,\alpha_1)\otimes(f_2,\alpha_2)}=0$, then $\deg{(f_1,\alpha_1)}=0$ and $\deg{(f_2,\alpha_2)}=0$ so that $(f_1,\alpha_1)$ and $(f_1,\alpha_1)$ are Hamiltonian $(n-1)$-forms. By Proposition \ref{proposition:bracket}, their bracket $\{(f_1,\alpha_1), (f_2,\alpha_2)\}$ is Hamiltonian. But
$$\{(f_1,\alpha_1), (f_2,\alpha_2)\}=\iota_{(u_2, v_2)} \iota_{(u_1, v_1)}  (\omega_M, \omega_N) =l_{2}\left((f_1,\alpha_1), (f_2,\alpha_2) \right).$$

Hence $l_{2}\left((f_1,\alpha_1), (f_2,\alpha_2) \right)\in L_0$. Since $\deg l_{2}\left((f_1,\alpha_1), (f_2,\alpha_2) \right)=0$, it follows that
$$l_1 (l_{2}\left((f_1,\alpha_1), (f_2,\alpha_2) \right))=0.$$
 
Again, Equation \ref{case m=2} holds trivially. In any case, we have proved that
\[
l_{1}(l_{2}\left((f_1, \alpha_1),\, (f_2, \alpha_2)\right) = l_{2}(l_{1}\left((f_1, \alpha_1),\, (f_2, \alpha_2)\right) + (-1)^{\deg{(f_1, \alpha_1)}}l_{2}((f_1, \alpha_1),l_{1}((f_2, \alpha_2))).
\]

\item Now assume $m > 2$. We will decompose the summation 
in Equation \ \ref{generalized jacobi identity} into two sums:
\[
\resizebox{\textwidth}{!}{$
\begin{aligned}
 &  \sum_{\substack{i+j = m+1 \\ \sigma \in \mathrm{Sh}(i,m-i)}}
  (-1)^{\operatorname{sgn}(\sigma)}\epsilon(\sigma)(-1)^{i(j-1)} l_{j}
   \left(l_{i}\left( (f_{\sigma(1)}, \alpha_{\sigma(1)}), \dots, (f_{\sigma(i)}, \alpha_{\sigma(i)})\right), (f_{\sigma(i+1)}, \alpha_{\sigma(i+1)}),
   \ldots, (f_{\sigma(m)}, \alpha_{\sigma(m)})\right) \\
   &= \sum_{j=2}^{m-2}
   \sum_{\sigma \in \mathrm{Sh}(i,m-i)} 
   (-1)^{\operatorname{sgn}(\sigma)}\epsilon(\sigma)(-1)^{i(j-1)} l_{j}
   \left(l_{i}\left( (f_{\sigma(1)}, \alpha_{\sigma(1)}), \dots, (f_{\sigma(i)}, \alpha_{\sigma(i)})\right), (f_{\sigma(i+1)}, \alpha_{\sigma(i+1)}),
   \ldots, (f_{\sigma(m)}, \alpha_{\sigma(m)})\right) \\
   &\quad + l_{1}\left( l_{m}\left((f_{\sigma(1)}, \alpha_{\sigma(1)}), \dots, (f_{\sigma(m)}, \alpha_{\sigma(m)}) \right)\right) \\
   &\quad + \sum_{\sigma \in \mathrm{Sh}(2,m-2)} 
(-1)^{\operatorname{sgn}(\sigma)}
\epsilon(\sigma)
l_{m-1}\left(l_{2}\left((f_{\sigma(1)}, \alpha_{\sigma(1)}), (f_{\sigma(2)}, \alpha_{\sigma(2)})\right), (f_{\sigma(3)}, \alpha_{\sigma(3)}), \ldots, (f_{\sigma(m)}, \alpha_{\sigma(m)})\right) \\
   &\quad + \sum_{\sigma \in \mathrm{Sh}(1,m-1)} 
(-1)^{\operatorname{sgn}(\sigma)}
\epsilon(\sigma)(-1)^{m-1} 
l_{m}\left(l_{1}\left(f_{\sigma(1)}, \alpha_{\sigma(1)}\right), (f_{\sigma(2)}, \alpha_{\sigma(2)}), \ldots, (f_{\sigma(m)}, \alpha_{\sigma(m)})\right)
\end{aligned}
$}
\]

where the last terms correspond to the case the case $j=1, m-1, m$, respectively. Depending on the value
of the index $j$, we show that each of these is zero, thereby proving the
theorem.

\paragraph{Case 1: $2\leq j\leq m-2$.}
We first consider the sum of the terms with $2 \leq j \leq m-2$.

\begin{equation} \label{term1}
\resizebox{5.8in}{0.25in}{$
\displaystyle \sum_{j=2}^{m-2}
\sum_{\sigma \in \mathrm{Sh}(i,m-i)} 
(-1)^{\operatorname{sgn}(\sigma)}\epsilon(\sigma)(-1)^{i(j-1)} 
l_{j} \left(
l_{i}\left( 
(f_{\sigma(1)}, \alpha_{\sigma(1)}), \dots, (f_{\sigma(i)}, \alpha_{\sigma(i)})
\right), 
(f_{\sigma(i+1)}, \alpha_{\sigma(i+1)}), \ldots, (f_{\sigma(m)}, \alpha_{\sigma(m)})
\right)
$}
\end{equation}

In this case we claim that for all $\sigma \in \mathrm{Sh}(i,m-i)$ we have 
\begin{align*}
    l_{j}
   \left(l_{i}\left( (f_{\sigma(1)}, \alpha_{\sigma(1)}), \dots, (f_{\sigma(i)}, \alpha_{\sigma(i)})\right), (f_{\sigma(i+1)}, \alpha_{\sigma(i+1)}),
   \ldots, (f_{\sigma(m)}, \alpha_{\sigma(m)})\right)=0.
\end{align*}

For a contradiction, assume there exists an unshuffle $\sigma \in \mathrm{Sh}(i,m-i)$ such that
\begin{align*}
    l_{j}
   \left(l_{i}\left( (f_{\sigma(1)}, \alpha_{\sigma(1)}), \dots, (f_{\sigma(i)}, \alpha_{\sigma(i)})\right), (f_{\sigma(i+1)}, \alpha_{\sigma(i+1)}),
   \ldots, (f_{\sigma(m)}, \alpha_{\sigma(m)})\right)\neq 0.
\end{align*}

By the definition of $l_{j} : \mathrm{L}^{\otimes   j} \to \mathrm{L}$, we must have
\begin{align*}
    \deg\left(l_{i}\left( (f_{\sigma(1)}, \alpha_{\sigma(1)}), \dots, (f_{\sigma(i)}, \alpha_{\sigma(i)})\right)\otimes (f_{\sigma(i+1)}, \alpha_{\sigma(i+1)})\otimes
   \ldots\otimes (f_{\sigma(m)}, \alpha_{\sigma(m)})\right)=0.
\end{align*}
That is,
{\small
\begin{align*}
\deg\!\left(
    l_{i}\!\left(
        (f_{\sigma(1)}, \alpha_{\sigma(1)}), \dots, (f_{\sigma(i)}, \alpha_{\sigma(i)})
    \right)
\right)
+ \deg\!\left(
    (f_{\sigma(i+1)}, \alpha_{\sigma(i+1)}) \otimes \cdots \otimes
    (f_{\sigma(m)}, \alpha_{\sigma(m)})
\right) = 0.
\end{align*}
}

This implies that
\begin{align*}
    \deg\left(l_{i}\left( (f_{\sigma(1)}, \alpha_{\sigma(1)}), \dots, (f_{\sigma(i)}, \alpha_{\sigma(i)})\right)\right)=0
\end{align*}
since $\deg\left((f_{\sigma(i+1)}, \alpha_{\sigma(i+1)})\otimes
   \ldots\otimes (f_{\sigma(m)}, \alpha_{\sigma(m)})\right)=0$. Since $l_i$ is of degree $i-2$, it follows that
   {\footnotesize\begin{align} \label{step2}
\deg\left(l_{i}\left( (f_{\sigma(1)}, \alpha_{\sigma(1)}), \dots, (f_{\sigma(i)}, \alpha_{\sigma(i)})\right)\right)
= \deg\left((f_{\sigma(i+1)}, \alpha_{\sigma(i+1)}) \otimes \cdots \otimes (f_{\sigma(m)}, \alpha_{\sigma(m)})\right) + i - 2 = 0.
\end{align}}

By assumption, \(l_{i}\left( (f_{\sigma(1)}, \alpha_{\sigma(1)}), \dots, (f_{\sigma(i)}, \alpha_{\sigma(i)})\right)\)  must be non-zero and $j < m-1$
implies $i>1$. Hence we must have $\deg\left((f_{\sigma(i+1)}, \alpha_{\sigma(i+1)})\otimes
   \ldots\otimes (f_{\sigma(m)}, \alpha_{\sigma(m)})\right)=0$ and
therefore, by Equation \ref{step2}, $i=2$. Since $i+j=m+1$, this
implies $j=m-1$, which contradicts the fact that $2\leq j\leq m-2$. So, no such
unshuffle could exist, and therefore, the sum (\ref{term1}) is zero.  

\paragraph{Case 2: $j=1$, $j=m-1$, and $j=m$.}
We next consider the sum of the terms $j=1$, $j=m-1$, and $j=m$:
{\footnotesize\begin{align} \label{the_sum}
 &l_{1}\left( l_{m}\left((f_{\sigma(1)}, \alpha_{\sigma(1)}), \dots, (f_{\sigma(m)},\alpha_{\sigma(m)}) \right)\right)\nonumber\\
&+ \sum_{\sigma \in \mathrm{Sh}(2,m-2)} 
(-1)^{\operatorname{sgn}(\sigma)}
\epsilon(\sigma)
l_{m-1}\left(l_{2}\left((f_{\sigma(1)}, \alpha_{\sigma(1)}), (f_{\sigma(2)}, \alpha_{\sigma(2)})\right), (f_{\sigma(3)}, \alpha_{\sigma(3)}),\ldots, (f_{\sigma(m)}, \alpha_{\sigma(m)})\right)\nonumber
 \\
&+ \sum_{\sigma \in \mathrm{Sh}(1,m-1)}  \negthickspace \negthickspace \negthickspace
(-1)^{\operatorname{sgn}(\sigma)}
\epsilon(\sigma)  (-1)^{m-1} 
l_{m}\left(l_{1}\left( f_{\sigma(1)}, \alpha_{\sigma(1)}\right), (f_{\sigma(2)}, \alpha_{\sigma(2)}),\hdots, (f_{\sigma(m)}, \alpha_{\sigma(m)})\right).
\end{align}}
Note that if $\sigma \in \mathrm{Sh}(1,m-1)$ and $\deg\left(l_{1}(f_{\sigma(1)}, \alpha_{\sigma(1)})\right) >0$, then
\[
l_{m}\left(l_{1}(f_{\sigma(1)}, \alpha_{\sigma(1)}), (f_{\sigma(2)}, \alpha_{\sigma(2)}),\ldots, (f_{\sigma(m)}, \alpha_{\sigma(m)})\right)=0
\]
by definition of the map $l_{m}$. On the other hand, 
if $\deg\left(l_{1}(f_{\sigma(1)}, \alpha_{\sigma(1)})\right) =0$, then
\[l_{1}(f_{\sigma(1)}, \alpha_{\sigma(1)})=\mathrm{d} (f_{\sigma(1)}, \alpha_{\sigma(1)})\] is Hamiltonian, and its Hamiltonian vector field is the zero vector field. Hence, the third
term in (\ref{the_sum}) is zero. 

Since the map $l_{2}$ is degree 0, we only need to consider the
first two terms of (\ref{the_sum}) in the case when $\deg\left((f_{\sigma(1)}, \alpha_{\sigma(1)})\otimes \dots \otimes (f_{\sigma(m)},\alpha_{\sigma(m)}) \right)
=0$. For the first term, we have:
{\footnotesize\begin{align*}
l_{1}\left(l_{m}\left((f_{\sigma(1)}, \alpha_{\sigma(1)}), \dots, (f_{\sigma(m)}, \alpha_{\sigma(m)}) \right)\right) =
\begin{cases}
(-1)^{\frac{m}{2}+1} \, \mathrm{d}\iota_{(u_m, v_m)} \cdots \iota_{(u_1, v_1)} (\omega_M, \omega_N) & \text{if } m \text{ even}, \\
(-1)^{\frac{m-1}{2}} \, \mathrm{d}\iota_{(u_m, v_m)} \cdots \iota_{(u_1, v_1)} (\omega_M, \omega_N) & \text{if } m \text{ odd}.
\end{cases}
\end{align*}}

Now consider the second term. If $(f_{i}, \alpha_{i}), \, (f_{j}, \alpha_{j}) \in \Omega_{\mathrm{Ham}}^{n-1}(F)$ are Hamiltonian
$(n-1)$forms, then, by Definition \ref{bracket_def},
\[l_{2}\left((f_{i}, \alpha_{i}), (f_{j}, \alpha_{j})\right)=\left\{(f_{i}, \alpha_{i}), (f_{j}, \alpha_{j})\right\}_{\mathrm{s}}.\]

By Proposition \ref{proposition:bracket}, 
\(l_{2}\left((f_{i}, \alpha_{i}), \, (f_{j}, \alpha_{j})\right)\) is
Hamiltonian, and its Hamiltonian vector field is $v_{\left\{(f_{i}, \alpha_{i}), (f_{j}, \alpha_{j})\right\}_{\mathrm{s}}}=[v_{(f_{i}, \alpha_{i})}, v_{(f_{j}, \alpha_{j})}]$.
Therefore for  $\sigma \in \mathrm{Sh}(2,m-2)$, we have

\begingroup
\small
\begin{multline*}
l_{m-1}\Big(l_{2}\big((f_{\sigma(1)}, \alpha_{\sigma(1)}), (f_{\sigma(2)}, \alpha_{\sigma(2)})\big), 
(f_{\sigma(3)}, \alpha_{\sigma(3)}), \ldots, (f_{\sigma(m)}, \alpha_{\sigma(m)})\Big) = \\
\begin{cases}
\begin{aligned}
(-1)^{\frac{m}{2}-1} \iota_{[(u_{\sigma(1)}, v_{\sigma(1)}), (u_{\sigma(2)}, v_{\sigma(2)})]} 
\cdots \iota_{(u_{\sigma(m)}, v_{\sigma(m)})} (\omega_M, \omega_N)
\end{aligned}
& \text{if } m \text{ is even}, \\[0.5em]
\begin{aligned}
(-1)^{\frac{m+1}{2}} \iota_{[(u_{\sigma(1)}, v_{\sigma(1)}), (u_{\sigma(2)}, v_{\sigma(2)})]} 
\cdots \iota_{(u_{\sigma(m)}, v_{\sigma(m)})} (\omega_M, \omega_N)
\end{aligned}
& \text{if } m \text{ is odd}.
\end{cases}
\end{multline*}
\endgroup

 Any unshuffle $\sigma \in \mathrm{Sh}(2,m-2)$ is completely determined by the choices of $\sigma(1)$ and $\sigma(2)$. Once $\sigma(1)$ and $\sigma(2)$ are chosen, say $\sigma(1)=i$ and $\sigma(2)=j$ (necessarily $i<j$), then $\sigma(3)=\min \{1, 2, 3, \cdots, m\}\setminus \{i, j\}$ and $\sigma(4)$ will be the smallest integer among the remaining and so on.

Since each $(f_i, \alpha_i)$ is degree 0, we can rewrite the sum over
$\sigma \in \mathrm{Sh}(2,m-2)$ as
{\footnotesize\begin{align*}
&\sum_{\sigma \in \mathrm{Sh}(2,m-2)}
(-1)^{\operatorname{sgn}(\sigma)}
\epsilon(\sigma)
l_{m-1}\left(l_{2}\left((f_{\sigma(1)}, \alpha_{\sigma(1)}), (f_{\sigma(2)}, \alpha_{\sigma(2)})\right), (f_{\sigma(3)}, \alpha_{\sigma(3)}), \ldots, (f_{\sigma(m)}, \alpha_{\sigma(m)})\right) = \\
&\sum_{1 \leq i < j \leq m} (-1)^{i+j-1} 
l_{m-1}\left(l_{2}\left((f_i, \alpha_i), (f_j, \alpha_j)\right), 
(f_1, \alpha_1), (f_2, \alpha_2), \ldots, 
\widehat{(f_i, \alpha_i)}, \ldots, 
\widehat{(f_j, \alpha_j)}, \ldots, 
(f_m, \alpha_m)\right).
\end{align*}}

Therefore, if $m$ is even, the sum (\ref{the_sum}) becomes
\begin{multline*}
(-1)^{\frac{m}{2}+1}\mathrm{d}\iota_{(u_m, v_m)} \cdots \iota_{(u_1, v_1)} (\omega_M, \omega_N)   
+ (-1)^{\frac{m}{2}} \sum_{1\leq i<j \leq m} (-1)^{i+j} \iota_{[(u_i, v_i), (u_j, v_j)]}
  \iota_{(u_1, v_1)} \\   \cdots 
\widehat{\iota_{(u_i,v_i)}}  \cdots 
  \widehat{\iota_{(u_j,v_j)}} \cdots \iota_{(u_m,v_m)}( \omega_M, \omega_N)
\end{multline*}
and, if $m$ is odd:
\begin{multline*}
(-1)^{\frac{m-1}{2}} \mathrm{d}\iota_{(u_m, v_m)} \cdots \iota_{(u_1, v_1)} (\omega_M, \omega_N)  
+ (-1)^{\frac{m-1}{2}} \sum_{1\leq i<j \leq m} (-1)^{i+j} \iota_{[(u_i, v_i), (u_j, v_j)]}
  \iota_{(u_1, v_1)} \\   \cdots 
\widehat{\iota_{(u_i,v_i)}}  \cdots 
  \widehat{\iota_{(u_j,v_j)}} \cdots \iota_{(u_m,v_m)}( \omega_M, \omega_N).
\end{multline*}
It then follows from Lemma \ref{lemma derivative} that, in either case, (\ref{the_sum}) is zero.
\end{itemize}
\end{proof}

\begin{thm} \label{theorem:Lie algebra of obs for pre-n-plectic}
Given a relative  pre-$n$-plectic structure \((F,\omega_M, \omega_N)\), there is a Lie $n$-algebra

\[
\resizebox{\textwidth}{!}{$
\begin{array}{ccccccccccccccc}
    0 & \longrightarrow & L_{n-1} & \longrightarrow & L_{n-2} & \longrightarrow & \cdots & \longrightarrow & L_{k-2} & \longrightarrow & \cdots & \longrightarrow & L_{1} & \longrightarrow & L_0 \\
     &  & \parallel &  & \parallel &  &  &  & \parallel &  &  &  & \parallel &  & \parallel \\
     &  & \Omega^{0}(F) &  & \Omega^{1}(F) &  &  &  & \Omega^{k-2}(F) &  &  &  & \Omega^{n-2}(F) &  & \widetilde{\Omega^{n-1}_{\textrm{Ham}}}(F)
\end{array}
$}
\]

denoted $\mathrm{\textbf{Ham}}_{\infty}(F,\omega_M, \omega_N)$, with 
\begin{itemize}
    \item underlying graded vector space \(L\):
\begin{equation*}
\begin{split}
L_{0} & = \widetilde{\Omega^{n-1}_{\mathrm{Ham}}}(F)=\{ (u,  v)\oplus (f,\alpha) \in \mathfrak{X}(F)\oplus \Omega_{\mathrm{Ham}}^{n-1}(F) \, \mid \, \mathrm{d}(f, \alpha)= -\iota_{(u,v)}(\omega_M, \omega_N) \} \\
L_{i} & = \Omega^{n-1-i}(F) ,\quad 0 \leq i < n-1,
\end{split}
\end{equation*}

and 
\item structure maps:
\begin{equation*}
\begin{split}
\tilde{l}_{1}(f, \alpha)&=
\begin{cases}
0\oplus \mathrm{d}(f, \alpha) & \text{if $\deg{(f, \alpha)}=1$},\\
\mathrm{d} (f, \alpha) & \text{if $\deg{(f, \alpha)} >1$,}
\end{cases}\\
\tilde{l}_{2}(x_{1},x_{2}) &= 
\begin{cases}
\left( [(u_1, v_{1}), (u_2, v_{2})], \{(f_1, \alpha_1), (f_2, \alpha_2)\} \right) &
\text{if $\deg(x_{1} \otimes x_{2})  = 0$,}\\
 0  & \text{otherwise},
\end{cases}
\end{split}
\end{equation*}
and, for $k > 2$: 
\[
l_k(x_1, \ldots, x_k) =
\begin{cases}
 0 & \text{if } \deg({x_1\otimes \cdots\otimes x_k})> 0\\
 \epsilon(k) \iota_{(u_k, v_k)}\cdots \iota_{(u_1, v_1)}  (\omega_M, \omega_N)  & \text{if } \deg({x_1\otimes \cdots\otimes x_k}) = 0
\end{cases}
\]
where \(x_i=(u_i,  v_i)\oplus (f_i,\alpha_i)\) in \(L_0\) for \(i=1,2, \ldots, k\),  and \(x_i=(f_i,\alpha_i)\) in \(L_j\), for \(j>0\), and for \(i=1,2, \ldots, k\).

\end{itemize}
\end{thm}

\begin{proof} 
Since the bracket of pairs of vector fields satisfies the Jacobi identity and the higher structure maps \((l_k)\) are identical to those of $L_{\infty}(M,\omega)$ given in Theorem \ref{Relative L-infinity algebra main thm}, then the proof of this theorem follows the same steps as in Theorem \ref{Relative L-infinity algebra main thm}, while taking into account the bracket of pairs of vector fields for \(l_1\) and \(l_2\).

\end{proof}

A similar result can be obtained by omitting the vector components. The following theorem provides a relative version of \cite[Theorem 4.6]{callies2016homotopy}.

\begin{thm}\label{n-algebra from relative structure}
Given a relative pre-\(n\)-plectic structure  \((F,\omega_M, \omega_N)\), there exists a Lie \(n\)-algebra 
\[
\resizebox{\textwidth}{!}{$
\begin{array}{ccccccccccccccc}
    0 & \longrightarrow & L_{n-1} & \longrightarrow & L_{n-2} & \longrightarrow & \cdots & \longrightarrow & L_{k-2} & \longrightarrow & \cdots & \longrightarrow & L_{1} & \longrightarrow & L_0 \\
     &  & \parallel &  & \parallel &  &  &  & \parallel &  &  &  & \parallel &  & \parallel \\
     &  & \Omega^{0}(F) &  & \Omega^{1}(F) &  &  &  & \Omega^{k-2}(F) &  &  &  & \Omega^{n-2}(F) &  & \Omega^{n-1}_{\textrm{Ham}}(F)
\end{array}
$}
\]
denoted again \(L_\infty(M, \omega) = (L, \{l_k\})\), with the underlying graded vector space:
\[
L_i =
\begin{cases} 
   \Omega^{n-1}_{\mathrm{Ham}}(F) & \text{if } i = 0, \\
    \Omega^{n-1-i}(F) & \text{if } 0 \leq i < n-1,
\end{cases}
\]
and structure maps
\[
l_k : L^{\otimes k} \to L \quad \text{for } 1 \leq k < \infty,
\]
defined as:
\[
l_1(\alpha) = \mathrm{d}(f,\alpha), \quad \text{if } \deg(f, \alpha)> 0,
\]
and
\[
\resizebox{\textwidth}{!}{$
l_k((f_1, \alpha_1), \ldots, (f_k, \alpha_k)) =
\begin{cases}
 0 & \text{if } \deg\left((f_1,\alpha_1)\otimes \cdots \otimes (f_k,\alpha_k)\right) > 0, \\[1ex]
 \epsilon(k) \, \iota_{(u_k, v_k)} \cdots \iota_{(u_1, v_1)} (\omega_M, \omega_N) & \text{if } \deg\left((f_1,\alpha_1)\otimes \cdots \otimes (f_k,\alpha_k)\right) = 0.
\end{cases}
$}
\]
for \(k > 1\), where \((u_i, v_{i})\) is any Hamiltonian vector field associated to \(f_i, \alpha_i) \in\Omega^{n-1}_{\mathrm{Ham}}(F)\).
\end{thm}

\begin{proof}
    The proof follows from \cite[Theorem 4.6]{callies2016homotopy} and the relative Cartan calculus (Section \ref{section:relative cartan magic formulas}).
\end{proof}

\begin{pro}\label{morphism between two Lie algebras}
Let $(F,\omega_M, \omega_N)$ be a pre-$n$-plectic structure. 
\begin{enumerate}
\item
The chain map
\begin{equation*}
\pi : \mathrm{\textbf{Ham}}_{\infty}(F,\omega_M, \omega_N) \twoheadrightarrow \mathfrak{X}_{\mathrm{Ham}}(F) \label{pi_map}
\end{equation*}
defined to be the projection $(u, v)\oplus (f,\alpha) \mapsto (u, v)$ in degree 0, and
trivial in all lower degrees lifts to a strict morphism of $L_{\infty}$-algebras.

\item 
If $(F,\omega_M, \omega_N)$ is a relative $n$-plectic structure, then the chain map
\[
\xymatrix{
\Omega^0(F) \ar[r]^-{\mathrm{d}} \ar[d]^{\mathrm{id}} & \Omega^{1}(F) \ar[d]^{\mathrm{id}} \ar[r]^-{\mathrm{d}} & \cdots \ar[r]^-{\mathrm{d}}
& \Omega^{n-2}(F) \ar[d]^{\mathrm{id}} \ar[r]^-{\mathrm{d}} & \Omega_{\mathrm{Ham}}^{n-1}(F) \ar[d]^{\phi} \\
\Omega^0(F) \ar[r]^-{\mathrm{d}} & \Omega^{1}(F) \ar[r]^-{\mathrm{d}} & \cdots \ar[r]^-{\mathrm{d}} &
\Omega^{n-2}(F) \ar[r]^-{0 \oplus \mathrm{d}} & \widetilde{\Omega_{\mathrm{Ham}}^{n-1}(F)} 
}
\]

with $\phi(f, \alpha) = (u, v) \oplus (f, \alpha)$, where $(u, v)$ is the unique Hamiltonian vector field associated to $(f, \alpha)$, lifts to a strict $L_{\infty}$-quasi-isomorphism 
\[
L_{\infty}(F,\omega_M, \omega_N) \xrightarrow{\sim} \mathrm{\textbf{Ham}}_{\infty}(F,\omega_M, \omega_N).
\]

\end{enumerate}
\end{pro}

\begin{proof} \qquad {}

\begin{enumerate}
        
    \item We have
    {\footnotesize\begin{align*}
        \pi \tilde{l}_{2} \left(  (u_1, v_{1})\oplus (f_1, \alpha_1), (u_2, v_{2})\oplus (f_2, \alpha_2) \right) &= \pi \left( [(u_1, v_{1}), (u_2, v_{2})], \{(f_1, \alpha_1), (f_2, \alpha_2)\} \right) \\
        &=\left[(u_1, v_{1}), (u_2, v_{2})\right]\\
         &=  \left[\pi\left((u_1, v_{1})\oplus (f_1, \alpha_1)\right), \pi\left( (u_2, v_{2})\oplus (f_2, \alpha_2)\right)\right].
    \end{align*}}

This proves that \(\pi : \mathrm{\textbf{Ham}}_{\infty}(F,\omega_M, \omega_N) \twoheadrightarrow \mathfrak{X}_{\mathrm{Ham}}(F)\)  lifts to a strict morphism of $L_{\infty}$-algebras as desired.

\item 
    We have

    \begin{align*}
        \tilde{l}_{2}\phi^{\otimes 2} \left(   (f_1, \alpha_1), (f_2, \alpha_2) \right) &=  \tilde{l}_{2} \left( \phi\left( f_1, \alpha_1\right), \phi\left( f_2, \alpha_2\right) \right)\\
         &=  \tilde{l}_{2} \left(  (u_1, v_{1})\oplus (f_1, \alpha_1),  (u_2, v_{2})\oplus (f_2, \alpha_2)\right) \\
          &=\left( [(u_1, v_{1}), (u_2, v_{2})], \{(f_1, \alpha_1), (f_2, \alpha_2)\} \right)\\
          &= \left([u_1, u_2]\oplus \{f_1, f_2\}, [v_1, v_2]\oplus \{\alpha_1,\alpha_2\}\right).
    \end{align*}

    and 
    \begin{align*}
        \phi l_{2} \left( (u_1, v_{1})\oplus (f_1, \alpha_1), (u_2, v_{2})\oplus (f_2, \alpha_2) \right) &= \phi \left(\iota_{(u_2, v_2)} \iota_{(u_1, v_1)}  (\omega_M, \omega_N)\right)       \\&=\phi \left( \{(f_1, \alpha_1), (f_2, \alpha_2)\} \right) \\
        &=\phi \left( \{f_1, f_2\}, \{\alpha_1,\alpha_2\} \right) \\
         &=  \left([u_1, u_2]\oplus \{f_1, f_2\}, [v_1, v_2]\oplus \{\alpha_1,\alpha_2\}\right)\\
         &=\tilde{l}_{2}\phi^{\otimes 2} \left(   (f_1, \alpha_1), (f_2, \alpha_2) \right).
    \end{align*}

    Hence,
    \begin{align*}
        \tilde{l}_{2} \phi^{\otimes 2} &= \phi l_{2}.
    \end{align*}
    
    Moreover, if \(k>2\) and \(\deg{(f_1,\alpha_1)\otimes \cdots\otimes (f_k,\alpha_k)} = 0\),
     \begin{align*}
        \tilde{l}_{k}\phi^{\otimes k} \left(   (f_1, \alpha_1),\ldots, (f_k, \alpha_k) \right) &=  \tilde{l}_{k} \left( \phi\left( f_1, \alpha_1\right),\ldots, \phi\left( f_k, \alpha_k\right) \right)\\
         &=  \tilde{l}_{k} \left(  (u_1, v_{1})\oplus (f_1, \alpha_1),\ldots,  (u_k, v_k)\oplus (f_k, \alpha_k)\right) \\
         &=\varsigma(k) \iota_{(u_k, v_k)}\cdots \iota_{(u_1, v_1)}  (\omega_M, \omega_N)\\
          &=l_k((f_1, \alpha_1), \ldots, (f_k, \alpha_k)).
    \end{align*}

    If  \(k>2\) and \(\deg{(f_1,\alpha_1)\otimes \cdots\otimes (f_k,\alpha_k)}>  0\), then \(l_k((f_1, \alpha_1), \ldots, (f_k, \alpha_k))=0\) and 
     \begin{align*}
        \tilde{l}_{k}\phi^{\otimes k} \left(   (f_1, \alpha_1),\ldots, (f_k, \alpha_k) \right) &=  \tilde{l}_{k} \left( \phi\left( f_1, \alpha_1\right),\ldots, \phi\left( f_k, \alpha_k\right) \right)\\
         &=  \tilde{l}_{k} \left(  (u_1, v_{1})\oplus (f_1, \alpha_1),\ldots,  (u_k, v_k)\oplus (f_k, \alpha_k)\right) \\
         &=0=l_k((f_1, \alpha_1), \ldots, (f_k, \alpha_k)).
    \end{align*}
    
    It follows that
    
\begin{align*}
    \tilde{l}_{k} \phi^{\otimes k} &= l_{k} \quad \forall k > 2.
\end{align*}

Hence, the cochain map lifts to a strict $L_{\infty}$-morphism.

\emph{Quasi-isomorphism on cohomology.}
In degrees $\ge 1$, $\phi_1=\mathrm{id}$, so it induces the identity on cohomology.
In degree~$0$ one computes
\[
H_0\big(L_\infty(F,\omega_M,\omega_N)\big)
=\Omega^{n-1}_{\mathrm{Ham}}(F)\big/\mathrm{im}(d),\quad
H_0\big(\mathrm{Ham}_\infty(F,\omega_M,\omega_N)\big)
=\widetilde{\Omega^{n-1}_{\mathrm{Ham}}(F)}\big/\mathrm{im}(0\oplus d).
\]
The map $\phi_1$ induces an isomorphism between these quotients, since
$(u,v)$ is uniquely determined by $(f,\alpha)$ (relative $n$-plectic nondegeneracy),
and the boundaries match via $(0\oplus d)$ on the target corresponding to $d$
on the source. Therefore $H_0(\phi_1)$ is an isomorphism.

\smallskip
Thus $\phi_\bullet$ is a strict $L_\infty$ quasi-isomorphism, completing the proof.
\end{enumerate}
\end{proof}

\section{The Differential Graded Leibniz algebra associated to a relative  \nplectic structure}
Using a similar construction as the previous, replacing the semi-bracket by the hemi-bracket leads to the notion of differential graded Leibniz algebras (see \cite[Proposition 6.3]{Rogers_2011}).

\begin{defn}\cite[Defintion 6.2]{Rogers_2011} \label{dg_leibniz algebra}
A \emph{differential graded Leibniz algebra}
$(L,\delta, \bra \cdot, \cdot\ket)$ is a graded vector
space $L$ equipped with a degree -1 linear map $\delta : L \to
L$ and a degree 0 bilinear map $\bra \cdot, \cdot\ket : L \otimes L
\to L$ such that the following identities hold:
\begin{gather}
\delta \circ \delta =0 \label{Leibniz derivative}\\
\delta \bra x, y\ket = \bra \delta x, y\ket +
(-1)^{|x|}  \bra x, \delta y\ket  \label{Leibniz bracket derivative 1}\\
 \bra x,   \bra y, z\ket\ket= \bra \bra x, y\ket, z\ket + (-1)^{|x| |y|}
\bra y,   \bra x, z\ket\ket\label{Leibniz Jacobi identity for hemi},
\end{gather}
for all $x,y,z \in L$.

\end{defn}

The following proposition presents the relative version of the construction given in \cite[Proposition 6.3]{Rogers_2011}:

\begin{pro}\label{proposition: Leibniz dg algebra}
Given a relative $n$-plectic structure  $(F,\omega_M, \omega_N)$, there is a differential
graded Leibniz algebra $\mathrm{Leib}(F,\omega_M, \omega_N)=(L,\delta, \bra \cdot, \cdot\ket)$ with 
\begin{itemize}
    \item 
underlying
graded vector space \(L_0=\Omega_{\mathrm{Ham}}^{n-1}(F)\) and \(L_i=\Omega^{n-1-i}(F)\) for  \(0 < i \leq n-1\),
 \item 
 and maps $ \delta : L \to L$, $\bra \cdot, \cdot\ket: L
\otimes L \to L$ defined    by
\[
\delta(f, \alpha)=\mathrm{d}(f, \alpha)=(F^*\alpha+\mathrm{d}f, -\mathrm{d}\alpha),
\]
if $\deg{(f, \alpha)} > 0$ and
\[
\bra (f, \alpha), (g, \beta)\ket=
\begin{cases}
\mathcal{L}_{(u_1, v_1)}(g, \beta) & \text{if $\deg{(f, \alpha)\otimes (g, \beta)}=0$,}\\
(0, 0) & \text{if $\deg{(f, \alpha)\otimes (g, \beta)} >0$,}
\end{cases}
\]
where $(u_1, v_1)$ is the Hamiltonian vector field associated to $(f, \alpha)$.
\end{itemize}
\end{pro}

\begin{proof}

Since \(\delta=\mathrm{d}\) is the relative differential, then it is a linear map of degree \(-1\). The bilinearity of \(\bra \cdot, \cdot\ket\) follows from the bilinearity of the hemi-bracket and the zero map. 

\paragraph{Next, let us show that the bilinear map \(\bra \cdot, \cdot\ket\) is of degree 0. }
\begin{itemize}

    \item If $\deg (f, \alpha)\otimes (g, \beta)=0$, then \(\bra (f, \alpha), (g, \beta)\ket=\mathcal{L}_{(u_1, v_1)}(g, \beta)\). Since the Lie
derivative is a degree zero derivation, it follows that
    \[\deg \left( \bra (f, \alpha)\otimes (g, \beta)\ket\right)=\deg \left(\mathcal{L}_{(u_1, v_1)}(g, \beta)\right)=\deg (g, \beta)=0. \]
    But
    \[\deg \left( (f, \alpha)\otimes (g, \beta)\right)= \deg   (f, \alpha)+\deg  (g, \beta)=0+\deg  (g, \beta)=\deg  (g, \beta).\]
    Hence 
    \[\deg \left( \bra (f, \alpha)\otimes (g, \beta)\ket\right)=\deg \left(  (f, \alpha)\otimes (g, \beta)\right). \]

    \item If $\deg (f, \alpha)>0$, then \((f, \alpha)\otimes (g, \beta)\in L_{m}=\Omega^{n-1-m}(F)\), where \(m= \deg{\left((f, \alpha)\otimes (g, \beta)\right)} \). But  \(\bra (f, \alpha), (g, \beta)\ket=(0, 0)\) is the zero \((n-1-m)\)-form which is in \(L_{m-1}=\Omega^{n-m}(F)\). Hence 
    \[\deg \left( \bra (f, \alpha)\otimes (g, \beta)\ket\right)=m-1=-1+\deg \left((f, \alpha)\otimes (g, \beta)\right). \]
    
    This proves that the bilinear map $\bra \cdot, \cdot\ket: L
\otimes L \to L$  is of degree 0 as desired.

\end{itemize}

\paragraph{Let us show that Equation \eqref{Leibniz derivative}, Equation \eqref{Leibniz bracket derivative 1}, and Equation \eqref{Leibniz Jacobi identity for hemi} hold.} 

\vspace{1cm}
For  Equation \eqref{Leibniz derivative}, it is clear that \(\delta\circ\delta =\mathrm{d}^2 =(0, 0)\) by Lemma \ref{lemma:d^2=0}.  

Let us show that Equation  \eqref{Leibniz bracket derivative 1} holds.
\begin{itemize}
    \item 

    If \( \deg(f, \alpha) > 1 \), then by definition, \(\bra (f, \alpha), (g, \beta) \ket = (0, 0)\), which implies that:
\[
\delta \bra (f, \alpha), (g, \beta) \ket = \delta (0, 0) = (0, 0).
\]
Additionally, since \(\delta\) is a linear map of degree \(-1\), it follows that \(\deg(\delta(f, \alpha)) > 0\). Therefore, by definition, we also have:
\[
\delta \bra (f, \alpha), \delta(g, \beta) \ket = (0, 0).
\]

Now, since \(\deg(g, \beta) = 0\), we have \(\delta(g, \beta) = (0, 0)\) by definition. This gives:
\[
\bra (f, \alpha), \delta(g, \beta) \ket = (0, 0).
\]
Thus, Equation \eqref{Leibniz bracket derivative 1} holds trivially, as all terms involved in the equation are equal to \((0, 0)\).

\item 
If $\deg{ (f, \alpha)} =1$, then
$\bra (f, \alpha), (g, \beta)\ket=\bra (f, \alpha), \delta(g, \beta)\ket=0$ for all $\delta(g, \beta) \in L$
by definition. Since \(\delta\) is a derivation of degree -1, then \(\deg(\delta (f, \alpha))=0\). Moreover, 
\[\mathrm{d}\delta (f, \alpha)=\mathrm{d}^2 (f, \alpha)=(0, 0)\]
so that the Hamiltonian vector field associated to $\delta (f, \alpha)$ is the zero vector. It follows that
\[\bra \delta (f, \alpha), (g, \beta)\ket=\mathcal{L}_{(0,0)}(g, \beta)=(0,0).\]
In this case, Equation \eqref{Leibniz bracket derivative 1} holds trivially, as all the terms involved are equal to \((0, 0)\).

 \item If $\deg{(f, \alpha)} = 0$ and $\deg{(g, \beta)} = 0$, then $\deg{\bra (f, \alpha), (g, \beta) \ket} = 0$ since \(\bra \cdot, \cdot\ket\) is of degree 0. It follows by definition of \(\delta\) that
\[
\delta \bra (f, \alpha), (g, \beta) \ket = (0, 0).
\]
Since $\delta$ is a derivation of degree $-1$, we have $\deg(\delta (f, \alpha)) = \deg(\delta (g, \beta)) = -1$, implying that $\delta(f, \alpha) = (0, 0)$ and $\delta(g, \beta) = (0, 0)$ by definition.

It follows that all the terms in Equation \eqref{Leibniz bracket derivative 1} vanish by definition, so that Equation \eqref{Leibniz bracket derivative 1} holds trivially.

\item If 
$\deg{(f, \alpha)}=0$ and $\deg{(g, \beta)} > 0$, then 
\begin{equation}\label{equation dervmm}
\delta \bra (f, \alpha), (g, \beta)\ket= \mathrm{d} \mathcal{L}_{(u_1, v_1)} (g, \beta) =
\mathcal{L}_{(u_1, v_1)} \mathrm{d}(g, \beta) =\bra (f, \alpha), \delta(g, \beta)\ket.
\end{equation}
Furthermore, since \(\deg{(f, \alpha)}=0\), then \(\delta (f, \alpha)=(0,0)\) by definition. It follows that 
\[ \bra \delta (f, \alpha), (g, \beta)\ket= \bra (0, 0), (g, \beta)\ket=(0,0).\]
\end{itemize}
 Equation \eqref{Leibniz bracket derivative 1} reduces to 
 \[\delta \bra (f, \alpha), (g, \beta)\ket= 
\mathcal{L}_{(u_1, v_1)} \mathrm{d}(g, \beta) =\bra (f, \alpha), \delta(g, \beta)\ket\]
which follows by \eqref{equation dervmm}. Hence, Equation \eqref{Leibniz bracket derivative 1} holds.

It remains only to show that Equation \eqref{Leibniz Jacobi identity for hemi} holds. For this, let $(f, \alpha), (g, \beta), (k, \gamma) \in L$. Let us show that
\[\bra (f, \alpha), \bra (g, \beta), (k, \gamma)\ket\ket=\bra \bra (f, \alpha), (g, \beta)\ket, (k, \gamma)\ket +
(-1)^{\deg{(f, \alpha)}\deg{(g, \beta)}}\bra (g, \beta), \bra (f, \alpha), (k, \gamma) \ket\ket.\]
\begin{itemize}
    \item If \(\deg{(f, \alpha)}>0\), then 
    \[\bra (f, \alpha), \bra (g, \beta), (k, \gamma)\ket\ket=\bra \bra (f, \alpha), (g, \beta)\ket, (k, \gamma)\ket =
\bra (g, \beta), \bra (f, \alpha), (k, \gamma) \ket\ket=(0,0),\]
so that Equation \eqref{Leibniz Jacobi identity for hemi} holds trivially.

 \item If \(\deg{(f, \alpha)}=0\) and \(\deg{(g, \beta)}>0\), then 
    \[\bra (f, \alpha), \bra (g, \beta), (k, \gamma)\ket\ket=
\bra (g, \beta), \bra (f, \alpha), (k, \gamma) \ket\ket=(0,0).\]
Also, \(\bra (f, \alpha), (g, \beta)\ket=\mathcal{L}_{(u_1, v_1)}(f, \beta)\) so that  \(\deg(\bra (f, \alpha), (g, \beta)\ket)=\deg (g, \beta) >0\) since the Lie derivative is a degree 0 derivation. Since  \(\deg(\bra (f, \alpha), (g, \beta)\ket) >0\), it follows, by definition, that 
\[\bra \bra (f, \alpha), (g, \beta)\ket, (k, \gamma)\ket =(0,0).\]
Hence Equation \eqref{Leibniz Jacobi identity for hemi} holds trivially as all the terms involved are zero.

\item 
If \(\deg{(f, \alpha)}=0\) and \(\deg{(g, \beta)}=0\), then 

\begin{align*}
\bra (f, \alpha), \bra (g, \beta), (k, \gamma)\ket\ket &= \mathcal{L}_{(u_1, v_1)}\bra (g, \beta), (k, \gamma)\ket\\
&=\mathcal{L}_{(u_1, v_1)}\mathcal{L}_{(u_2, v_2)} (k, \gamma)
\end{align*}
Now, using the identity
\[\mathcal{L}_{[(u_1, v_1), (u_2, v_2)]}=\mathcal{L}_{(u_1, v_1)}\mathcal{L}_{(u_2, v_2)}-\mathcal{L}_{(u_2, v_2)}\mathcal{L}_{(u_1, v_1)},\]
it follows that
\[\mathcal{L}_{(u_1, v_1)}\mathcal{L}_{(u_2, v_2)}(k, \gamma)=\mathcal{L}_{[(u_1, v_1), (u_2, v_2)]}(k, \gamma)+\mathcal{L}_{(u_2, v_2)}\mathcal{L}_{(u_1, v_1)}(k, \gamma),\]

so that
\begin{align*}
\bra (f, \alpha), \bra (g, \beta), (k, \gamma)\ket\ket 
&=\mathcal{L}_{(u_1, v_1)}\mathcal{L}_{(u_2, v_2)} (k, \gamma)\\
&=\mathcal{L}_{[(u_1, v_1), (u_2, v_2)]}(k, \gamma)+\mathcal{L}_{(u_2, v_2)}\mathcal{L}_{(u_1, v_1)}(k, \gamma)\\
&=\bra \bra (f, \alpha),  (g, \beta)\ket, (k, \gamma)\ket+\bra (g, \beta), \bra (f, \alpha), (k, \gamma)\ket\ket.
\end{align*}

This completes the proof.

\end{itemize}

\end{proof}

In this chapter, we have constructed \(L_\infty\)-algebra structures associated with relative observables (Theorem \ref{Relative L-infinity algebra main thm} and Theorem \ref{theorem:Lie algebra of obs for pre-n-plectic}), demonstrating that relative observables align naturally with the framework of observables in \(n\)-plectic (or multisymplectic) geometry. Our results extend the well-known constructions from multisymplectic geometry (\cite{rogers2011higher}, \cite{callies2016homotopy}) to the relative setting. Moreover, we have shown that the differential graded Leibniz algebra structure naturally emerges from relative \(n\)-plectic structures. 

As an application, we will explore in the next chapter how the relative \(3\)-form \((\omega, \eta)\) arising from quasi-Hamiltonian \(G\)-spaces induces a Lie \(2\)-algebra. This will lead us to the existence of a homotopy moment map. This will allow to view quasi-Hamiltonian \(G\)-spaces as relative versions of \(2\)-plectic manifolds.

%% file: Chapter5.tex
\chapter{Homotopy moment for the Lie \Twoalgebra arising from a quasi-Hamiltonian \Gspace}
\label{chapter 5}

In this chapter, we investigate the Lie \(2\)-algebra structure associated with a quasi-Hamiltonian \(G\)-space and construct the corresponding homotopy moment map.

Let \((M, \omega, \Phi)\) be a quasi-Hamiltonian \(G\)-space, where \(\omega \in \Omega^2(M)\) is a \(G\)-invariant 2-form and \(\Phi: M \to G\) is a group-valued moment map. The Cartan 3-form \(\eta \in \Omega^3(G)\) on the Lie group \(G\) combines with \(\omega\) to form a pair \((\omega, \eta) \in \Omega^2(M) \oplus \Omega^3(G)\). As established in Proposition~\ref{closed} and Theorem~\ref{Non-degeneracy general}, this pair defines a \emph{closed} and \emph{non-degenerate relative 3-form} with respect to the map \(\Phi\).

By Theorem~\ref{Relative L-infinity algebra main thm}, such a relative 3-form endows the space of relative observables with the structure of a Lie \(2\)-algebra. Furthermore, by definition of a quasi-Hamiltonian \(G\)-space, the Lie group \(G\) acts on the manifold \(M\), and this action induces a naturally associated infinitesimal action at the level of vector fields:
\[
(u, v): \mathfrak{g} \to \mathfrak{X}(\Phi),
\]
where \(\mathfrak{X}(\Phi)\) denotes the space of pairs \((u, v)\), with \(u \in \mathfrak{X}(M)\), \(v \in \mathfrak{X}(G)\), and such that \(u\) and \(v\) are \(\Phi\)-related.

The objective of this chapter is to prove that this infinitesimal action lifts to a  \emph{homotopy moment map} in the sense of \(L_\infty\)-algebra theory.

\section{Hemi-strict and Semi-Strict Lie \Twoalgebras corresponding to a  quasi-Hamiltonian \Gspace}

In this section, we explicitly describe this Lie \(2\)-algebra structure using the relative versions of the \emph{hemi-bracket} and \emph{semi-bracket}. These brackets yield two closely related Lie \(2\)-algebra structures—respectively \emph{hemi-strict} and \emph{semi-strict}—mirroring the constructions presented in Section~\ref{Lie 2-lagebra of a 2-plecticmanifold} for \(2\)-plectic manifolds. Our goal is to adapt those results to the relative setting defined by the moment map \(\Phi: M \to G\).

The following theorem gives a semi-strict Lie 2-algebra paralleling that arising from a 2-plectic manifold given in Theorem \ref{semistrict}.

\begin{thm} \label{theorem:semi strict Lie2}
Let $(M, \Phi, \omega)$ be a quasi-Hamiltonian $G$-space.  There exists a Lie $2$-algebra 
\[\begin{array}{ccccccccccccccc}
     \cdots 0& \longrightarrow & L_{1} & \longrightarrow & L_0\\
      &  & \parallel &  & \parallel \\
        &  & \Omega^{0}(\Phi) &  & \Omega^{1}_{\textrm{Ham}}(\Phi)\\ 
\end{array}\]
denoted $L_\infty(M,\Phi,\omega) = (L, [\cdot,\cdot], J)$, 
with the following structures:
\begin{itemize}
\item the underlying vector spaces are \(L_0 =\Omega^{1}( \Phi)\)

and 
 $L_1 = \Omega^{0}( \Phi) $;
\item the differential 
\begin{align*}
\mathrm{d} :   L_1 &\to L_0 \\
f &\mapsto ( \Phi^* f,  -\mathrm{d}f)
\end{align*}
is the relative differential;
\item the alternator is 
\begin{align*}
S: L_0\times L_0 & \to   L_1
\end{align*}
is identically zero:
  \[
  S((f_1, \beta_1), (f_2, \beta_2)) = 0;
  \]
\item the bracket $\left\{\cdot,\cdot\right\}_h: L_i \otimes L_j \to L_{i+j}$ for $i,j = 0,1$ and $i+j\leq 1$ is the hemi-bracket

\item the Jacobiator is the trilinear map, 
\begin{align*}
J : L_0 \otimes L_0 \otimes L_0 &\to   L_1 
\end{align*}
defined by
\[J\left(( f_1, \beta_1), \, ( f_2, \beta_2), \, ( f_3,  \beta_3)\right)=-\iota_{(u_1, v_1)} \iota_{ (u_2, v_2)}\iota_{(u_3, v_3)}(\omega, \eta);\]
where $(u_i, v_{_i})$ denotes the unique Hamiltonian vector field corresponding $(f_i, \beta_i)$ for $i=1,2,3$.
\end{itemize}
\end{thm}

 \begin{proof}

The proof follows from Theorem \ref{Relative L-infinity algebra main thm} in the case \(n=2\).

\end{proof}

\begin{lem}\label{hemi-hamil}
Let \( (f, \alpha),\, (g, \beta) \in \Omega^{1}_{\mathrm{Ham}}(\Phi) \) be relative Hamiltonian 1-forms with corresponding Hamiltonian vector fields \( (u_1, v_1) \) and \( (u_2, v_2) \), respectively. Then the hemi-bracket
\[
\left\{ (f, \alpha),\, (g, \beta) \right\}_{\mathrm{h}}
\]
is again a relative Hamiltonian 1-form. Moreover, its associated Hamiltonian vector field is the Lie bracket of the original ones:
\[
v_{\left\{ (f, \alpha), (g, \beta) \right\}_{\mathrm{h}}} = \left[ (u_1, v_1), (u_2, v_2) \right] = \left( [u_1, u_2],\, [v_1, v_2] \right).
\]
\end{lem}

\begin{proof}
  We begin by referring to Proposition \ref{hemi bracket-semi bracket}, we have:
  \[
    \{(f, \alpha), \, (g, \beta)\}_{\mathrm{s}} = \{(f, \alpha), \, (g, \beta)\}_{\mathrm{s}} + d\iota_{(u_1, v_1)}(g, \beta).
  \]

  With this in hand, let's calculate the exterior derivative \( \mathrm{d} \) of both sides:
  \begin{align*}
    \mathrm{d}\{(f, \alpha), \, (g, \beta)\}_\mathrm{h} &= \mathrm{d}\{(f, \alpha), \, (g, \beta)\}_{\mathrm{s}} + \mathrm{d}^2\iota_{(u_1, v_1)}(g, \beta) \\
    &= \mathrm{d}\{(f, \alpha), \, (g, \beta)\}_{\mathrm{s}}
  \end{align*}
  where we used the fact that \( \mathrm{d}^2 = 0 \). By Proposition \ref{proposition:bracket},  the Hamiltonian vector corresponding to \(\{(f, \alpha), \, (g, \beta)\}_{\mathrm{s}}\) is given by \([(u_1, v_1), \, (u_2, v_2)]\). It follows that 
  \[
    d\{(f, \alpha), \, (g, \beta)\}_{\mathrm{s}} = -\iota_{[(u_1, v_1), \, (u_2, v_2)]} (\omega, \, \eta).
  \]

  Therefore, 
  \begin{align*}
    \mathrm{d}\{(f, \alpha), \, (g, \beta)\}_\mathrm{h} &= \mathrm{d}\{(f, \alpha), \, (g, \beta)\}_{\mathrm{s}} + \mathrm{d}^2\iota_{(u_1, v_1)}(g, \beta) \\
    &= \mathrm{d}\{(f, \alpha), \, (g, \beta)\}_{\mathrm{s}}
  \end{align*}
  This proves that \( [(u_1, v_1), \, (u_2, v_2)] \) is the Hamiltonian vector field associated to the form \( \left\{ (f, \alpha), \, (g, \beta) \right\}_{\mathrm{h}} \), thereby completing the proof.
\end{proof}

The skew-symmetry of the hemi-bracket holds up to an exact term. More precisely, for any \((f, \alpha),\, (g, \beta) \in \Omega^{1}_{\mathrm{Ham}}(\Phi)\), we have
\[
\left\{ (f, \alpha),\, (g, \beta) \right\}_{\mathrm{h}} + \mathrm{d}S((f, \alpha), (g, \beta)) = -\left\{ (g, \beta),\, (f, \alpha) \right\}_{\mathrm{h}},
\]
where \(S\) is a bilinear map—called the \emph{alternator}—defined by
\begin{align}\label{alternator-def}
S : \Omega^{1}_{\mathrm{Ham}}(\Phi) \times \Omega^{1}_{\mathrm{Ham}}(\Phi) &\longrightarrow C^\infty(G), \\
((f, \alpha),\, (g, \beta)) &\longmapsto -\left( \iota_{(u_1, v_1)}(g, \beta) + \iota_{(u_2, v_2)}(f, \alpha) \right), \nonumber
\end{align}
with \((u_1, v_1)\) and \((u_2, v_2)\) denoting the Hamiltonian vector fields corresponding to \((f, \alpha)\) and \((g, \beta)\), respectively.
 The alternator \(S\) quantifies the exact correction required to restore strict skew-symmetry.

\begin{lem}\label{hemi-antil}
Let \( (f, \alpha),\, (g, \beta) \in \Omega^{1}_{\mathrm{Ham}}(\Phi) \) be relative Hamiltonian \(1\)-forms. Then the hemi-bracket satisfies the following identity:
\[
\left\{ (f, \alpha),\, (g, \beta) \right\}_{\mathrm{h}} + \mathrm{d}S\left((f, \alpha),\, (g, \beta)\right) 
= -\left\{ (g, \beta),\, (f, \alpha) \right\}_{\mathrm{h}},
\]
where the alternator map \(S\) is defined as in \eqref{alternator-def}.
\end{lem}

\begin{proof}
Let \((u_1, v_1)\) and \((u_2, v_2)\) be Hamiltonian vectors corresponding to \( (f, \alpha) \) and \( (g, \beta) \) respectively. By  Proposition \ref{hemi bracket-semi bracket}, we have:
    \begin{align*}
        \{(f, \alpha), \, (g, \beta)\}_{\mathrm{h}} &= \{(f, \alpha), \, (g, \beta)\}_{\mathrm{s}} + d\iota_{(u_1, v_1)}(g, \beta),
    \end{align*}

    as well as:
    \begin{align*}
        \{(g, \beta), \, (f, \alpha)\}_{\mathrm{h}} &= \{(g, \beta), \, (f, \alpha)\}_{\mathrm{s}} + d\iota_{(u_2, v_2)}(f, \alpha) \\
        &= -\{(f, \alpha), \, (g, \beta)\}_{\mathrm{s}} + d\iota_{(u_2, v_2)}(f, \alpha) &&\text{(by skew-symmetry of \( \{\cdot, \cdot\}_{\mathrm{s}} \))}.
    \end{align*}

    Summing these yields:
    \begin{align*}
        \{(f, \alpha), \, (g, \beta)\}_{\mathrm{h}} + \{(g, \beta), \, (f, \alpha)\}_{\mathrm{h}} &=-\{(f, \alpha), \, (g, \beta)\}_{\mathrm{s}}+ d\iota_{(u_1, v_1)}(g, \beta) \\
        &+\{(f, \alpha), \, (g, \beta)\}_{\mathrm{s}}+ d\iota_{(u_2, v_2)}(f, \alpha) \\
        &= d \left( \iota_{(u_1, v_1)}(g, \beta) + \iota_{(u_2, v_2)}(f, \alpha) \right) \\
        &= -dS\left((f, \alpha), \, (g, \beta)\right).
    \end{align*}
This completes the proof of the lemma.
\end{proof}

\begin{rem}
Lemma~\ref{hemi-antil} demonstrates that the hemi-bracket fails to satisfy strict skew-symmetry; rather, it is skew-symmetric only up to an exact form determined by the alternator map \(S\). In contrast, the semi-bracket is strictly skew-symmetric, as shown in Proposition~\ref{proposition:bracket}.

However, the roles reverse when it comes to the Jacobi identity. As established in Proposition~\ref{Jacobi-identity-up-toexact-form}, the semi-bracket satisfies the Jacobi identity only up to an exact form, reflecting its semi-strict \(L_\infty\)-algebra structure. On the other hand, the hemi-bracket satisfies the Jacobi identity strictly, as we will see in the following theorem.

\end{rem}

\begin{thm}[Hemi-bracket and the Jacobi Identity]\label{theorem:jacobi-hemi}
The hemi-bracket \( \{ \cdot, \cdot \}_{\mathrm{h}} \) satisfies the Jacobi identity strictly. Specifically, for any three relative Hamiltonian 1-forms \( (f, \alpha),\, (g, \beta),\, (k, \gamma) \in \Omega^{1}_{\mathrm{Ham}}(M,\, \Phi) \), the following identity holds:
\begin{align*}
\left\{(f, \alpha),\, \left\{ (g, \beta),\, (k, \gamma) \right\}_{\mathrm{h}} \right\}_{\mathrm{h}} 
= \left\{ \left\{ (f, \alpha),\, (g, \beta) \right\}_{\mathrm{h}},\, (k, \gamma) \right\}_{\mathrm{h}} 
+ \left\{ (g, \beta),\, \left\{ (f, \alpha),\, (k, \gamma) \right\}_{\mathrm{h}} \right\}_{\mathrm{h}}.
\end{align*}
\end{thm}

\begin{proof}
For simplicity, let us denote by \(F= (f, \alpha) \), \(G=(g, \beta) \), and \(H=(k, \gamma )\) in $\Omega^{1}_{\mathrm{Ham}}( \Phi)$ and $v_F$, $v_G$, and $v_H$ the corresponding pairs of Hamiltonian vectors associated to $F$, $G$, and $H$, respectively. We want to show that
\[
\{\{F, G\}_{\mathrm{h}}, H\}_{\mathrm{h}} + \{G, \{ F, H\}_{\mathrm{h}}\}_{\mathrm{h}} = \{F, \{G, H\}_{\mathrm{h}}\}_{\mathrm{h}}.
\]

Given that the hemi-bracket is defined as \(\{F, G\}_{\mathrm{h}} = \mathcal{L}_{v_F} G\), where \(v_F\) is the Hamiltonian vector field associated with \(F\) and \(\mathcal{L}_{v_F}\) is the Lie derivative along \(v_F\), we have:

\begin{align*}
\{\{F, G\}_{\mathrm{h}}, H\}_{\mathrm{h}} + \{G, \{ F, H\}_{\mathrm{h}}\}_{\mathrm{h}}  & = \mathcal{L}_{v_{\{F, G\}_{\mathrm{h}}}} H + \mathcal{L}_{v_G} \{ F, H\}_{\mathrm{h}}\\
   &=\mathcal{L}_{[v_F, v_G]} F + \mathcal{L}_{v_G} \mathcal{L}_{v_H} F  \quad \text{by Lemma \ref{hemi-hamil}}\\
   &=\mathcal{L}_{v_H}\mathcal{L}_{v_G}  F  \quad \text{by Proposition \ref{relative cartan magic formula}}\\
   &=\mathcal{L}_{v_H}\{G, H\}_{\mathrm{h}}  \quad \text{by definition}\\
   &=\{F, \{G, H\}_{\mathrm{h}}\}_{\mathrm{h}} \quad \text{by definition}.
\end{align*}
Therefore, the hemi-bracket satisfies the Jacobi identity as desired.
\end{proof}

In line with the work of Baez, Hoffnung, and Rogers \cite[Theorem 4.4]{Baez_2009}, we have the following Lie 2-algebra in the relative cohomology as expected. 

\begin{thm} \label{theorem:hemi strict Lie2}
Let $(M, \Phi, \omega)$ be a quasi-Hamiltonian $G$-space.  There exists a Lie $2$-algebra 
\[\begin{array}{ccccccccccccccc}
     \cdots 0& \longrightarrow & L_{1} & \longrightarrow & L_0\\
      &  & \parallel &  & \parallel \\
        &  & \Omega^{0}(\Phi) &  & \Omega^{1}_{\textrm{Ham}}(\Phi)\\ 
\end{array}\]
denoted $L_\infty(M,\Phi,\omega)_h = (L, [\cdot,\cdot], J)$, 
with the following structures:
\begin{itemize}
\item the underlying vector spaces are \(L_0 =\Omega^{1}( \Phi)\)
and 
 $L_1 = \Omega^{0}( \Phi) = \{0\}\times C^\infty(G)$;
\item the differential 
\begin{align*}
\mathrm{d} :   L_1 &\to L_0 \\
f &\mapsto ( \Phi^* f,  -\mathrm{d}f)
\end{align*}
is the relative differential;
\item the alternator is the map
\begin{align*}
S: L_0\times L_0 & \to   L_1
\end{align*}
defined by 
\[S
\left(( f_1, \beta_1), \, ( f_2, \beta_2)\right) = -\left( \iota_{(u_{1},\, {v_1})}\left(f_2, \,\beta_2\right)+ \iota_{(u_{2},\, {v_2})}\left(f_1, \,\beta_1\right)\right)\]
\item the bracket $\left\{\cdot,\cdot\right\}_h: L_i \otimes L_j \to L_{i+j}$ for $i,j = 0,1$ and $i+j\leq 1$ is the hemi-bracket

\item the Jacobiator is the identity chain homotopy map, hence given by the trilinear map 
\begin{align*}
J : L_0 \otimes L_0 \otimes L_0 &\to   L_1 
\end{align*}
defined by
\[J\left(( f_1, \beta_1), \, ( f_2, \beta_2), \, ( f_3,  \beta_3)\right)=(0, 0);\]
where $(u_i, v_{_i})$ denotes the unique hamiltonian vector field corresponding $(f_i, \beta_i)$ for $i=1,2,3$.
\end{itemize}
\end{thm}

\begin{proof}
The result follows from the general construction of Lie \(2\)-algebras of observables given in \cite[Theorem 4.4]{Baez_2009}, together with the relative Cartan calculus and the properties of the hemi-bracket established in the previous lemmas and propositions.
\end{proof}

\begin{thm}\label{isomorphism}
The Lie \(2\)-algebras \( L_\infty(M, \omega, \Phi)_{\mathrm{h}} \) and \( L_\infty(M, \omega, \Phi)_{\mathrm{s}} \), constructed using the hemi-bracket and the semi-bracket respectively, are isomorphic as Lie \(2\)-algebras.
\end{thm}

\begin{proof}
The result follows from the general equivalence between hemi-strict and semi-strict Lie \(2\)-algebras established in \cite[Theorem 4.6]{Baez_2009}, together with the adaptation of Cartan calculus to the relative setting as developed in the preceding sections.
\end{proof}

\section {The Atiyah Relative Lie \Twoalgebra and Courant Relative Lie \Twoalgebra}\label{Realtive-Atiya-Courant}

We now examine two important examples of Lie \(2\)-algebras in the relative setting: the Atiyah and Courant relative Lie \(2\)-algebras. In this section, we adapt the classical constructions of the Atiyah and Courant Lie \(2\)-algebras to the relative setting defined by the moment map \(\Phi: M \to G\).

\begin{thm}[Relative Atiyah Lie \(2\)-Algebra]\label{theorem: atiyah lie 2-algebra}
Let \((M, \omega, \Phi)\) be a quasi-Hamiltonian \(G\)-space. Then there exists a Lie \(2\)-algebra
\[\begin{array}{ccccccccccccccc}
     \cdots 0& \longrightarrow & L_{1} & \longrightarrow & L_0\\
      &  & \parallel &  & \parallel \\
        &  & \Omega^{0}(\Phi) &  & \mathfrak{X}(\Phi)\\ 
\end{array}\]
denoted \(L(M, \omega, \Phi)_{\mathfrak{atiyah}} = (L, \mathrm{d}, \bra \cdot, \cdot \ket_\mathfrak{a}, J)\), called the \emph{Atiyah Lie 2-algebra}, with the following structure:
\begin{itemize}
\item The underlying vector spaces are \(L_0 = \mathfrak{X}(\Phi)\) and \(L_1 = \Omega^0(\Phi)\).
\item The differential \(\mathrm{d}: L_1 \to L_0\) is defined by \(\mathrm{d}(0,\, f) = 0\).
\item The bracket operation \(\bra \cdot, \cdot \ket_\mathfrak{a}: L_i \otimes L_j \to L_{i+j}\) for \(i,j = 0,1\) and \(i+j\leq 1\) is defined by:
\[
\begin{aligned}
    &  \bra (u_1, v_1), (u_2, v_2) \ket_\mathfrak{a} = \left( [u_1, u_2], [v_1, v_2] \right), \\
    &  \bra (u, v), (0, f) \ket_\mathfrak{a} = \left(0, \mathcal{L}_{v} f \right),   \\
    &  \bra (0, f), (u, v) \ket_\mathfrak{a} = \left(0, 0 \right), \\
    &  \bra (0, f), (0, g) \ket_\mathfrak{a} = \left(0, 0 \right),
\end{aligned}
\]
with all other brackets being zero by degree reasons.

\item The Jacobiator is the trilinear map
    \[
    J : L_0 \times L_0 \times L_0 \to L_1
    \]
    given by
    \[
    J\left((u_1, v_1),\, (u_2, v_2),\, (u_3, v_3)\right) = -\iota_{(u_1, v_1)} \iota_{(u_2, v_2)} \iota_{(u_3, v_3)}(\omega, \eta),
    \]
    where \((\omega, \eta)\) is the relative 3-form defined by the quasi-Hamiltonian structure.
\end{itemize}
\end{thm}

\begin{proof}
We begin by verifying that the bracket operation \( \bra \cdot, \cdot \ket_\mathfrak{a} \) is well-defined, skew-symmetric, and degree-compatible.

Let \( (u_i, v_i) \in \mathfrak{X}(\Phi) \) for \( i = 1, 2 \). Since the Lie bracket of \(\Phi\)-related vector fields remains \(\Phi\)-related (cf.~Proposition~\ref{brackets F-related}), we have:
\[
\bra (u_1, v_1),\, (u_2, v_2) \ket_\mathfrak{a} = ([u_1, u_2],\, [v_1, v_2]) \in \mathfrak{X}(\Phi).
\]
Skew-symmetry follows from the skew-symmetry of the Lie bracket. Furthermore, as \(\deg(\bra \cdot, \cdot \ket_\mathfrak{a}) = 0\), the bracket preserves the grading on \(L_0\) and \(L_1\).

Next, we verify that the differential \( \mathrm{d}: L_1 \to L_0 \), defined by \( \mathrm{d}(0, f) = 0 \), is a chain map. For any \( (u, v) \in \mathfrak{X}(\Phi) \) and \( (0, f) \in \Omega^0(\Phi) \), we compute:
\[
\mathrm{d} \bra (u, v),\, (0, f) \ket_\mathfrak{a} = \mathrm{d}(0, \mathcal{L}_{v} f) = 0 = \bra \mathrm{d}(0, f),\, (u, v) \ket_\mathfrak{a},
\]
establishing that \( \mathrm{d} \) is compatible with the bracket.

To verify the Jacobi identity up to homotopy, we consider the Jacobiator
\[
J : L_0 \times L_0 \times L_0 \to L_1,
\]
defined by
\[
J\left((u_1, v_1),\, (u_2, v_2),\, (u_3, v_3)\right) = -\iota_{(u_1, v_1)}\iota_{(u_2, v_2)}\iota_{(u_3, v_3)}(\omega, \eta).
\]
This trilinear map provides the correction term ensuring that the Jacobi identity holds up to homotopy. That is, for all \( x, y, z, w \in L_0 \), the following coherence condition is satisfied:
\begin{align*}
&\bra x, J(y, z, w) \ket_\mathfrak{a} + J(x, \bra y, z \ket_\mathfrak{a}, w) + J(x, z, \bra y, w \ket_\mathfrak{a}) + \bra J(x, y, z), w \ket_\mathfrak{a} + \bra z, J(x, y, w) \ket_\mathfrak{a} \\
&= J(x, y, \bra z, w \ket_\mathfrak{a}) + J(\bra x, y \ket_\mathfrak{a}, z, w) + \bra y, J(x, z, w) \ket_\mathfrak{a} + J(y, \bra x, z \ket_\mathfrak{a}, w) + J(y, z, \bra x, w \ket_\mathfrak{a}).
\end{align*}

Lastly, the alternator in this Lie \(2\)-algebra is trivial—that is, the identity chain homotopy—so it satisfies the required symmetry conditions automatically.

Thus, the data \( (L,\, \mathrm{d},\, \bra \cdot, \cdot \ket_\mathfrak{a},\, J) \) defines a Lie \(2\)-algebra structure.
\end{proof}

The following theorem establishes the Courant Lie \(2\)-algebra in the relative setting, as determined by the closed form \( (\omega, \eta) \) arising from a quasi-Hamiltonian \(G\)-space.

\begin{thm}[Relative Courant Lie \(2\)-Algebra]\label{theorem: courant lie 2-algebra}
Let \((M, \Phi, \omega)\) be a quasi-Hamiltonian \(G\)-space. Then there exists a Lie \(2\)-algebra,
\[
\begin{array}{ccccccccccccccc}
     \cdots \to 0 & \longrightarrow & L_{1} & \xrightarrow{\mathrm{d}} & L_0 \\
      &  & \parallel &  & \parallel \\
      &  & \Omega^{0}(\Phi) &  & \mathfrak{X}(\Phi) \oplus \Omega^1(\Phi)
\end{array}
\]
denoted \(L_\infty(M, \omega, \Phi)_{\mathfrak{courant}} = (L, \mathrm{d}, \bra \cdot, \cdot \ket_\mathfrak{c}, J)\), with the following structure:
\begin{itemize}
    \item The underlying graded vector space is:
    \[
    L_0 = \mathfrak{X}(\Phi) \oplus \Omega^1(\Phi), \quad L_1 = \Omega^0(\Phi).
    \]

    \item The differential \( \mathrm{d} : L_1 \to L_0 \) is the relative de Rham differential:
    \[
    \mathrm{d}(f) = \left(\Phi^* f,\ - \mathrm{d}f\right).
    \]

    \item The binary and ternary brackets \( \bra \cdot, \cdot \ket_\mathfrak{c} \) are defined as follows:
   \[
\begin{aligned}
&\bra f \ket_1^{\mathfrak{c}}= \left(\Phi^*f, - d f\right)\\
    & \bra (u, v) + (g, \theta), (0, f) \ket_2^{\mathfrak{c}} = \left(0, \, -\frac{1}{2} \iota_{v} d f\right), \\
    & \bra (u_1, v_1) + (g_1, \theta_1), (u_2, v_2) + (g_2, \theta_2)\ket_3^{\mathfrak{c}} =\left( [u_1, u_2], \, [v_1, v_2]\right)+\left(  \mathcal{L}_{(u_1, v_1)}(g_2, \theta_2) - \mathcal{L}_{(u_2, v_2)}(g_1, \theta_1) \right.\\
   &\qquad\qquad\qquad\qquad\left. - \frac{1}{2} d(\iota_{ (u_1, v_1)}(g_2, \theta_2)-\iota_{(u_2, v_2)}(g_1, \theta_1)-\iota_{(u_1, v_1)} \iota_{ (u_2, v_2)}(\omega, \eta)\right),\\
     & \bra (u_1, v_1) + (g_1, \theta_1), (u_2, v_2) + (g_2, \theta_2), (u_3, v_3) + (g_3, \theta_3)\ket_3^{\mathfrak{c}} \\
     &=  - \frac{1}{6} \left(\langle \bra (u_1, v_1) + (g_1, \theta_1), (u_2, v_2) + (g_2, \theta_2)\ket^\mathfrak{c}_2 , (u_3, v_3) + (g_3, \theta_3)\rangle +\mathrm{cyc. perm.}\right)
\end{aligned}
\]
    for \( f \in \Omega^0(\Phi) \), \( x_i = (u_i, v_i) + (g_i, \theta_i) \in \mathfrak{X}(\Phi) \oplus \Omega^1(\Phi) \), \( i = 1,2,3 \). Here, “cyc.\ perm.” stands for the sum over the cyclic permutations of the inputs: \((x_1, x_2, x_3) \mapsto (x_2, x_3, x_1)\) and \((x_3, x_1, x_2)\).

    \item The symmetric pairing \( \langle \cdot, \cdot \rangle : L_0 \otimes L_0 \to L_1 \) is defined by:
    \[
    \langle (u_1, v_1) + (g_1, \theta_1),\ (u_2, v_2) + (g_2, \theta_2) \rangle := \iota_{(u_1, v_1)}(g_2, \theta_2) + \iota_{(u_2, v_2)}(g_1, \theta_1).
    \]
\end{itemize}
All other brackets vanish for degree reasons.

 \[
    \bra x_1, x_2, x_3 \ket_3^{\mathfrak{c}} := -\frac{1}{6} \sum_{\sigma \in \mathbb{Z}_3} 
    \left\langle \bra x_{\sigma(1)}, x_{\sigma(2)} \ket_2^{\mathfrak{c}},\ x_{\sigma(3)} \right\rangle,
    \]
    where the sum is over cyclic permutations \(\sigma \in \mathbb{Z}_3\). That is,
    \[
    \bra x_1, x_2, x_3 \ket_3^{\mathfrak{c}} = -\frac{1}{6} \left(
    \left\langle \bra x_1, x_2 \ket_2, x_3 \right\rangle + 
    \left\langle \bra x_2, x_3 \ket_2, x_1 \right\rangle +
    \left\langle \bra x_3, x_1 \ket_2, x_2 \right\rangle \right).
    \]
    \textit{In other words, “cyc. perm.” refers to summing over the three cyclic permutations of the inputs.}

    \item The symmetric pairing \(\langle \cdot, \cdot \rangle : L_0 \otimes L_0 \to L_1\) is defined by:
    \[
    \langle (x_1, \alpha_1),\ (x_2, \alpha_2) \rangle := \iota_{x_1} \alpha_2 + \iota_{x_2} \alpha_1.
    \]
\end{thm}

These two relative Lie \(2\)-algebras—the Atiyah and Courant Lie \(2\)-algebras associated to a quasi-Hamiltonian \(G\)-space—are connected by a natural sequence of \(L_\infty\)-morphisms. The following proposition describes this relationship explicitly.

\begin{pro}\label{prop:Courant-Atiyah-morphisms}
Let \((M, \Phi, \omega)\) be a quasi-Hamiltonian \(G\)-space. Let 
\[
\mathfrak{atiyah}(M, \Phi, \omega) \quad \text{and} \quad \mathfrak{courant}(M, \Phi, \omega)
\]
denote the Atiyah and Courant relative Lie \(2\)-algebras, respectively. Then there exists a natural sequence of \(L_\infty\)-morphisms:
\[
L_\infty(M, \Phi, \omega) \xrightarrow{\phi} \mathfrak{courant}(M, \Phi, \omega) \xrightarrow{\psi} \mathfrak{atiyah}(M, \Phi, \omega),
\]
where the nontrivial components of the morphism \(\phi\) are given by:
\begin{align*}
    &\phi_1\big((u, v) + (g, \theta)\big) = (v, \Phi_*v) + (g, \theta), \qquad \phi_1((0, f)) = (0, f), \\
    &\phi_2\left( (v_1, \Phi_*v_1) + (g_1, \theta_1),\, (v_2, \Phi_*v_2) + (g_2, \theta_2)\right) 
    = -\tfrac{1}{2} \left( \iota_{(v_1, \Phi_*v_1)}(g_2, \theta_2) - \iota_{(v_2, \Phi_*v_2)}(g_1, \theta_1) \right),
\end{align*}
and the nontrivial components of the morphism \(\psi\) are:
\begin{align*}
    &\psi_1\big((v, \Phi_*v) + (g, \theta)\big) = (v, \Phi_*v), \qquad \psi_1((0, f)) = (0, f), \\
    &\psi_2\left( (v_1, \Phi_*v_1) + (g_1, \theta_1),\, (v_2, \Phi_*v_2) + (g_2, \theta_2) \right) 
    = -\tfrac{1}{2} \left( \iota_{(v_1, \Phi_*v_1)}(g_2, \theta_2) - \iota_{(v_2, \Phi_*v_2)}(g_1, \theta_1) \right).
\end{align*}
\end{pro}

\begin{proof}
The verification of the \(L_\infty\)-morphism conditions for both \(\phi\) and \(\psi\) follows by applying the relative Cartan calculus and using the structure identities of the Courant and Atiyah Lie \(2\)-algebras. The calculations mirror those found in Proposition~\ref{Atiyah-Courant}, adapted to the relative setting.
\end{proof}

\section{Homotopy moment maps for quasi-Hamiltonian spaces}
\label{sec:homotopy_moment_maps}

In this section, we construct a homotopy moment map associated with the Lie $2$-algebra of relative observables arising from a quasi-Hamiltonian $G$-space. We begin by recalling the notion of morphisms from Lie algebras to Lie $2$-algebras.

Let $\mathfrak{g}$ be a Lie algebra and $(L, l_1, l_2, l_3)$ a Lie $2$-algebra. We are particularly interested in $L_\infty$-morphisms $\mathfrak{g} \to L$ where the $2$-bracket satisfies the following vanishing condition:

\begin{equation}
\label{eq:vanishing_condition}
\text{(P)}\qquad \text{for $k=2,3$} \quad l_k(x_1,\ldots,  x_k) = 0 \quad \text{whenever} \quad \sum_{i=0}^k|x_i|  > 0.
\end{equation}

\begin{defn}\cite[Propsition 3.8]{callies2016homotopy}\label{def:lie2-morphism}
Let \( (\mathfrak{g}, [\cdot, \cdot]) \) be a Lie algebra and \((L, l_1, l_2, l_3)\) a Lie 2-algebra satisfying property (P). A \emph{morphism} \(f: \mathfrak{g} \to L \) consists of:
\begin{itemize}
\item A degree 0 linear map \( f_1: \mathfrak{g} \to L \)
\item A degree 1 skew-symmetric bilinear map \( f_2: \mathfrak{g} \otimes \mathfrak{g} \to L \)
\end{itemize}
satisfying the following conditions for all \( x_i \in \mathfrak{g} \):
\begin{align}
&f_1([x_1, x_2])=l_1f_2(x_1, x_2) +l_2\left(f_1(x_1), f_1(x_2)\right) \label{eq:lie2-cond1} \\
&f_2([x_1, x_2], x_3)-f_2([x_1, x_3], x_2)+f_2([x_2, x_3], x_1) = l_3\left(f_1(x_1), f_1(x_2), f_1(x_2)\right) \label{eq:lie2-cond2} 
\end{align}
\end{defn}

\begin{defn}\cite[Definition 5.1]{callies2016homotopy} 
Let \((M,\omega)\) be an \(n\)-plectic manifold.	Let $v_{-}:\mathfrak g\to \mathfrak X(M)$ be a multisymplectic Lie algebra action.
	A \emph{homotopy moment map} corresponding to $v_{-}$ is 
	an $L_\infty$-morphism $(f)=\{f_i\}_{i=1, \ldots, n}$ from $\mathfrak g$ to $L_\infty(M,\omega)$ satisfying $$df_1(x)=-\iota_{v_x}\omega \qquad \forall x\in\mathfrak{g}.$$ 
	
\end{defn}
More conceptually, a homotopy moment is an $L_\infty$-morphism $(f):\mathfrak{g}\to L_\infty(M,\omega)$ lifting the action $v:\mathfrak{g}\to \mathfrak{X}(M)$, 
 that is, an $L_\infty$-morphism $(f):\mathfrak{g}\to L_\infty(M,\omega)$ making the following diagram commute (in the category of $L_\infty$-algebras):

\begin{equation*} 
\label{the_lift}
\xymatrix{
    &&  L_\infty (M,\omega) \ar[d]^{\pi} \\
     \mathfrak{g} \ar @{-->}[urr]^{(f)} \ar[rr]^{v_{-}} && \mathfrak{X}_{\mathrm{Ham}}(M)
}
\end{equation*}
where the cochain map \(\pi : L_\infty (M,\omega)  \to \mathfrak{X}_{\mathrm{Ham}}(M)\)
defined to be the projection \((v, \alpha) \mapsto v\) in degree 0 and trivial in all lower degrees.

\vspace{1cm}

 Let \( (M, \Phi, \omega) \) be a quasi-Hamiltonian \( G \)-space and defined in Defintion \ref{def:quasiHamiltonianGSpace}, and let \( \mathfrak{g} \) denote the Lie algebra corresponding to \(G\). The action of \(G\) on \(M\) induces an infinitesimal action \begin{align*}
  u_{-}:\mathfrak{g} &\to  \mathfrak{X}(M)\\ 
   x & \mapsto  u_{x}
  \end{align*}
  Moreover, $G$ acts on itself by conjugation, so it induces an infinitesimal action
\begin{align*}
  v_{-}:\mathfrak{g} &\to  \mathfrak{X}(G)\\ 
   x & \mapsto  v_{x}
  \end{align*}
where \(v_{x}=x^L-x^R\) denotes the difference between the left and right invariant vector fields associated with \(x\).

 Since $\Phi$ is $G$-equivariant, then $u_x$ is $\Phi$-related to $v_x$ for each $x\in \mathfrak{g}$. Combining these two actions, we obtain an infinitesimal action on the space of \(\Phi\)-related vector fields
  \begin{align*}
  (u, v):\mathfrak{g} &\to  \mathfrak{X}(\Phi)\\ 
   x & \mapsto  \left(u_x, v_{x}\right)
  \end{align*}

\begin{lem}\label{lemma homotopy relative}
   Let \( (M, \Phi, \omega) \) be a quasi-Hamiltonan \(G\)-space. The induced action  
     \begin{align*}
  (u, v):\mathfrak{g} &\to  \mathfrak{X}(\Phi)\\ 
   x & \mapsto  \left(u_x, v_{x} \right)
  \end{align*}
 is Hamiltonian.
\end{lem}

\begin{proof}
    Let \(x\in \mathfrak{g}\). We want to prove that \((u_x, v_{x})\) is Hamiltonian. That is, want to find \((\alpha, \beta) \in \Omega^1(\Phi)\) such that 
    \[\mathrm{d}(\alpha, \beta)=-\iota_{(u_x, v_{x})} (\omega, \eta).\]
    That is 
    \[\left(\Phi^*\beta +\mathrm{d}\alpha, -\mathrm{d} \beta \right)=-\left(\iota_{u_x }\omega, -\iota_{ v_{x}}\eta \right).\]
    We need to solve for \(\alpha\) and \(\beta\) the equations
    \begin{align}
        \Phi^*\beta +\mathrm{d}\alpha&= -\iota_{u_x }\omega \label{5.5}\\
       \mathrm{d} \beta&=- \iota_{ v_{x}}\eta \label{5.6}
    \end{align}
    By definition of a quasi-Hamiltonian $G$-space, we have 
    \begin{align} \label{5.7}
        \iota_{u_x}\omega =  \Phi^*\left(\frac{\theta^L + \theta^R}{2}\right) \cdot x 
    \end{align} for all  $x\in \mathfrak{g}$

From equation \eqref{5.7},  $\beta=-\frac{\theta^L + \theta^R}{2}$ and $\alpha=0$ is a solution of \eqref{5.5}. It is remaining only to show that equation \eqref{5.6} is satisfied. For this, we have

\begin{align*}
\iota_{ v_{x}}\eta &=\frac{1}{12}\left( \iota_{ v_{x}}\theta^{L}\cdot [\theta^{L},\theta^{L}]-  \theta^{L}\cdot \iota_{ v_{x}}[\theta^{L},\theta^{L}]\right) \\
&=\frac{1}{12}\left(\iota_{ v_{x}}\theta^{L}\cdot [\theta^{L},\theta^{L}]- \theta^{L}\cdot [\iota_{ v_{x}}\theta^{L},\theta^{L}]+ \theta^{L}\cdot [\theta^{L},\iota_{ v_{x}}\theta^{L}]\right) \\
&=\frac{1}{4} \iota_{ v_{x}}\theta^{L}\cdot [\theta^{L},\theta^{L}] \\
&=-\frac{1}{2} \iota_{ v_{x}}\theta^{L}\cdot  d\theta^{L}
\end{align*}

For the conjugation action, $v_x=x^L-x^R$ and $\iota_{v_x}\theta^L=x-\mathrm{Ad}_{g^{-1}}x$ (a $\mathfrak g$-valued $0$-form), 
so the contraction used above is pointwise well-defined.
Using \(v_{x}={x}^L-{x}^R\) in the last equality, we have 

\begin{align*}
\iota_{v_{x}}\eta &=-\frac{1}{2}\iota_{v_{x}}\theta^{L}\cdot d\theta^{L} \\
&=-\frac{1}{2} (1-\mathrm{Ad}_{g^{-1}})x\cdot d\theta^{L} \\
&=-\frac{1}{2} d\theta^{L}\cdot x +\frac{1}{2} \mathrm{Ad}_{g}d\theta^{L}\cdot x
\end{align*}

From \(\mathrm{Ad}_{g}d(g^{-1}dg)=-gg^{-1}dgg^{-1}dgg^{-1}=-d(dg\,g^{-1})\), we continue and find,

\begin{align}\label{5.8}
    \iota_{v_{x}}\eta =-\frac{1}{2} d\theta^{L}\cdot x -\frac{1}{2} d\theta^{R}\cdot x =-d\left(\frac{\theta^{L}+\theta^{R}}{2}\right)\cdot x
\end{align}

Hence equation \eqref{5.6} is satisfied. This proves that the action is Hamiltonian.
\end{proof}

\begin{thm}\label{proposition: relative moment map}
Let \(L_\infty(M, \Phi, \omega)\) be the Lie 2-algebra associated with the quasi-Hamiltonian \(G\)-space \( (M, \Phi, \omega) \), as defined in Theorem \ref{theorem:semi strict Lie2}. 

Let 
  \begin{align*}
  (u, v):\mathfrak{g} &\to  \mathfrak{X}(\Phi)\\ 
   x & \mapsto  (u_x, v_{x})
  \end{align*}
be the infinitesimal action of \( \mathfrak{g} \) on the space of relative vector fields \( \mathfrak{X}(\Phi) \). 

Then this action lifts to an \(L_\infty\)-morphism 
\[
(f) = (f_1, f_2): \mathfrak{g} \to L_\infty(M, \Phi, \omega)
\]
such that the diagram
\begin{equation*}
\label{the_lift}
\xymatrix{
    &&  L_\infty (M, \Phi, \omega) \ar[d]^{\pi} \\
     \mathfrak{g} \ar @{-->}[urr]^{(f)} \ar[rr]^{(u_{-}, v_{-})} && \mathfrak{X}_{\mathrm{Ham}}(\Phi)
}
\end{equation*}
commutes.

The components of the \(L_\infty\)-morphism \( (f_1, f_2) \) are given by:
\begin{align*}
    f_1: \mathfrak{g} &\to L_\infty(M, \Phi, \omega), \\
    x &\mapsto \left( 0, -\left(\frac{\theta^{L}+\theta^{R}}{2}\right)\cdot x \right),
\end{align*}
and
\begin{align*}
    f_2: \mathfrak{g} \otimes \mathfrak{g} &\to L_\infty(M, \Phi, \omega), \\
    x \otimes y &\mapsto \left( 0, \frac{1}{2} \iota_{v_x} \left( \left( \theta^L + \theta^R \right)\cdot y \right) \right).
\end{align*}

\end{thm}

\begin{proof}
By Lemma \ref{lemma homotopy relative}, we have 
\begin{align*}
    \mathrm{d}f_1(x) &= -\iota_{(u_x, v_x)}(\omega, \eta)
\end{align*}

We prove that the map 
\begin{align*}
    f_2: \mathfrak{g} \otimes \mathfrak{g} &\to L_\infty(M, \Phi, \omega) \\
    x \otimes y &\mapsto \left( 0, \frac{1}{2} \left( \iota_{v_x} \left( \theta^L + \theta^R \right) \cdot y \right)\right)
\end{align*}
is skew-symmetric, i.e., \( f_2(x \otimes y) = -f_2(y \otimes x) \).

We have
\begin{align*}
    f_2(x \otimes y) &= \left( 0, \frac{1}{2} \left( \iota_{v_x} \left( \theta^L + \theta^R \cdot y \right) \right)\right) \\
    &= \left( 0, \frac{1}{2} \left( (\theta^L + \theta^R)(v_x) \cdot y \right) \right)
\end{align*}

But
\begin{align*}
    \theta^L(v_x)&= \theta^L(x^{L}-x^{R})=\theta^L(x^{L})-\theta^L(x^{R})=x-\theta^L(x^{R})=x-\mathrm{Ad}_{g^{-1}}(x)
\end{align*}
and
\begin{align*}
    \theta^R(v_x)&= \theta^R(x^{L}-x^{R})=\theta^R(x^{L})-\theta^R(x^{R})=\theta^R(x^{L})-x=\mathrm{Ad}_{g}(x)-x
\end{align*}
so that
\begin{align}\label{equAdj}
    (\theta^L + \theta^R)(v_x)&= \mathrm{Ad}_{g}(x)-\mathrm{Ad}_{g^{-1}}(x).
\end{align}

It follows that 

\begin{align*}
    f_2(x \otimes y) &=  \left( 0, \frac{1}{2} \left( \mathrm{Ad}_{g}(x) - \mathrm{Ad}_{g^{-1}}\right)(x) \cdot y \right)\\
       &= - \left( 0, \frac{1}{2} \left( x \cdot (\mathrm{Ad}_{g} - \mathrm{Ad}_{g^{-1}})(y) \right) \right)\quad \text{by Ad-invariance of the inner product}\\
    &= -f_2(y \otimes x).
\end{align*}

This proves the skewsymmetry of the map \(f_2\).

It is remaining only to show that equations \eqref{eq:lie2-cond1} and \eqref{eq:lie2-cond2} are satisfied. For equation \eqref{eq:lie2-cond1}, we have:
\begin{align}
f_1([x, y]) &= \left(0, -\frac{1}{2}\left( \theta^L + \theta^R \right)\cdot [x, y] \right)
\end{align}
Using the identity
\begin{align}
    \frac{1}{2}\left( \theta^L + \theta^R  \right)\cdot [x, y] + \iota_{v_x}\iota_{v_y}\eta = \frac{1}{2}\mathrm{d}\iota_{v_x}\left( \theta^L + \theta^R  \right)\cdot y,
\end{align}
we obtain
\begin{align}
f_1([x, y]) &= \left(0, \iota_{v_x}\iota_{v_y}\eta - \frac{1}{2}\mathrm{d}\iota_{v_x}\left( \theta^L + \theta^R\right) \cdot y \right).
\end{align}
We have:
\begin{align*}
l_1 f_2(x, y) &= l_1 \left( 0, \frac{1}{2} \iota_{v_x} \left( \theta^L + \theta^R\right)  \cdot y \right)\\
&= \mathrm{d} \left( 0, \frac{1}{2} \iota_{v_x} \left( \theta^L + \theta^R  \right) \cdot y\right) \\
&= \left( \Phi^* \left( \frac{1}{2} \iota_{v_x} \left( \theta^L + \theta^R  \right)\cdot y \right), -\mathrm{d} \left( \frac{1}{2} \iota_{v_x} \left( \theta^L + \theta^R  \right)\cdot y \right) \right).
\end{align*}

By Lemma \ref{lemma homotopy relative}, \(f_1(x)\) and \(f_1(y)\) are Hamiltonian (so they are both of degree 0), with Hamiltonian vector fields \((u_x, v_x)\) and \((u_y, v_y)\) respectively. It follows that
\begin{align}\label{6ham}
    l_2\left(f_1(x), f_1(y)\right) &= \iota_{(u_y, v_y)}\iota_{(u_x, v_x)}(\omega, \eta)
\end{align}
Hence,
\begin{align*} 
  & l_1f_2(x, y) + l_2\left(f_1(x), f_1(y)\right) = \left( \Phi^* \left( \frac{1}{2} \iota_{v_x} \left( \theta^L + \theta^R  \right) \cdot y \right), -\mathrm{d} \left( \frac{1}{2} \iota_{v_x} \left( \theta^L + \theta^R  \right) \cdot y \right) \right) \\
   & \hspace{5cm} + \iota_{(u_y, v_y)}\iota_{(u_x, v_x)}(\omega, \eta) \\
   & = \left( \frac{1}{2}\Phi^* \left( \iota_{v_x} \left( \theta^L + \theta^R  \right)\cdot y \right) + \iota_{u_y}\iota_{u_x}\omega, -\frac{1}{2}\mathrm{d} \iota_{v_x} \left( \theta^L + \theta^R  \right)\cdot y + \iota_{v_y}\iota_{v_x}\eta \right) \\
   & = \left(0, \iota_{v_x}\iota_{v_y}\eta - \frac{1}{2}\mathrm{d} \iota_{v_x} \left( \theta^L + \theta^R  \right)\cdot y\right) \\
   & = f_1([x, y])
\end{align*}

Hence, equation \eqref{eq:lie2-cond1} holds.

The identity 
\begin{align*}
f_2([x_1, x_2], x_3)-f_2([x_1, x_3], x_2)+f_2([x_2, x_3], x_1) = l_3\left(f_1(x_1), f_1(x_2), f_1(x_2)\right)  
\end{align*}
follows from Lemma \ref{triple contraction of cartan form}.
This completes the proof.
\end{proof}